\documentclass[a4paper,11pt,twoside]{article}
\usepackage[left=2.5cm,right=2cm,top=2cm,bottom=2cm]{geometry}

\usepackage{amsmath,amsfonts,amssymb,amsthm,array}
\usepackage{amsmath}
\usepackage{algorithm}
\usepackage{caption}
\usepackage{subcaption}
\usepackage{algpseudocode}
\usepackage{bm}
\usepackage{color}
\usepackage{graphicx}

\newcommand{\Exp}{\mathbb{E}}

\newcommand{\E}[1]{{\mathbb{E}\left[#1\right] }}    

\newcommand{\Prob}{\mathbb{P}}
\newcommand{\R}{\mathbb{R}}

\newcommand{\bA}{\mathbf{A}}
\newcommand{\bB}{\mathbf{B}}
\newcommand{\bC}{\mathbf{C}}
\newcommand{\bD}{\mathbf{D}}

\newcommand{\bH}{\mathbf{H}}
\newcommand{\bI}{\mathbf{I}}
\newcommand{\bL}{\mathbf{L}}
\newcommand{\bM}{\mathbf{M}}

\newcommand{\bQ}{\mathbf{Q}}

\newcommand{\bS}{\mathbf{S}}

\newcommand{\bW}{\mathbf{W}}
\newcommand{\bZ}{\mathbf{Z}}

\newcommand{\eqdef}{:=}
\newcommand{\ac}{\alpha}

\newcommand{\cD}{{\cal D}}
\newcommand{\cE}{{\cal E}}

\newcommand{\cG}{{\cal G}}

\newcommand{\cL}{{\cal L}}

\newcommand{\cN}{{\cal N}}

\newcommand{\cS}{{\cal S}}

\newcommand{\cV}{{\cal V}}

\newcommand{\range}[1]{\operatorname{Range}\left(#1\right)}
\newcommand{\mA}{{\bf A}}
\newcommand{\mB}{{\bf B}}

\newcommand{\mH}{{\bf H}}

\newcommand{\mL}{{\bf L}}
\newcommand{\mM}{{\bf M}}

\newcommand{\mS}{{\bf S}}

\newcommand{\mW}{{\bf W}}

\newcommand{\mZ}{{\bf Z}}

\theoremstyle{plain}
\newtheorem{thm}{Theorem}[]
\newtheorem{lem}[thm]{Lemma}

\newtheorem{defn}[thm]{Definition}
\newtheorem{rem}{Remark}[]
\newtheorem{cor}{Corollary}[]
\theoremstyle{remark}

\usepackage[colorinlistoftodos,bordercolor=orange,backgroundcolor=orange!20,linecolor=orange,textsize=scriptsize]{todonotes}

\newcommand{\peter}[1]{\todo[inline]{{\textbf{Peter:} \emph{#1}}}}

\usepackage[colorlinks=true,linkcolor=blue]{hyperref}
\providecommand{\keywords}[1]{\textbf{Keywords} #1} 
\providecommand{\ams}[1]{\textbf{Mathematical Subject Classifications } #1}
\begin{document}
\title{Revisiting Randomized Gossip Algorithms: General Framework,  Convergence Rates and Novel Block and Accelerated Protocols}
\author{Nicolas Loizou \\ University of Edinburgh\\ n.loizou@sms.ed.ac.uk 
\and Peter Richt\'{a}rik \\ KAUST ,  MIPT \\peter.richtarik@kaust.edu.sa}
\date{May 20, 2019\footnote{Part of the results was presented in \cite{LoizouRichtarik,loizou2018accelerated,loizou2018provably}}}
\maketitle

\begin{abstract}
In this work we present a new framework for the analysis and design of randomized gossip algorithms for solving the average consensus problem. We show how classical randomized iterative methods for solving linear systems can be interpreted as gossip algorithms when applied to special systems encoding the underlying network and explain in detail their decentralized nature. 
Our general framework recovers a comprehensive array of well-known gossip algorithms as special cases, including the pairwise randomized gossip algorithm and path averaging gossip, and allows for the development of provably faster variants. The flexibility of the new approach enables the design of a number of new specific gossip methods. 
For instance, we propose and analyze novel \textit{block} and the first provably  \textit{accelerated} randomized gossip protocols, and \textit{dual} randomized gossip algorithms.

From a numerical analysis viewpoint, our work is the first that explores in depth the decentralized nature of randomized iterative methods for linear systems and proposes them as methods for solving the average consensus problem. 

We evaluate the performance of the proposed gossip protocols by performing extensive experimental testing on typical wireless network topologies.
\end{abstract}

\noindent \keywords{ randomized gossip algorithms $\cdot$ average consensus $\cdot$ weighted average consensus $\cdot$ stochastic methods $\cdot$ linear systems $\cdot$ randomized Kaczmarz $\cdot$  randomized block Kaczmarz $\cdot$  randomized coordinate descent $\cdot$ heavy ball momentum $\cdot$ Nesterov's acceleration $\cdot$ duality $\cdot$ convex optimization $\cdot$ wireless sensor networks} \\

\noindent \ams{93A14 $\cdot$ 68W15 $ \cdot$ 68Q25 $\cdot$ 68W20 $\cdot$ 68W40 $\cdot$ 65Y20 $\cdot$ 90C15 $\cdot$ 90C20 $\cdot$ 90C25 $\cdot$ 15A06 $\cdot$ 15B52 $\cdot$ 65F10 }

\newpage

{\footnotesize
\tableofcontents
}

\newpage

\section{Introduction}
\label{sec:ACP}
Average consensus is a fundamental problem in distributed computing and multi-agent systems. It comes up in many real world applications such as coordination of autonomous agents, estimation, rumour spreading in social networks, PageRank and distributed data fusion on ad-hoc networks and decentralized optimization. Due to its great importance there is much classical \cite{tsitsiklis1986distributed,degroot1974reaching} and recent   \cite{xiao2005scheme, xiao2004fast, boyd2006randomized} work on the design of efficient algorithms/protocols for solving it.

In the average consensus (AC) problem we are given an undirected connected network $\cG=(\cV,\cE)$ with node set $\cV=\{1,2,\dots,n\}$ and edges $\cE$. Each node $i \in \cV$ ``knows'' a private value $c_i \in \R$. The goal of AC is for every node to compute the average of these private values, $\bar{c}\eqdef\frac{1}{n}\sum_i c_i$, in a distributed fashion. That is, the exchange of information can only occur between connected nodes (neighbors). 

One of the most attractive classes of protocols for solving the average consensus problem are gossip algorithms. The development and design of gossip algorithms was studied extensively in the last decade. The seminal 2006 paper of Boyd et al.\ \cite{boyd2006randomized} motivated a fury of subsequent research and gossip algorithms now appear in many applications, including distributed data fusion in sensor networks \cite{xiao2005scheme}, load balancing \cite{cybenko1989dynamic} and clock synchronization \cite{freris2012fast}.  For a survey of selected relevant work prior to 2010, we refer the reader to the work of  Dimakis et al.\ \cite{dimakis2010gossip}. For more recent results on randomized gossip algorithms we suggest \cite{zouzias2015randomized, liu2013analysis,olshevsky2014linear,liu2018privacy, nedic2018network, aybat2017decentralized}. See also  \cite{dimakis2008geographic, aysal2009broadcast, olshevsky2009convergence,hanzely2017privacy, hanzely2019privacy}. 

\subsection{Main contributions} 
In this work, we connect two areas of research which until now have remained remarkably disjoint in the literature: randomized iterative (projection) methods for solving linear systems and randomized gossip protocols for solving the average consensus. This connection enables us to make contributions by borrowing from each body of literature to the other and using it we propose a new framework for the design and analysis of novel efficient randomized gossip protocols.

The main contributions of our work include:

\begin{itemize}
\item \textbf{RandNLA.} We show how classical randomized iterative methods for solving linear systems can be interpreted as gossip algorithms when applied to special systems encoding the underlying network and explain in detail their decentralized nature.
Through our general framework we recover a comprehensive array of well-known gossip protocols as special cases. In addition our approach allows for the development of novel block and dual variants of all of these methods. From a randomized numerical linear algebra (RandNLA) viewpoint our work is the first that explores in depth, the decentralized nature of randomized iterative methods for solving linear systems and proposes them as efficient methods for solving the average consensus problem (and its weighted variant). 

\item \textbf{Weighted AC.}  The methods presented in this work solve the more general \textit{weighted} average consensus (Weighted AC) problem (Section~\ref{weightedAC}) popular in the area of distributed cooperative spectrum sensing networks. The proposed protocols are the first randomized gossip algorithms that directly solve this problem with finite-time convergence rate analysis. In particular, we prove linear convergence of the proposed protocols and explain how we can obtain further acceleration using momentum. To the best of our knowledge, the existing decentralized protocols that solve the weighted average consensus problem show convergence but without convergence analysis. 

\item \textbf{Acceleration.} We present novel and provably \textit{accelerated} randomized gossip protocols. In each step, of the proposed algorithms, all nodes of the network update their values using their own information but only a subset of them exchange messages. The protocols are inspired by the recently proposed accelerated variants of randomized Kaczmarz-type methods and use momentum terms on top of the sketch and project update rule (gossip communication) to obtain better theoretical and practical performance.
To the best of our knowledge, our accelerated protocols are the first randomized gossip algorithms that converge to a consensus with a provably accelerated linear rate without making any further assumptions on the structure of the network. Achieving an accelerated linear rate in this setting using randomized gossip protocols was an open problem.

\item \textbf{Duality.} We reveal a hidden duality of randomized gossip algorithms, with the dual iterative process maintaining variables attached to the edges of the network. We show how the randomized coordinate descent and randomized Newton methods work as edge-based dual randomized gossip algorithms.

\item \textbf{Experiments.} We corroborate our theoretical results with extensive experimental testing on typical wireless network topologies.
We numerically verify the linear convergence of the our protocols for solving the weighted AC problem. We explain the benefit of using block variants in the gossip protocols where more than two nodes update their values in each iteration. We explore the performance of the proposed provably accelerated gossip protocols and show that they significantly outperform the standard pairwise gossip algorithm and existing fast pairwise gossip protocols with momentum. An experiment showing the importance of over-relaxation in the gossip setting is also presented.
\end{itemize}

This paper contains a synthesis and a unified presentation of the randomized gossip protocols proposed in Loizou and Richt\'{a}rik \cite{LoizouRichtarik}, Loizou and Richt\'{a}rik \cite{loizou2018accelerated} and Loizou et al. \cite{loizou2018provably}. In \cite{LoizouRichtarik}, building upon the results from \cite{gower2015stochastic}, a connection between the area of randomized iterative methods for linear systems and gossip algorithms was established and block gossip algorithm were developed. Then in \cite{loizou2018accelerated} and  \cite{loizou2018provably} faster and provably accelerated gossip algorithms were proposed using the heavy ball momentum and Nesterov's acceleration technique, respectively. This paper expands upon these results and presents proofs for theorems that are referenced in the above papers. We also conduct several new experiments.

We believe that this work could potentially open up new avenues of research in the area of decentralized gossip protocols. 

\subsection{Structure of the paper}
This work is organized as follows. Section~\ref{background} introduces the necessary background on basic randomized iterative methods for linear systems that will be used for the development of randomized gossip protocols.  Related work on the literature of linear system solvers, randomized gossip algorithms for averaging and gossip algorithms for consensus optimization is presented.   In Section~\ref{skecthsection} the more general weighted average consensus problem is described and the connections between the two areas of research (randomized projection methods for linear systems and gossip algorithms) is established. In particular we explain how methods for solving linear systems can be interpreted as gossip algorithms when applied to special systems encoding the underlying network and elaborate in detail their distributed nature. Novel block gossip variants are also presented. In Section~\ref{AccelerateGossip} we describe and analyze fast and provably accelerated randomized gossip algorithms. In each step of these protocols all nodes of the network update their values but only a subset of them exchange their private values. Section~\ref{DualBlock} describes dual randomized gossip algorithms that operate with values that are associated to the edges of the network and Section~\ref{FurtherConnections} highlights further connections between methods for solving linear systems and gossip algorithms.
Numerical evaluation of the new gossip protocols is presented in Section~\ref{experiments}. Finally, concluding remarks are given in Section~\ref{conclusion}.

\subsection{Notation}
For convenience, a table of the most frequently used notation is included in the Appendix \ref{NotationTable}. 
In particular, with boldface upper-case letters denote matrices; $\bI$ is the identity matrix. By $\|\cdot \|$ and $\|\cdot \|_F$ we denote the Euclidean norm and  the Frobenius norm, respectively.  For a positive integer number $n$, we write $[n]\eqdef \{1,2, \dots ,n\}$.  By $\cL$ we denote the solution set of the linear system $\bA x=b$, where $\bA \in \R^{m\times n}$ and $b\in \R^m$.

We shall often refer to specific matrix expressions involving several matrices. In order to keep these expressions brief throughout the paper it will be useful to define the following two matrices:
\begin{equation}
\label{ZETA}
\mH \eqdef  \mS (\mS^\top \mA \bB^{-1} \mA^\top \mS)^\dagger \mS^\top \quad \text{and} \quad \mZ \eqdef  \mA^\top \bH \mA,
\end{equation}
depending on a  random matrix $\mS \in \R^{m \times q}$ drawn from a given distribution $\cD$ and on an $n\times n$ positive definite matrix $\mB$ which defines the geometry of the space. In particular we define $\bB-$inner product in $\R^n$ via $\langle x,z \rangle_\mB \eqdef \langle \mB x, z\rangle= x^\top \bB z$ and an induced $\bB-$norm, $\|x\|_\mB\eqdef (x^\top \mB x)^{1/2}$.
By $\bA_{i:}$ and $\bA_{:j}$ we denote the $i^{th}$ row and the $j^{th}$ column of matrix $\bA$, respectively. By $\dagger$ we denote the Moore-Penrose pseudoinverse. 
 
The complexity of all gossip protocols presented in this paper is described by the spectrum of matrix 
\begin{equation}
\label{MatrixW}
\bW=\bB^{-1/2}\mA^\top \Exp[\bH] \mA\bB^{-1/2}\overset{\eqref{ZETA}}{=}\bB^{-1/2}\Exp[\bZ]\bB^{-1/2},
\end{equation}
where the expectation is taken over $\bS\sim \cD$.  
With $\lambda_{\min}^+$ and $\lambda_{\max}$ we indicate the smallest nonzero and the largest eigenvalue of matrix $\bW$, respectively. 

Vector $x^k = (x^k_1,\dots,x^k_n) \in \R^n$ represents the vector with the private values of the $n$ nodes of the network at the $k^{th}$ iteration while with $x_i^{k}$ we denote the value of node $i \in [n]$ at the $k^{th}$ iteration. $\cN_i \subseteq \cV$ denotes the set of nodes that are neighbors of node $i \in \cV$.  By $\ac(\cG)$ we denote the algebraic connectivity of graph $\cG$.
 
Throughout the paper, $x^*$ is the projection of $x^0$ onto $\cL$ in the $\bB$-norm. We write $x^*=\Pi_{\cL,\bB}(x^0)$.
An explicit formula for the projection of $x$ onto set $\cL$ is given by
\begin{equation*}
\Pi_{\cL,\bB}(x)\eqdef \arg\min_{x' \in \cL} \|x'-x\|_{\bB} =x-\bB^{-1}\bA^\top (\bA \bB^{-1} \bA ^ \top )^\dagger (\bA x-b).
\end{equation*}

Finally, with $\bQ \in \R^{|\cE| \times n}$ we define the incidence matrix and with $\bL \in \R^{n\times n} $ the Laplacian matrix of the network. Note that it holds that $\bL=\bQ^\top \bQ$. Further, with $\bD$ we denote the degree matrix of the graph. That is, $\bD=\text{\textbf{Diag}}(d_1,d_2,\dots, d_n) \in \R^{n \times n}$ where $d_i$ is the degree of node $i \in \cV$. 

\section{Background - Technical Preliminaries}
\label{background}

Solving linear systems is a central problem in numerical linear algebra and plays an important role in computer science, control theory, scientific computing, optimization, computer vision, machine learning, and many other fields. With the advent of the age of big data, practitioners are looking for ways to solve linear systems of unprecedented sizes. In this large scale setting, randomized iterative methods are preferred mainly because of their cheap per iteration cost and because they can easily scale to extreme dimensions.

\subsection{Randomized iterative methods for linear systems}
Kaczmarz-type methods are very popular for solving linear systems $\bA x =b$ with many equations. The (deterministic) Kaczmarz method for solving consistent linear systems was originally introduced by Kaczmarz in 1937 \cite{kaczmarz1937angenaherte}. Despite the fact that a large volume of papers  was written  on the topic, the first provably linearly convergent variant of the Kaczmarz method---the randomized Kaczmarz Method (RK)---was developed more than 70 years later, by Strohmer and Vershynin \cite{RK}. This result sparked renewed interest in design of randomized methods for solving linear systems \cite{needell2010randomized, RBK, eldar2011acceleration, MaConvergence15, zouzias2013randomized, l2015randomized, schopfer2016linear, liu2016accelerated}. More recently, Gower and Richt\'{a}rik \cite{gower2015randomized} provide a unified analysis for several randomized iterative methods for solving linear systems using a sketch-and-project framework. We adopt this framework in this paper. 

In particular, the analysis in \cite{gower2015randomized} was done under the assumption that matrix $\bA$ has full column rank. This assumption was lifted in \cite{gower2015stochastic}, and a duality theory for the method developed. Later, in \cite{ASDA}, it was shown that the sketch and project method of \cite{gower2015randomized} can be interpreted as stochastic gradient descent applied to a suitable stochastic optimization problem and relaxed variants of the proposed methods have been presented.

The sketch-and-project algorithm \cite{gower2015randomized, ASDA} for solving a consistent linear system $\bA x= b$ has the form
\begin{eqnarray}
\label{sketchproject}
x^{k+1} &=& x^k -\omega \bB^{-1}\bA^\top \bS_k (\bS_k^\top \bA \bB^{-1} \bA^\top \bS_k)^\dagger \bS_k^\top (\bA x^k-b) \notag\\ 
&\overset{\eqref{ZETA}}{=}& x^k - \omega \bB^{-1} \bA^\top \bH_k (\bA x^k-b),
\end{eqnarray}
where in each iteration, matrix $\mS_k \in \R^{m \times q}$ is sampled afresh from an arbitrary distribution $\cD$.\footnote{We stress that there is no restriction on the number of columns of matrix $\mS_k$ ($q$ can be varied)} In \cite{gower2015randomized} it was shown that many popular randomized algorithms for solving linear systems, including RK method and randomized coordinate descent method (a.k.a Gauss-Seidel method) can be cast as special cases of the above update by choosing an appropriate combination of the distribution $\cD$ and the positive definite matrix $\bB$. 

In the special case that $\omega=1$ (no relaxation), the update rule of equation \eqref{sketchproject} can be equivalently written as follows:
\begin{equation}
\begin{aligned}
& \quad \quad \quad x^{k+1}
& & \eqdef \text{argmin}_{x \in \R^n} \|x-x^k\|_{\bB}^2\\
& \text{subject to}
& & \bS_k^\top \bA x=\bS_k^\top b \;.
\end{aligned}
\end{equation}
This equivalent presentation of the method justifies the name \emph{Sketch and Project}. In particular, the method is a two step procedure: (i) Draw random matrix $\mS_k \in \R^{m \times q}$ from distribution $\cD$ and formulate a \textit{sketched} system $\bS_k^\top \bA x=\bS_k^\top b$, (ii) \textit{Project} the last iterate $x^k$ into the solution set of the sketched system.

A formal presentation of the Sketch and Project method is shown in Algorithm~\ref{FullSkecth}.
\begin{algorithm}[H]
	\caption{Sketch and Project Method \cite{ASDA}}
	\label{FullSkecth}
	\small \small
	\begin{algorithmic}[1]
		\State {\bf Parameters:} Distribution $\mathcal{D}$ from which method samples matrices; stepsize/relaxation parameter $\omega \in \R$; momentum parameter $\beta$.
		\State {\bf Initialize:} $x^0,x^1 \in \R^n$
		\For{$k=0,1,2,\dots$} 
		\State Draw a fresh $\bS_k \sim \cD$
		\State   Set 
$x^{k+1}=x^k -\omega \bB^{-1}\bA^\top \bS_k (\bS_k^\top \bA \bB^{-1} \bA^\top \bS_k)^\dagger \bS_k^\top (\bA x^k-b)$
		\EndFor
		\State {\bf Output:} The last iterate $x^k$
	\end{algorithmic}
\end{algorithm}

In this work, we are mostly interested in two special cases of the sketch and project framework--- the randomized Kaczmarz (RK) method and its block variant, the randomized block Kaczmarz (RBK) method. In addition, in the following sections we present novel scaled and accelerated variants of these two selected cases and interpret their gossip nature. In particular, we focus on explaining how these methods can solve the average consensus problem and its more general version, the weighted average consensus (subsection~\ref{weightedAC}).

Let  $e_i \in \R^m$  be the $i^{\text{th}}$ unit coordinate vector in $ \R^m$ and let $\bI_{:C}$ be column submatrix of the $m \times m$ identity matrix with columns indexed by  $C\subseteq [m]$.  Then RK and RBK methods can be obtained as special cases of the  update rule \eqref{sketchproject} as follows:
\begin{itemize}
\item RK: Let $\bB=\bI$ and $\bS_k=e_i$, where $i \in [m]$ is chosen independently at each iteration, with probability $p_i>0$. In this setup the update rule \eqref{sketchproject} simplifies to  
\begin{equation}
\label{RK}
x^{k+1}=x^k - \omega \frac{\bA_{i :} x^k -b_{i}}{\|\bA_{i :}\|^2} \bA_{i :}^ \top  .
\end{equation}
\item RBK: Let $\bB=\bI$ and $\bS=\bI_{:C}$, where set $C\subseteq [m]$ is chosen independently at each iteration, with probability $p_C\geq 0$. In this setup the update rule \eqref{sketchproject} simplifies to 
\begin{equation}
\label{RBK}
x^{k+1}=x^k - \omega \bA_{C:}^\top (\bA_{C:}\bA_{C:}^\top)^\dagger (\bA_{C:}x^k-b_C).
\end{equation}
\end{itemize}

In several papers \cite{gower2015stochastic, ASDA, loizou2017momentum,loizou2019Inexact}, it was shown that even in the case of consistent linear systems with {\em multiple} solutions, Kaczmarz-type methods converge linearly to one particular solution: the projection (on $\bB$-norm) of the initial iterate $x^0$ onto the solution set of the linear system. This naturally leads to the formulation of the {\em best approximation problem}:
\begin{equation}
\label{best approximation}
\min_{x = (x_1,\dots, x_n) \in \R^n} \frac{1}{2} \|x-x^0\|_{\bB}^2 
\quad \text{subject to}  \quad \bA x = b.
\end{equation}
where $\bA\in \R^{m\times n}$.  In the rest of this manuscript, $x^*$ denotes the solution of \eqref{best approximation} and we write $x^*=\Pi_{\cL,\bB}(x^0)$.

\paragraph{Exactness.} An important assumption that is required for the convergence analysis of the randomized iterative methods under study is \textit{exactness}. That is:
\begin{equation}
\label{exactnessCondition}
{\rm Null}(\Exp_{\bS \sim \cD} [\bZ])={\rm Null}(\bA).
\end{equation}
The exactness property is of key importance in the sketch and project framework, and should be seen as an assumption on the distribution $\cD$ and not on matrix $\bA$. 

Clearly, an assumption on the distribution $\cD$ of the random matrices $\bS$ should be required for the convergence of Algorithm~\ref{FullSkecth}. For an instance, if $\cD$ is such that, $\bS=e_1$ with probability 1, where $e_1 \in \R^m$  be the $1^{\text{st}}$ unit coordinate vector in $ \R^m$, then the algorithm will select the same row of matrix $\bA$ in each step. For this choice of distribution it is clear that the algorithm will not converge to a solution of the linear system. The exactness assumption guarantees that this will not happen.

For necessary and sufficient conditions for exactness, we refer the reader to \cite{ASDA}. Here it suffices to remark that the exactness condition is very weak, allowing $\cD$ to be virtually any reasonable distribution of random matrices. For instance, a sufficient condition for exactness is for the matrix $\E{\mH}$ to be positive definite \cite{gower2015stochastic}. This is indeed a weak condition since it is easy to see that this matrix is symmetric and positive semidefinite without the need to invoke any assumptions; simply by design.  

A much stronger condition than exactness is $\E{\mZ}\succ 0$ which has been used for the analysis of the sketch and project method in \cite{gower2015randomized}. In this case, the matrix $\bA \in \R^{m \times n}$ of the linear system requires to have full column rank and as a result the consistent linear system has a unique solution.

The convergence performance of the Sketch and Project method (Algorithm~\ref{FullSkecth}) under the exactness assumption for solving the best approximation problem is described by the following theorem.

\begin{thm}[\cite{ASDA}]
\label{ConvergenceSketchProject}
Let assume exactness and let $\{x^k\}_{k=0}^\infty$ be the iterates produced by the sketch and project method (Algorithm~\ref{FullSkecth}) with step-size $\omega \in (0,2)$. Set, $x^*=\Pi_{\cL,\bB}(x^0)$. Then,
 \begin{equation}
 \label{ConvergenceBasic}
 \Exp[\|x^k-x^*\|_{\bB}^2]\leq \rho^k \|x^0-x^*\|_{\bB}^2,
 \end{equation}
where
\begin{equation}
\label{RateRho}
\rho \eqdef 1 - \omega (2-\omega) \lambda_{\min}^+ \in [0,1].
\end{equation}
Recall that $\lambda_{\min}^+$ denotes the minimum nonzero eigenvalue of matrix $\bW\eqdef \Exp[\bB^{-1/2}\bA^\top \bH \bA\bB^{-1/2}]$. 
\end{thm}

In other words, using standard arguments, from Theorem~\ref{ConvergenceSketchProject} we observe that for a given $\epsilon \in (0,1)$ we have that:
$$k \geq \frac{1}{1-\rho} \log \left(\frac{1}{\epsilon} \right) \quad \Rightarrow \quad \Exp[\|x^k-x^*\|_{\bB}^2]\leq \epsilon \|x^0-x^*\|_{\bB}^2.$$
We say that the iteration complexity of sketch and project method is, 
$$O\left(\frac{1}{1-\rho} \log \left(\frac{1}{\epsilon} \right)\right).$$

\subsection{Other related work}

\paragraph{On Sketch and Project Methods.} 
Variants of the sketch-and-project methods have been recently proposed for solving several other problems. \cite{gower2016randomized} and \cite{gower2016linearly} use similar ideas for the development of linearly convergent randomized iterative methods for computing/estimating the inverse and pseudoinverse of a large matrix, respectively. A limited memory variant of the stochastic block BFGS method  for solving the empirical risk minimization problem arising in machine learning  was proposed by \ \cite{gower2016stochastic}. Tu et al.\ \cite{tu2017breaking} utilize the sketch-and-project framework to show that breaking block locality can accelerate block Gauss-Seidel methods. In addition, they develop an accelerated variant of the method for a specific distribution $\cD$.  A sketch and project method with the heavy ball momentum was studied in \cite{loizou2017momentum, loizou2017linearly} and an accelerated (in the sense of Nesterov) variant of the method proposed in \cite{gower2018accelerated} for the more general Euclidean setting and applied to matrix inversion and quasi-Newton updates. Inexact variants of Algorithm~\ref{FullSkecth} have been proposed in \cite{loizou2019Inexact}. As we have already mentioned, in \cite{ASDA},  through the development of stochastic reformulations, a stochastic gradient descent interpretation of the sketch and project method has been proposed.  Recently, using a different stochastic reformulation, \cite{gower2019sgd} performed a tight convergence analysis of stochastic gradient descent in a more general convex setting. The analysis proposed in  \cite{gower2019sgd} recovers the linear convergence rate of sketch and project method (Theorem~\ref{ConvergenceSketchProject}) as special case. 

\paragraph{Gossip algorithms for average consensus} The problem of average consensus has been extensively studied in the automatic control and signal processing literature for the past two decades~\cite{dimakis2010gossip}, and was first introduced for decentralized processing in the seminal work~\cite{tsitsiklis1986distributed}. A clear connection between the rate of convergence and spectral characteristics of the underlying network topology over which message passing occurs was first established in~\cite{boyd2006randomized} for pairwise randomized gossip algorithms. 

Motivated by network topologies with salient properties of wireless networks (e.g., nodes can communicate directly only with other nearby nodes), several methods were proposed to accelerate the convergence of gossip algorithms. For instance, \cite{benezit2010order} proposed averaging among a set of nodes forming a path in the network (this protocol can be seen as special case of our block variants in Section~\ref{BlockGossip}). Broadcast gossip algorithms have also been analyzed \cite{aysal2009broadcast} where the nodes communicate with more than one of their neighbors by broadcasting their values.

While the gossip algorithms studied in~\cite{boyd2006randomized,benezit2010order,aysal2009broadcast} are all first-order (the update of $x^{k+1}$ only depends on $x^k$), a faster randomized pairwise gossip protocol was proposed in~\cite{cao2006accelerated} which suggested to incorporate additional memory to accelerate convergence. The first analysis of this protocol was later proposed in~\cite{liu2013analysis} under strong conditions. It is worth to mention that in the setting of deterministic gossip algorithms theoretical guarantees for accelerated convergence were obtained in \cite{oreshkin2010optimization,kokiopoulou2009polynomial}.
In Section~\ref{AccelerateGossip} we propose fast and provably accelerated randomized gossip algorithms with memory and compare them in more detail with the fast randomized algorithm proposed in~\cite{cao2006accelerated, liu2013analysis}.

\paragraph{Gossip algorithms for multiagent consensus optimization.} 
In the past decade there has been substantial interest in consensus-based mulitiagent optimization methods that use gossip updates in their update rule \cite{nedic2018network, yuan2016convergence, shi2015extra}. In multiagent consensus optimization setting , $n$ agents or nodes, cooperate to solve an optimization problem. In particular, a local objective function $f_i:\R^d \rightarrow \R$ is associated with each node $i \in [n]$ and the goal is for all nodes to solve the optimization problem \begin{equation}
\label{optMulti}
\min_{x \in \R^d} \frac{1}{n}\sum_{i=1}^n f_i(x)
\end{equation} by communicate only with their neighbors. In this setting gossip algorithms works in two steps by first executing some local computation followed by communication over the network \cite{nedic2018network}. Note that the average consensus problem with $c_i$ as node $i$ initial value can be case as a special case of the optimization problem \eqref{optMulti} when the function values are $f_i(x)=(x- c_i)^2$.

Recently there has been an increasing interest in applying mulitagent optimization  methods to solve convex and non-convex optimization problems arising in machine learning~\cite{tsianos2012communication, lian2018asynchronous, assran2018stochastic, assran2018asynchronous,colin2016gossip, koloskova2019decentralized, hendrikx2018accelerated}. In this setting most consensus-based optimization methods make use of standard, first-order gossip, such as those described in~\cite{boyd2006randomized}, and incorporating momentum into their updates to improve their practical performance. 

\section{Sketch and Project Methods as Gossip Algorithms}
\label{skecthsection}
In this section we show how by carefully choosing the linear system in the constraints of the best approximation problem \eqref{best approximation} and the combination of the parameters of the Sketch and Project method (Algorithm~\ref{FullSkecth}) we can design efficient randomized gossip algorithms. We show that the proposed protocols can actually solve the weighted average consensus problem, a more general version of the average consensus problem described in Section~\ref{sec:ACP}. In particular we focus, on a scaled variant of the RK method \eqref{RK} and on the RBK \eqref{RBK} and understand the convergence rates of these methods in the consensus setting, their distributed nature and how they are connected with existing gossip protocols. 

\subsection{Weighted average consensus}
\label{weightedAC}
In the \emph{weighted average consensus} (Weighted AC) problem we are given an undirected connected network $\cG=(\cV,\cE)$ with node set $\cV=\{1,2,\dots,n\}$ and edges $\cE$. Each node $i \in \cV$ holds a private value $c_i \in \R$ and its weight $w_i$. The goal of this problem is for every node to compute the weighted average of the private values, 
$$\bar{c}\eqdef \frac{\sum_{i=1}^n w_i c_i}{\sum_{i=1}^n w_i},$$
in a distributed fashion. That is, the exchange of information can only occur between connected nodes (neighbors). 

Note that in the special case when the weights of all nodes are the same ($w_i=r$ for all $i \in [n]$) the weighted average consensus is reduced to the standard average consensus problem. However, there are more special cases that could be interesting. For instance the weights can represent the degree of the nodes ($w_i= d_i$) or they can denote a probability vector and satisfy $\sum_i^n w_i=1$ with $w_i>0$. 

It can be easily shown that the weighted average consensus problem can be expressed as optimization problem as follows:
\begin{equation}
\label{weightedCOnsensus}
\min_{x = (x_1,\dots, x_n) \in \R^n} \frac{1}{2} \|x-c\|_{\bB}^2 
\quad \text{subject to}  \quad x_1=x_2=\dots=x_n
\end{equation}
where matrix $\bB=\textbf{Diag}(w_1, w_2,\dots, w_n)$ is a diagonal positive definite matrix (that is $w_i>0$ for all $i \in [n]$) and $c=(c_1,\dots,c_n)^\top$ the vector with the initial values $c_i$ of all nodes $i \in \cV$. The optimal solution of this problem is $x^*_i=\frac{\sum_{i=1}^n w_i c_i}{\sum_{i=1}^n w_i}$ for all $i \in [n]$ which is exactly the solution of the weighted average consensus.

As we have explained, the standard average consensus problem can be cast as a special case of weighted average consensus. However, in the situation when the nodes have access to global information related to the network, such as the size of the network (number of nodes $n=|\cV|$) and the sum of the weights $\sum_{i=1}^n w_i$, then any algorithm that solves the standard average consensus can be used to solve the weighted average consensus problem with the initial private values of the nodes changed from $c_i$ to $\frac{n w_i c_i}{\sum_{i=1}^n w_i}$. 

The weighted AC problem is popular in the area of distributed cooperative spectrum sensing networks \cite{hernandes2018improved, pedroche2014convergence, zhang2015distributed, zhang2011distributed}. In this setting, one of the goals is to develop decentralized protocols for solving the cooperative sensing problem in cognitive radio systems. The weights in this case represent a ratio related to the channel conditions of each node/agent \cite{hernandes2018improved}. The development of methods for solving the weighted AC problem is an active area of research (check \cite{hernandes2018improved} for a recent comparison of existing algorithms). However, to the best of our knowledge, existing analysis for the proposed algorithms focuses on showing convergence and not on providing convergence rates. Our framework allows us to obtain novel randomized gossip algorithms for solving the weighted AC problem. In addition, we provide a tight analysis of their convergence rates. In particular, we show convergence with a linear rate. See Section~\ref{Weightexperiments} for an experiment confirming linear convergence of one of our proposed protocols on typical wireless network topologies.

\subsection{Gossip algorithms through sketch and project framework}
We propose that randomized gossip algorithms should be viewed as special case of the Sketch and Project update to a particular problem of the form \eqref{best approximation}. In particular, we let $c=(c_1,\dots,c_n)$ be the initial values stored at the nodes of $\cG$, and choose $\bA$ and $b$ so that the constraint $\bA x = b$ is equivalent to the requirement that $x_i=x_j$ (the value stored at node $i$ is equal to the value stored at node $j$) for all $(i,j)\in \cE$.

\begin{defn}
\label{defACsystem}
 We say that  $\bA x = b$ is an ``average consensus (AC) system'' when $\bA x = b$ iff $x_i = x_j$ for all $(i,j) \in \cE$.
\end{defn}

It is easy to see that $\bA x = b$ is an AC system precisely when $b=0$ and the nullspace of $\bA$ is $\{t 1_n : t\in \R\}$, where $1_n$ is the vector of all ones in $\R^n$.  Hence, $\bA$ has rank $n-1$. Moreover in the case that $x^0=c$, it is easy to see that for any AC system, the solution of \eqref{best approximation}  necessarily is $x^* = \bar{c} \cdot 1_n$ --- this is why we singled out AC systems. In this sense, {\em any} algorithm for solving \eqref{best approximation} will ``find'' the (weighted) average $\bar{c}$. However, in order to obtain a distributed algorithm we need to make sure that only ``local'' (with respect to $\cG$) exchange of information is allowed.  

It can be shown that many linear systems satisfy the above definition. 

For example, we can choose $b=0$ and $\bA=\bQ \in \R^{|\cE| \times n}$ to be the incidence matrix of  $\cG$. That is, $\bQ \in \R^{|\cE| \times n}$ such that  $\bQ x = 0$ directly encodes the constraints $x_i=x_j$ for $(i,j)\in \cE$. That is, row  $e=(i,j) \in \cE$ of matrix $\bQ$ contains value $1$ in column $i$, value $-1$ in column $j$ (we use an arbitrary but fixed order of nodes defining each edge in order to fix $\bQ$) and zeros elsewhere. 
A different choice is to pick $b=0$ and  $\mA =\mL=\bQ^\top \bQ$, where $\mL$ is the Laplacian matrix of network $\cG$. Depending on what AC system is used, the sketch and project methods can have different interpretations as gossip protocols.  

In this work we mainly focus on the above two AC systems but we highlight that other choices are possible\footnote{Novel gossip algorithms can be proposed by using different AC systems to formulate the average consensus problem. For example one possibility is using the random walk normalized Laplacian $\bL^{rw}=\bD^{-1} \bL$. For the case of degree-regular networks the symmetric normalized Laplacian matrix $\bL^{sym}=\bD^{-1/2} \bL \bD^{-1/2}$ can also being used.}. In Section~\ref{accSubsection} for the provably accelerated gossip protocols we also use a normalized variant ($\|\bA_{i:}\|^2=1$) of the Incidence matrix. 

\subsubsection{Standard form and mass preservation}
Assume that $\bA x = b$ is an AC system. Note that since $b=0$, the general sketch-and-project update rule \eqref{sketchproject} simplifies to:
\begin{equation}
\label{updateSkProj}
x^{k+1}= \left[ \bI- \omega\bA^\top \bH_k\bA \right] x^k=\left[ \bI- \omega  \bZ_k \right] x^k.
\end{equation}
This is the standard form in which randomized gossip algorithms are written. What is new here is that the iteration matrix $\bI- \omega  \bZ_k$ has a specific structure which guarantees convergence to $x^*$ under very weak assumptions (see Theorem~\ref{ConvergenceSketchProject}). Note that if $x^0=c$, i.e., the starting primal iterate is the vector of private values (as should be expected from any gossip algorithm), then the iterates of \eqref{updateSkProj} enjoy a mass preservation property (the proof follows the fact that $\bA 1_n = 0$):
\begin{thm}[Mass preservation]
If $\bA x =b$ is an AC system, then the iterates produced by \eqref{updateSkProj} satisfy:  $\frac{1}{n}\sum_{i=1}^{n}x_i^k=\bar{c}$, for all  $k \geq 0$.
 \end{thm}
 \begin{proof} Let fix $k\geq0$ then,
 $$\frac{1}{n} 1_n^\top x^{k+1}=\frac{1}{n} 1_n^\top ( \bI - \omega\bA^\top \bH_k \bA) x^k=\frac{1}{n} 1_n^\top  \bI x^k - \frac{1}{n} 1_n^\top \omega\bA^\top \bH_k \bA x^k\overset{\bA 1_n = 0}{=}\frac{1}{n} 1_n^\top x^k.$$
 \end{proof}
 
\subsubsection{$\varepsilon$-Averaging time}

Let $z^k\eqdef \|x^k - x^*\|$. The typical measure of convergence speed employed in the randomized gossip literature, called $\varepsilon$-averaging time and here denoted by $T_{ave}(\varepsilon)$, represents the smallest time $k$ for which  $x^{k}$ gets within $\varepsilon z^0$ from $x^*$, with probability greater than $1-\varepsilon$, uniformly over all starting values $x^0=c$. More formally, we define
\[
T_{ave}(\varepsilon)\eqdef \sup_{c\in \R^n} \inf  \left\{k\;:\; \Prob \left(z^k > \varepsilon z^0 \right)\leq\varepsilon \right\}.
\]
This definition differs slightly from the standard one in that we use $z^0$ instead of $\|c\|$.

Inequality \eqref{ConvergenceBasic}, together with Markov inequality, can be used  to give a bound on $K(\varepsilon)$, formalized next:

\begin{thm}\label{thm:complexity_standard} Assume $\bA x=b$ is an AC system. Let $x^0=c$ and $\bB$ be positive definite diagonal matrix. Assume exactness. Then for any $0<\varepsilon < 1$ we have
$$T_{ave}(\epsilon) \leq 3 \frac{\log(1/\varepsilon)}{\log(1/\rho)} \leq 3\frac{\log(1/\epsilon)}{1-\rho},$$
where $\rho$ is defined in \eqref{RateRho}. 
\end{thm}
\begin{proof}
See Appendix~\ref{ProofTave}.
\end{proof}

Note that under the assumptions of the above theorem, $\bW=\bB^{-1/2} \Exp[\bZ] \bB^{-1/2}$ only has a single zero eigenvalue, and hence $\lambda_{\min}^+ (\bW)$ is the second smallest eigenvalue of $\bW$. Thus, $\rho$ is the second largest eigenvalue of $\bI - \bW$. The bound on $K(\varepsilon)$ appearing in Thm~\ref{thm:complexity_standard} is often written with $\rho$ replaced by $\lambda_2(\bI - \bW)$ \cite{boyd2006randomized}.

In the rest of this section we show how two special cases of the sketch and project framework, the randomized Kaczmarz (RK) and its block variant, randomized block Kaczmatz (RBK) work as gossip algorithms for the two AC systems described above. 

\subsection{Randomized Kaczmarz method as gossip algorithm}
As we described before the sketch and project update rule \eqref{sketchproject} has several parameters that should be chosen in advance by the user. These are the stepsize $\omega$ (relaxation parameter), the positive definite matrix $\bB$ and the distribution $\cD$ of the random matrices $\bS$. 

In this section we focus on one particular special case of the sketch and project framework, a scaled/weighted variant of the randomized Kaczmarz method (RK) presented in \eqref{RK},  and we show how this method works as gossip algorithm when applied to special systems encoding the underlying network. In particular, the linear systems that we solve are the two AC systems described in the previous section where the matrix is either the incidence matrix $\bQ$ or the Laplacian matrix $\bL$ of the network.

As we described in \eqref{RK} the standard RK method can be cast as special case of the sketch and project update \eqref{sketchproject} by choosing $\bB=\bI$ and $\bS=e_i$. In this section, we focus on a small modification of this algorithm and we choose the positive definite matrix $\bB$ to be  $\bB=\textbf{Diag}(w_1, w_2,\dots, w_n)$, the diagonal matrix of the weights presented in the weighted average consensus problem. 

\paragraph{Scaled RK:} Let us have a general consistent linear system $\bA x= b$ with $\bA \in \R^{m \times n}$. Let us also choose $\bB=\textbf{Diag}(w_1, w_2,\dots, w_n)$ and $\bS_k=e_i$, where $i\in [m]$ is chosen in each iteration independently, with probability $p_i>0$. In this setup the update rule \eqref{sketchproject} simplifies to  
\begin{equation}
\label{scaledRK}
x^{k+1}=x^k - \omega \frac{e_i^\top (\bA x^k -b) }{e_i^\top \bA \bB^{-1} \bA^\top e_i} \bB^{-1}\bA^ \top e_i  =x^k - \omega \frac{\bA_{i :} x^k -b_{i}}{\|\bB^{-1/2} \bA_{i :}^\top\|^2_{2}} \bB^{-1}\bA_{i :}^ \top.
\end{equation}
This small modification of RK allow us to solve the more general weighted average consensus presented in Section~\ref{weightedAC} (and at the same time the standard average consensus problem if $\bB= r \bI$ where $r \in \R$). To the best of our knowledge, even if this variant is special case of the general Sketch and project update, was never precisely presented before in any setting.

\subsubsection{AC system with incidence matrix $\bQ$}
Let us represent the constraints of problem \eqref{weightedCOnsensus} as linear system with matrix $\bA=\bQ \in \R^{|\cE| \times n}$ be the Incidence matrix of the graph and right had side $b=0$. Lets also assume that the random matrices $\bS \sim \cD$ are unit coordinate vectors in $\R^{m}=\R^{|\cE|}$. 

Let $e=(i,j) \in \cE$ then from the definition of matrix $\bQ$ we have that $\bQ_{e:}^\top=f_i-f_j$ where $f_i, f_j$ are unit coordinate vectors in $\R^n$. In addition, from the definition the diagonal positive definite matrix $\bB$ we have that 
\begin{equation}
\label{ansdl}
\|\bB^{-1/2} \bQ_{e :}^\top\|^2=\|\bB^{-1/2} (f_i-f_j)\|^2=\frac{1}{w_1}+\frac{1}{w_j}.
\end{equation}

Thus in this case the update rule \eqref{scaledRK} simplifies:
\begin{eqnarray}
\label{updateSRKQ}
x^{k+1}&\overset{b=0, \bA=\bQ, \eqref{scaledRK}}{=}& x^k - \omega \frac{\bQ_{e :} x^k}{\|\bB^{-1/2} \bQ_{e :}^\top\|^2} \bB^{-1}\bQ_{e :}^ \top \notag \\
&\overset{\eqref{ansdl}}{=}&x^k - \omega \frac{\bQ_{e :} x^k}{\frac{1}{w_1}+\frac{1}{w_j}} \bB^{-1}\bQ_{e :}^ \top \notag \\
&=&x^k-\frac{\omega (x^k_i-x^k_j)}{\frac{1}{w_i}+\frac{1}{w_j}}  \left(\frac{1}{w_i} f_i- \frac{1}{w_j } f_j \right). 
\end{eqnarray}

From \eqref{updateSRKQ} it can be easily seen that only the values of coordinates $i$ and $j$ update their values. These coordinates correspond to the private values $x^k_i$ and $x^k_j$ of the nodes of the selected edge $e=(i,j)$. In particular the values of $x^k_i $ and $x^k_j$ are updated as follows:

\begin{equation}
\label{ScRKwithQ}
x^{k+1}_i=\left(1-\omega \frac{w_j}{w_j+w_i} \right) x^{k}_i + \omega \frac{ w_j}{w_j+w_i} x^k_j \quad \text{and} \quad x^{k+1}_j= \omega \frac{ w_i}{w_j+w_i} x^{k}_i + \left(1-\omega\frac{ w_i}{w_j+w_i} \right) x^k_j.
\end{equation}

\begin{rem}
In the special case that $\bB=r \bI$ where $r\in \R$ (we solve the standard average consensus problem) the update of the two nodes is simplified to
$$x^{k+1}_i=\left(1-\frac{\omega}{2} \right) x^{k}_i + \frac{\omega }{2} x^k_j \quad \text{and} \quad x^{k+1}_j= \frac{\omega}{2} x^{k}_i + \left(1-\frac{\omega}{2} \right) x^k_j.$$
If we further select $\omega=1$ then this becomes:
\begin{equation}
\label{pairwiseUpdate}
x^{k+1}_i=x^{k+1}_j= \frac{x^{k}_i +x^k_j}{2},
\end{equation}
which is the update of the standard pairwise randomized gossip algorithm first presented and analyzed in \cite{boyd2006randomized}.
\end{rem}

\subsubsection{AC system with Laplacian matrix $\bL$}
The AC system takes the form $\bL x=0$, where matrix $\bL \in \R^{n \times n}$ is the Laplacian matrix of the network. In this case, each row of the matrix corresponds to a node. Using the definition of the Laplacian, we have that $\bL_{i:}^\top=d_if_i-\sum_{j \in \cN_i}f_j$, where $f_i, f_j$ are unit coordinate vectors in $\R^n$ and $d_i$ is the degree of node $i \in \cV$. 

Thus, by letting $ \bB=\textbf{Diag}(w_1, w_2,\dots, w_n)$ to be the diagonal matrix of the weights we obtain:

\begin{equation}
\label{ansdl2}
\|\bB^{-1/2} \bL_{i :}^\top\|^2=\left\|\bB^{-1/2} (d_if_i-\sum_{j \in \cN_i}f_j) \right\|^2=\frac{d_i^2}{w_i}+\sum_{j \in \cN_i} \frac{1}{w_j}.
\end{equation}

In this case, the update rule \eqref{scaledRK} simplifies to:
\begin{eqnarray}
\label{updateSRKL}
x^{k+1}&\overset{b=0, \bA=\bL, \eqref{scaledRK}}{=}& x^k - \omega \frac{\bL_{i :} x^k}{\|\bB^{-1/2} \bL_{i :}^\top\|^2_{2}} \bB^{-1}\bL_{i :}^ \top \notag \\
&\overset{\eqref{ansdl2}}{=}&x^k - \omega \frac{\bL_{i :} x^k}{\frac{d_i^2}{w_i}+\sum_{j \in \cN_i} \frac{1}{w_j}} \bB^{-1}\bL_{i :}^ \top \notag \\
&=&x^k-\frac{\omega \left(d_i x^k_i- \sum_{j\in \cN_i} x^k_j \right) }{\frac{d_i^2}{w_i}+\sum_{j \in \cN_i} \frac{1}{w_j}}  \left(\frac{d_i}{w_i} f_i- \sum_{j\in \cN_i} \frac{1}{w_j } f_j \right).
\end{eqnarray}

From \eqref{updateSRKL}, it is clear that only coordinates $ \{i\} \cup \cN_i$ update their values. All the other coordinates remain unchanged. In particular, the value of the selected node $i$ (coordinate $i$) is updated as follows:

\begin{eqnarray}
x^{k+1}_i = x^k_i-\frac{\omega \left(d_i x^k_i- \sum_{j\in \cN_i} x^k_j \right) }{\frac{d_i^2}{w_i}+\sum_{j \in \cN_i} \frac{1}{w_j }}  \frac{d_i}{w_i}, 
\end{eqnarray}
while the values of its neighbors $ j \in \cN_i$ are updated as:
\begin{eqnarray}
x^{k+1}_j =x^k_j+\frac{\omega  \left(d_i x^k_i- \sum_{\ell \in \cN_i} x^k_\ell \right)  }{\frac{d_i^2}{w_i}+\sum_{\ell \in \cN_i} \frac{1}{w_\ell}} \frac{1}{w_j }. 
\end{eqnarray}

\begin{rem}
\label{naosjkl}
Let  $\omega=1$ and $\bB=r \bI$ where $r\in \R$ then the selected nodes update their values as follows:
\begin{eqnarray}
\label{aoskmpas}
x^{k+1}_i = \frac{ \sum_{\ell \in \{i \cup \cN_i\}} x^k_{\ell} }{d_i+1} \quad \text{and} \quad
x^{k+1}_j =x^k_j+\frac{ (d_i x^k_i- \sum_{\ell \in \cN_i} x^k_\ell)  }{d_i^2+d_i}.
\end{eqnarray}
That is, the selected node $i$ updates its value to the average of its neighbors and itself, while all the nodes $j \in \cN_i$ update their values using the current value of node $i$ and all nodes in $\cN_i$.
\end{rem}

In a wireless network, to implement such an update, node $i$ would first broadcast its current value to all of its neighbors. Then it would need to receive values from each neighbor to compute the sums over $\cN_i$, after which node $i$ would broadcast the sum to all neighbors (since there may be two neighbors $j_1, j_2 \in \cN_i$ for which $(j_1, j_2) \notin \cE$). In a wired network, using standard concepts from the MPI library, such an update rule could be implemented efficiently by defining a process group consisting of $\{i\} \cup \cN_i$, and performing one \texttt{Broadcast} in this group from $i$ (containing $x_i$) followed by an \texttt{AllReduce} to sum $x_\ell$ over $\ell \in \cN_i$. Note that the terms involving diagonal entries of $\bB$ and the degrees $d_i$ could be sent once, cached, and reused throughout the algorithm execution to reduce communication overhead.

\subsubsection{Details on complexity results}
Recall that the convergence rate of the sketch and project method (Algorithm~\ref{FullSkecth}) is equivalent to:
$$\rho \eqdef 1 - \omega (2-\omega) \lambda_{\min}^+(\bW),$$
where $\omega \in (0,2)$ and $\bW= \bB^{-1/2}\bA^\top \Exp[\bH] \bA\bB^{-1/2}$ (from Theorem~\ref{ConvergenceSketchProject}).  In this subsection we explain how the convergence rate of the scaled RK method \eqref{scaledRK} is modified for different choices of the main parameters of the method.

Let us choose $\omega=1$ (no over-relaxation). In this case, the rate is simplified to $\rho=1 - \lambda_{\min}^+$. 

Note that the different ways of modeling the problem (AC system) and the selection of the main parameters (weight matrix $\bB$ and distribution $\cD$) determine the convergence rate of the method through the spectrum of matrix $\bW$.

Recall that in the $k^{th}$ iterate of the scaled RK method \eqref{scaledRK} a random vector $\bS_k=e_i$ is chosen with probability $p_i>0$. 
For convenience, let us choose\footnote{Similar probabilities have been chosen in \cite{gower2015randomized} for the convergence of the standard RK method ($\bB=\bI$). The distribution $\cD$ of the matrices $\bS$ used in equation \eqref{convProb} is common in the area of randomized iterative methods for linear systems and is used to simplify the analysis and the expressions of the convergence rates. For more choices of distributions we refer the interested reader to \cite{gower2015randomized}. It is worth to mention that the probability distribution that optimizes the convergence rate of the RK and other projection methods can be expressed as the solution to a convex semidefinite program \cite{gower2015randomized, dai2014randomized}.}:
\begin{equation}
\label{convProb}
p_i= \frac{\|\bB^{-1/2} \bA_{i:}^\top\|^2}{\|\bB^{-1/2}\bA^\top\|^2_F} \;.
\end{equation}
Then we have that:
\begin{eqnarray}
\label{ajskmxa}
\Exp[ \bH] &=&\Exp[ \mS (\mS^\top \mA \bB^{-1} \mA^\top \mS)^\dagger \mS^\top] \notag\\
&=&\sum_{i=1}^m p_i \frac{e_i e_i^\top}{e_i^\top \mA \bB^{-1} \mA^\top e_i}=\sum_{i=1}^m p_i \frac{e_i e_i^\top}{\|\bA_{i:}^\top\|_{\bB^{-1} }^2}=\sum_{i=1}^m p_i \frac{e_i e_i^\top}{\|\bB^{-1/2} \bA_{i:}^\top\|^2}\notag\\
& \overset{\eqref{convProb}}{=} & \sum_{i=1}^m \frac{e_i e_i^\top}{\|\bB^{-1/2}\bA^\top\|^2_F}=\frac{1}{\|\bB^{-1/2}\bA^\top\|^2_F} \bI,
\end{eqnarray}
and
\begin{equation}
\label{Wconvinience}
\bW\overset{\eqref{MatrixW}, \eqref{ajskmxa}}{=} \frac{\bB^{-1/2}\bA^\top \bA\bB^{-1/2}}{\|\bB^{-1/2}\bA^\top\|^2_F}.
\end{equation}

\paragraph{Incidence Matrix:} Let us choose the AC system to be the one with the incidence matrix $\bA=\bQ$. Then $\|\bB^{-1/2}\bQ^\top\|^2_F= \sum_{i=1}^n \frac{d_i}{b_i}$ and we obtain 
$$\bW\overset{\bA=\bQ, \eqref{Wconvinience}}{=}\frac{\bB^{-1/2}\bL \bB^{-1/2}}{\|\bB^{-1/2}\bQ^\top\|^2_F}=\frac{\bB^{-1/2}\bL \bB^{-1/2}}{\sum_{i=1}^n \frac{d_i}{\bB_{ii}}}.$$

If we further have $\bB=\bD$, then $\bW=\frac{\bD^{-1/2}\bL \bD^{-1/2}}{n}$ and the convergence rate simplifies to:
$$\rho=1-\frac{\lambda_{\min}^+\left(\bD^{-1/2}\bL \bD^{-1/2}\right)}{n}=1-\frac{\lambda_{\min}^+\left(\bL^{sym} \right)}{n}.$$

If $\bB=r \bI$ where $r\in \R$ (solve the standard average consensus problem), then $\bW=\frac{\bL}{\|\bQ\|^2_F}= \frac{\bL}{\sum_{i=1}^n d_i}=\frac{\bL}{2m}$ and the convergence rate simplifies to 
\begin{equation}
\label{ratePairwise}
\rho=1-\frac{\lambda_{\min}^+(\bL)}{2m}=1-\frac{\ac(\cG)}{2m}.
\end{equation}
The convergence rate \eqref{ratePairwise} is identical to the rate proposed for the convergence of the standard pairwise gossip algorithm in \cite{boyd2006randomized}. Recall that in this special case the proposed gossip protocol has exactly the same update rule with the algorithm presented in \cite{boyd2006randomized}, see equation \eqref{pairwiseUpdate}.

\paragraph{Laplacian Matrix:} If we choose to formulate the AC system using the Laplacian matrix $\bL$, that is $\bA=\bL$, then $\|\bB^{-1/2}\bL^\top\|^2_F= \sum_{i=1}^n \frac{d_i(d_i+1)}{\bB_{ii}}$ and we have:
$$\bW\overset{\bA=\bL, \eqref{Wconvinience}}{=}\frac{\bB^{-1/2}\bL^\top \bL \bB^{-1/2}}{\sum_{i=1}^n \frac{d_i(d_i+1)}{\bB_{ii}}}.$$

If $\bB=\bD$, then the convergence rate simplifies to:
$$\rho=1-\frac{\lambda_{\min}^+(\bD^{-1/2}\bL^\top \bL\bD^{-1/2})}{\sum_{i=1}^n (d_i+1)}=1-\dfrac{\lambda_{\min}^+\left(\bD^{-1/2}\bL^2 \bD^{-1/2}\right)}{n+ \sum_{i=1}^n d_i}\overset{\sum_{i=1}^n d_i=2m}{=}1-\dfrac{\lambda_{\min}^+\left(\bD^{-1/2}\bL^2 \bD^{-1/2}\right)}{n+ 2m}.$$

If $\bB=r \bI$, where $r\in \R$, then $\bW=\frac{\bL^2}{\|\bL\|^2_F}= \frac{\bL^2}{\sum_{i=1}^n d_i(d_i+1)}$ and the convergence rate simplifies to $$\rho=1-\frac{\lambda_{\min}^+(\bL^2)}{\sum_{i=1}^n d_i(d_i+1)}=1-\frac{\ac(\cG)^2}{\sum_{i=1}^n d_i(d_i+1)}.$$

\subsection{Block gossip algorithms}
\label{BlockGossip}
Up to this point we focused on the basic connections between the convergence analysis of the sketch and project methods and the literature of randomized gossip algorithms. We show how specific variants of the randomized Kaczmarz method (RK) can be interpreted as gossip algorithms for solving the weighted and standard average consensus problems. 

In this part we extend the previously described methods to their block variants related to randomized block Kaczmarz (RBK) method \eqref{RBK}. In particular, in each step of the sketch and project method \eqref{sketchproject}, the random matrix $\bS$ is selected to be a random column submatrix of the $m \times m$ identity matrix corresponding to columns indexed by a random subset $C  \subseteq [m]$. That is, $\bS=\bI_{:C}$, where a set $C\subseteq [m]$ is chosen in each iteration independently, with probability $p_C\geq 0$ (see equation \eqref{RBK}). Note that in the special case that set $C$ is a singleton with probability 1 the algorithm is simply the randomized Kaczmarz method of the previous section.

To keep things simple, we assume that $\bB=\bI$ (standard average consensus, without weights) and choose the stepsize $\omega=1$. In the next section, we will describe gossip algorithms with heavy ball momentum and explain in detail how the gossip interpretation of RBK change in the more general case of $\omega \in (0,2)$.

Similar to the previous subsections, we formulate the consensus problem using either $\bA=\bQ$ or $\bA=\bL$ as the matrix in the AC system.
In this setup, the iterative process \eqref{sketchproject} has the form:
\begin{eqnarray}
\label{RBKgossip}
x^{k+1} &\overset{\eqref{RBK},\eqref{updateSkProj}}{=}& x^k - \bA^ \top \bI_{:C}(\bI_{:C}^\top \bA\bA^\top \bI_{:C})^{\dagger}\bI_{:C}^\top \bA x^k= x^k - \bA_{C:}^\top (\bA_{C:}\bA_{C:}^\top)^\dagger \bA_{C:}x^k,
\end{eqnarray}
which, as explained in the introduction, can be equivalently written as: 
\begin{equation}
\label{RBKaLgorithm}
x^{k+1}=\underset{x \in \R^n}{\operatorname{argmin}} \{\|x-x^k\|^2 \;:\; \bI_{:C}^\top \bA x=0\}.
\end{equation} 
Essentially in each step of this method the next iterate is evaluated to be the projection of the current iterate $x^k$ onto the solution set of a row subsystem of $\bA x=0$. 

\paragraph{AC system with Incidence Matrix:} In the case that $\bA=\bQ$ the selected rows correspond to a random subset $C \subseteq \cE$ of selected edges. While \eqref{RBKgossip} may seem to be a complicated algebraic (resp.\ variational) characterization of the method, due to our choice of $\bA=\bQ$ we have the following result which gives a natural interpretation of RBK as a gossip algorithm (see also Figure~\ref{fig:RBK}).

\begin{thm}[RBK as Gossip algorithm: RBKG]
\label{TheoremRBK}
Consider the AC system with the constraints being expressed using the Incidence matrix $\bQ$.  
Then each iteration of RBK (Algorithm~\eqref{RBKgossip}) works as gossip algorithm as follows:
\begin{enumerate}
\item Select a random set of edges $C \subseteq \cE$, 
\item Form subgraph $\cG_k$ of $\cG$ from the selected edges 
\item For each connected component of $\cG_k$, replace node values with their average.
\end{enumerate}
\end{thm}
\begin{proof}
See Appendix~\ref{ProofRBK}.
\end{proof}
Using the convergence result of general Theorem~\ref{ConvergenceSketchProject} and the form of matrix $\bW$ (recall that in this case we assume $\bB=\bI$, $\bS=\bI_{:C}\sim \cD$ and $\omega=1$), we obtain the following complexity for the algorithm:
\begin{equation}
\label{anisojxalksda}
\Exp[\|x^k-x^*\|^2]\leq \left[1 - \lambda_{\min}^+ \left( \Exp\left[\bQ_{C:}^\top (\bQ_{C:}\bQ_{C:}^\top)^\dagger \bQ_{C:}\right] \right) \right]^k \|x^0-x^*\|^2.
\end{equation}
For more details on the above convergence rate of randomized block Kaczmarz method with meaningfully bounds on the rate in a more general setting we suggest the papers \cite{RBK,l2015randomized}.

There is a very closed relationship between the gossip interpretation of RBK explained in Theorem~\ref{TheoremRBK} and several existing randomized gossip algorithms that in each step update the values of more than two nodes. 
For example the {\em path averaging} algorithm porposed in \cite{benezit2010order} is a special case of RBK, when set $C$ is restricted to correspond to a path of vertices. That is, in path averaging, in each iteration a path of nodes is selected and the nodes that belong to it update their values to their exact average.
A different example is the recently proposed clique gossiping \cite{liu2017clique} where the network is already divided into cliques and through a random procedure a clique is activated and the nodes of it update their values to their exact average.
In \cite{boyd2006randomized} a synchronous variant of gossip algorithm is presented where in each step multiple node pairs communicate exactly at the same time with the restriction that these simultaneously active node pairs are disjoint.

It is easy to see that all of the above algorithms can be cast as special cases of RBK if the distribution $\cD$ of the random matrices is chosen carefully to be over random matrices $\bS$ (column sub-matrices of Identity) that update specific set of edges in each iteration. As a result our general convergence analysis can recover the complexity results proposed in the above works.

Finally, as we mentioned, in the special case in which set $C$ is always a singleton, Algorithm~\eqref{RBKgossip} reduces to the standard randomized Kaczmarz method. This means that only a random edge is selected in each iteration and the nodes incident with this edge replace their local values with their average. This is the pairwise gossip algorithm of Boyd er al. \cite{boyd2006randomized} presented in equation \eqref{pairwiseUpdate}. Theorem~\ref{TheoremRBK} extends this interpretation to the case of the RBK method. 

\begin{figure}[htb]
\begin{minipage}[b]{1.0\linewidth}
  \centering
  \centerline{\includegraphics[scale=0.4]{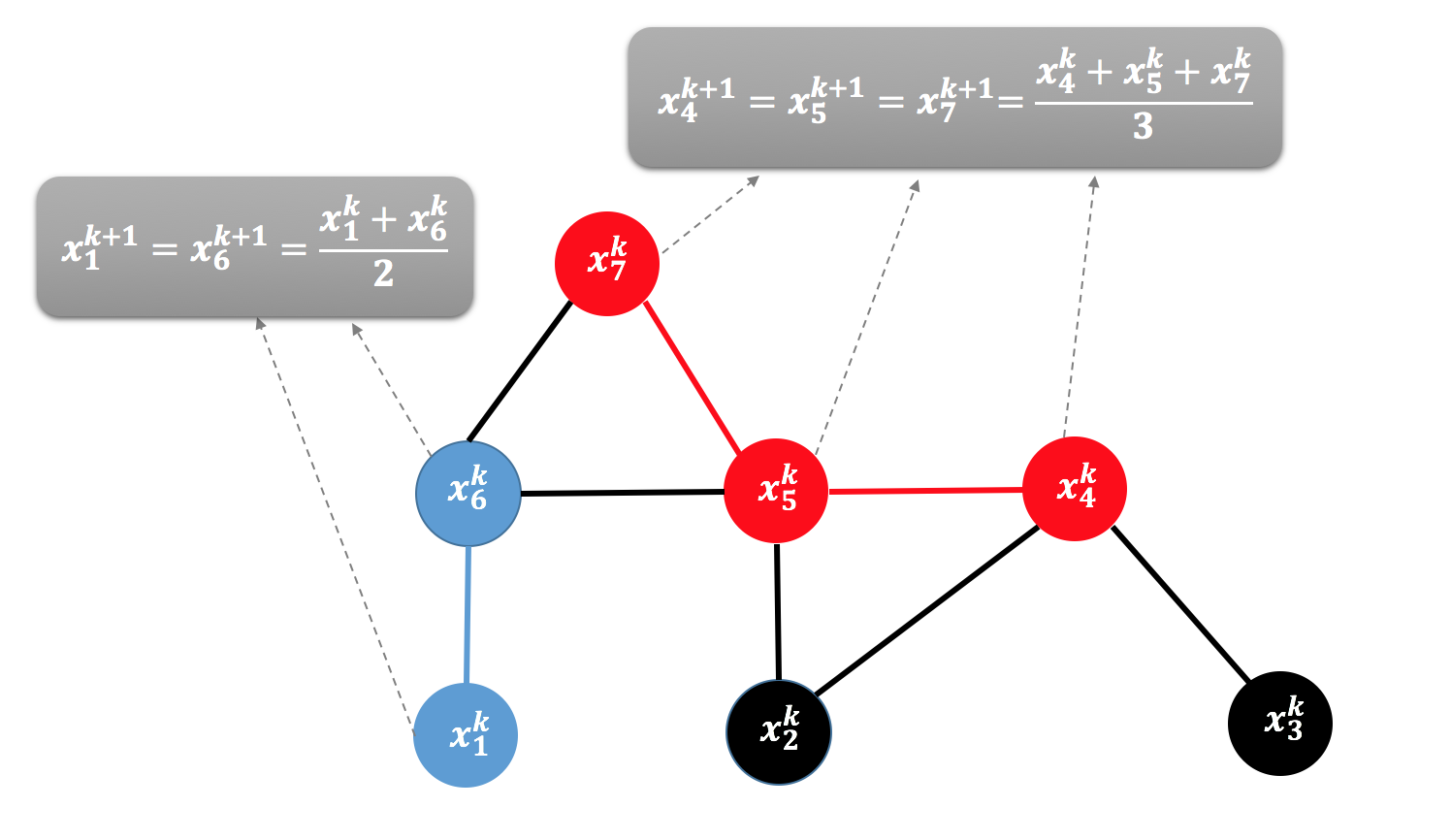}}
  \caption{\footnotesize Example of how the RBK method works as gossip algorithm in case of AC system with Incidence matrix. In the presented network 3 edges are randomly selected and a subgraph of two connected components (blue and red) is formed. Then the nodes of each connected component update their private values to their average.}
  \label{fig:RBK}
\end{minipage}
\end{figure}

\paragraph{AC system with Laplacian Matrix:}
For this choice of AC system the update is more complicated. To simplify the way that the block variant work as gossip we make an extra assumption. We assume that the selected rows of the constraint $\bI_{:C}^\top \bL x=0$ in update \eqref{RBKaLgorithm} have no-zero elements at different coordinates. This allows to have a direct extension of the serial variant presented in Remark~\ref{naosjkl}. Thus, in this setup, the RBK update rule \eqref{RBKgossip} works as gossip algorithm as follows:

\begin{enumerate}
\item $|C|$ nodes are activated (with restriction that the nodes are not neighbors and they do not share common neighbors)
\item For each node $i \in C$ we have the following update:
\begin{eqnarray}
x^{k+1}_i = \frac{ \sum_{\ell \in \{i \cup \cN_i\}} x^k_{\ell} }{d_i+1} \quad \text{and} \quad
x^{k+1}_j =x^k_j+\frac{ \left(d_i x^k_i- \sum_{\ell \in \cN_i} x^k_\ell \right)  }{d_i^2+d_i} .
\end{eqnarray}
\end{enumerate}
The above update rule can be seen as a parallel variant of update \eqref{aoskmpas}. Similar to the convergence in the case of Incidence matrix, the RBK for solving the AC system with a Laplacian matrix converges to $x^*$ with the following rate (using result of Theorem~\ref{ConvergenceSketchProject}):
$$\Exp[\|x^k-x^*\|^2]\leq \left[1 - \lambda_{\min}^+ \left( \Exp\left[\bL_{C:}^\top (\bL_{C:}\bL_{C:}^\top)^\dagger \bL_{C:}\right] \right) \right]^k \|x^0-x^*\|^2.$$

\section{Faster and Provably Accelerated Randomized Gossip Algorithms}
\label{AccelerateGossip}

The main goal in the design of gossip protocols is for the computation and communication to be done as quickly and efficiently as possible. In this section, our focus is precisely this. We design randomized gossip protocols which converge to consensus fast with provable accelerated linear rates. To the best of our knowledge, the proposed protocols are the first randomized gossip algorithms that converge to consensus with an accelerated linear rate. 

In particular, we present novel protocols for solving the average consensus problem where in each step all nodes of the network update their values but only a subset of them exchange their private values. The protocols are inspired from the recently developed accelerated variants of randomized Kaczmarz-type methods for solving consistent linear systems where the addition of momentum terms on top of the sketch and project update rule provides better theoretical and practical performance.

In the area of optimization algorithms, there are two popular ways to accelerate an algorithm using momentum. The first one is using the Polyak's heavy ball momentum \cite{polyak1964some} and the second one is using the theoretically much better understood momentum introduced by Nesterov \cite{nesterov1983method, nesterov2013introductory}.  Both momentum approaches have been recently proposed and analyzed to improve the performance of randomized iterative methods for solving linear systems.

To simplify the presentation, the accelerated algorithms and their convergence rates are presented for solving the standard average consensus problem ($\bB=\bI$). Using a similar approach as in the previous section, the update rules and the convergence rates can be easily modified to solve the more general weighted average consensus problem. For the protocols in this section we use the incidence matrix $\bA=\bQ$ or its normalized variant to formulate the AC system.

\subsection{Gossip algorithms with heavy ball momentum}
The recent works \cite{loizou2017linearly,loizou2017momentum} propose and analyze heavy ball momentum variants of several stochastic optimization algorithms for solving stochastic quadratic optimization problems and linear systems.  One of the proposed algorithms is the sketch and project method (Algorithm~\ref{FullSkecth}) with heavy ball momentum. In particular, the authors focus on explaining how the method can be interpreted as SGD with momentum---also known as the stochastic heavy ball method (SHB). SHB is a well known algorithm in the optimization literature for solving stochastic optimization problems, and extremely popular in areas such as deep learning \cite{sutskever2013importance, szegedy2015going, krizhevsky2012imagenet, wilson2017marginal}. However, even if SHB is used extensively in practice, its theoretical convergence behavior is not well understood.  \cite{loizou2017linearly,loizou2017momentum} were the first works that prove linear convergence of SHB in any setting. 

In this subsection we focus on the sketch and project method with heavy ball momentum. We present the main theorems showing its convergence performance as presented in \cite{loizou2017linearly,loizou2017momentum} and explain how special cases of the general method work as gossip algorithms when are applied to a special system encoding the underlying network.

\subsubsection{Sketch and project with heavy ball momentum}

The update rule of the sketch and project method with heavy ball momentum as proposed and analyzed in \cite{loizou2017linearly,loizou2017momentum} is formally presented in the following algorithm:

\begin{algorithm}[H]
	\caption{Sketch and Project with Heavy Ball Momentum}
	\label{SHBgradient}
	\small \small
	\begin{algorithmic}[1]
		\State {\bf Parameters:} Distribution $\mathcal{D}$ from which method samples matrices; stepsize/relaxation parameter $\omega \in \R$; momentum parameter $\beta$.
		\State {\bf Initialize:} $x^0,x^1 \in \R^n$
		\For{$k=1,2,\dots$} 
		\State Draw a fresh $\bS_k \sim \cD$
		\State   Set 
\begin{equation}
\label{SPmomentum}
x^{k+1}=x^k -\omega \bB^{-1}\bA^\top \bS_k (\bS_k^\top \bA \bB^{-1} \bA^\top \bS_k)^\dagger \bS_k^\top (\bA x^k-b) + \beta(x^k - x^{k-1}).
\end{equation}
		\EndFor
		\State {\bf Output:} The last iterate $x^k$
	\end{algorithmic}
\end{algorithm}

Using, $\bB=\bI$ and the same choice of distribution $\cD$ as in equations \eqref{RK} and \eqref{RBK} we can now obtain momentum variants of the RK and RBK as special case of the above algorithm as follows: 
\begin{itemize}
\item RK with momentum (mRK): \\
\begin{equation}
\label{ncajsolkmal}
x^{k+1}=x^k -\omega \frac{\bA_{i :} x^k -b_{i}}{\|\bA_{i :}\|^2} \bA_{i :}^ \top + \beta(x^k - x^{k-1}).
\end{equation}
\item RBK with momentum (mRBK):
\begin{equation}
\label{nacsklals}
x^{k+1}=x^k -\omega \bA_{C:}^\top (\bA_{C:}\bA_{C:}^\top)^\dagger (\bA_{C:}x^k-b_C) + \beta(x^k - x^{k-1}). 
\end{equation}
\end{itemize}

In \cite{loizou2017momentum}, two main theoretical results describing the behavior of Algorithm~\ref{SHBgradient}  (and as a result also the special cases mRK and mRBK) have been presented\footnote{Note that in \cite{loizou2017momentum} the analysis have been made on the same framework of Theorem~\ref{ConvergenceSketchProject} with general positive definite matrix $\bB$.}. 

\begin{thm}[Theorem 1, \cite{loizou2017momentum}]
\label{okams}
Choose $x^0= x^1\in \R^n$. Let $\{x^k\}_{k=0}^\infty$ be the sequence of random iterates produced by Algorithm~\ref{SHBgradient} and let assume exactness.  Assume $0< \omega < 2$ and $\beta \geq 0$ and that the expressions
$a_1 \eqdef 1+3\beta+2\beta^2 - (\omega(2-\omega) +\omega\beta)\lambda_{\min}^+$ and
$a_2 \eqdef \beta +2\beta^2 + \omega \beta \lambda_{\max}$
satisfy $a_1+a_2<1$. Set $x^*=\Pi_{\cL,\bB}(x^0)$. Then 
\begin{equation}\label{eq:nfiug582}\Exp[\|x^{k}-x^*\|^2] \leq q^k (1+\delta)  \|x^0-x^*\|^2,
\end{equation}
where  $q=\frac{1}{2} (a_1+\sqrt{a_1^2+4a_2})$ and $\delta=q-a_1$. Moreover, $a_1+a_2 \leq q <1$.
\end{thm}

\begin{thm}[Theorem 4, \cite{loizou2017momentum}]
\label{asdlkand}
Assume exactness. Let $\{x^k\}_{k=0}^{\infty}$ be the sequence of random iterates produced by Algorithm~\ref{SHBgradient}, started with $x^0= x^1\in \R^n$, with relaxation parameter (stepsize)  $0<\omega \leq1/\lambda_{\max}$ and momentum parameter  $(1-\sqrt{\omega \lambda_{\min}^+})^2 < \beta <1$. Let $x^* = \Pi_{\cL,\bI}(x^0)$. Then there exists a constant $C >0$ such that for all $k\geq0$ we have 
$$\|\Exp[x^{k} -x^*]\|^2  \leq \beta^k C.$$
\end{thm}

Using Theorem~\ref{asdlkand} and by a proper combination of the stepsize $\omega$ and the momentum parameter $\beta$, Algorithm~\ref{SHBgradient} enjoys an accelerated linear convergence rate in mean \cite{loizou2017momentum}. 

\begin{cor}
\begin{enumerate}
\item[(i)] If $ \omega= 1$ and $\beta= (1- \sqrt{0.99 \lambda_{\min}^+}) ^2$, then the iteration complexity of Algorithm~\ref{SHBgradient} becomes: $O\left(\sqrt{1/ \lambda_{\min}^+} \log(1/\epsilon) \right)$.
\item[(ii)] If $ \omega= 1/\lambda_{\max}$ and $\beta= (1- \sqrt{ 0.99\lambda_{\min}^+/ \lambda_{\max}})^2$, then the iteration complexity of Algorithm~\ref{SHBgradient} becomes: $O\left(\sqrt{\lambda_{\max}/ \lambda_{\min}^+} \log(1/\epsilon)\right)$.
\end{enumerate}
\end{cor}

Having presented Algorithm~\ref{SHBgradient} and its convergence analysis results, let us now describe its behavior as a randomized gossip protocol when applied on the AC system $\bA x=0$ with $\bA=\bQ \in |\cE| \times n$ (incidence matrix of the network).  

Note that since $b=0$ (from the AC system definition), method \eqref{SPmomentum} can be simplified to:
\begin{equation}
\label{momentumupdateb0}
x^{k+1}=\left[\bI - \omega \bA^\top \bS_k (\bS_k^\top \bA \bA^\top \bS_k)^\dagger \bS_k^\top \bA \right] x^k + \beta(x^k - x^{k-1}).
\end{equation}

In the rest of this section we focus on  two special cases of \eqref{momentumupdateb0}: RK with heavy ball momentum (equation \eqref{ncajsolkmal} with $b_i=0$) and RBK with heavy ball momentum (equation \eqref{nacsklals} with $b_C=0$).

\subsubsection{Randomized Kaczmarz gossip with heavy ball momentum}
As we have seen in previous section when the standard RK is applied to solve the AC system $\bQ x=0$, one can recover the famous pairwise gossip algorithm~\cite{boyd2006randomized}.  Algorithm~\ref{RKmomentum} describes how a relaxed variant of randomized Kaczmarz with heavy ball momentum ($0<\omega < 2$ and $ 0 \leq \beta <1$) behaves as a gossip algorithm. See also Figure~\eqref{fig:mRK} for a graphical illustration of the method.

\begin{algorithm}[t!]
	\caption{mRK: Randomized Kaczmarz with momentum as a gossip algorithm}
	\label{RKmomentum}
	\small \small
	\begin{algorithmic}[1]
		\State {\bf Parameters:} Distribution $\mathcal{D}$ from which method samples matrices; stepsize/relaxation parameter $\omega \in \R$; heavy ball/momentum parameter $\beta$.
		\State {\bf Initialize:} $x^0 ,x^1 \in \R^n$
		\For{$k=1,2,\dots$} 
		\State Pick an edge $e=(i,j)$ following the distribution $\cD$
		\State The values of the nodes are updated as follows:
\begin{itemize}
\item Node $i$: $x_i^{k+1}= \frac{2-\omega}{2}x_i^k+ \frac{\omega}{2}x_j^k+\beta (x_i^k - x_i^{k-1})$
\item Node $j$: $x_j^{k+1}= \frac{2-\omega}{2} x_j^k+\frac{\omega}{2}x_i^k+\beta (x_j^k - x_j^{k-1})$
\item Any other node $\ell$: $x_\ell^{k+1}=x_\ell^k+\beta (x_\ell^k - x_\ell^{k-1})$
\end{itemize}
		\EndFor
		\State {\bf Output:} The last iterate $x^k$
	\end{algorithmic}
\end{algorithm}

\begin{rem}
In the special case that $\beta=0$ (zero momentum) only the two nodes of edge $e=(i,j)$ update their values. In this case the two selected nodes do not update their values to their exact average but to a convex combination that depends on the stepsize $\omega \in (0,2)$. To obtain the pairwise gossip algorithm of \cite{boyd2006randomized}, one should further choose $\omega=1$.
\end{rem}

\textbf{Distributed Nature of the Algorithm:} Here we highlight a few ways to implement mRK in a distributed fashion.
\begin{itemize}
\item \emph{Pairwise broadcast gossip:}  In this protocol each node $i \in \cV$ of the network $\cG$ has a clock that ticks at the times of a rate 1 Poisson process. The inter-tick times are exponentially distributed, independent across nodes, and independent across time. This is equivalent to a global clock ticking at a rate $n$ Poisson process which wakes up an edge of the network at random. In particular, in this implementation mRK works as follows:  In the $k^{th}$ iteration (time slot) the clock of node $i$ ticks and node $i$ randomly contact one of its neighbors and simultaneously broadcast a signal to inform the nodes of the whole network that is updating (this signal does not contain any private information of node $i$). The two nodes $(i,j)$ share their information and update their private values following the update rule of Algorithm~\ref{RKmomentum} while all the other nodes update their values using their own information. In each iteration only one pair of nodes exchange their private values. 
 
\item \emph{Synchronous pairwise gossip:} In this protocol a single global clock is available to all nodes. The time is assumed to be slotted commonly across nodes and in each time slot only a pair of nodes of the network is randomly activated and exchange their information following the update rule of Algorithm~\ref{RKmomentum}.  The remaining not activated nodes update their values using their own last two private values.  Note that this implementation of mRK comes with the disadvantage that it requires a central entity which in each step requires to choose the activated pair of nodes\footnote{We speculate that a completely distributed synchronous gossip algorithm that finds pair of nodes in a distributed manner without any additional computational burden can be design following the same procedure proposed in Section III.C of  \cite{boyd2006randomized}.}.

\item  \emph{Asynchronous pairwise gossip with common counter:} 
Note that the update rule of the selected pair of nodes $(i,j)$ in Algorithm~\ref{RKmomentum} can be rewritten as follows:
$$ x_i^{k+1}= x_i^k + \beta (x_i^k - x_i^{k-1}) + \frac{\omega}{2} (x_j^k -x_i^k),$$
$$x_j^{k+1}= x_j^k + \beta (x_j^k - x_j^{k-1}) + \frac{\omega}{2} (x_i^k -x_j^k).$$
In particular observe that the first part of the above expressions $x_i^k + \beta (x_i^k - x_i^{k-1})$ (for the case of node $i$) is exactly the same with the update rule of the non activate nodes at $k^{th}$ iterate (check step 5 of Algorithm~\ref{RKmomentum}) . Thus, if we assume that all nodes share a common counter that keeps track of the current iteration count and that each node $i \in \cV$ remembers the iteration counter $k_i$ of when it was last activated, then step 5 of Algorithm~\ref{RKmomentum} takes the form:
\begin{itemize}
\item $ x_i^{k+1}= i_k \left[x_i^k + \beta (x_i^k - x_i^{k-1}) \right]+ \frac{\omega}{2} (x_j^k -x_i^k),$
\item $x_j^{k+1}= j_k \left[x_j^k + \beta (x_j^k - x_j^{k-1}) \right] + \frac{\omega}{2} (x_i^k -x_j^k),$
\item $k_i = k_j =k+1,$
\item Any other node $\ell$: $x_\ell^{k+1}=x_\ell^k,$
\end{itemize}
where $i_k=k-k_{i}$ ($j_k=k-k_{j}$) denotes the number of iterations between the current iterate and the last time that the $i^{th}$ ($j^{th}$) node is activated. In this implementation only a pair of nodes communicate and update their values in each iteration (thus the justification of asynchronous), however it requires the nodes to share a common counter that keeps track the current iteration count in order to be able to compute the value of $i_k=k-k_{i}$.
\end{itemize}

\begin{figure}[t!]
\vspace{6pt}
\begin{minipage}[b]{1.0\linewidth}
  \centering
  \centerline{\includegraphics[scale=0.55]{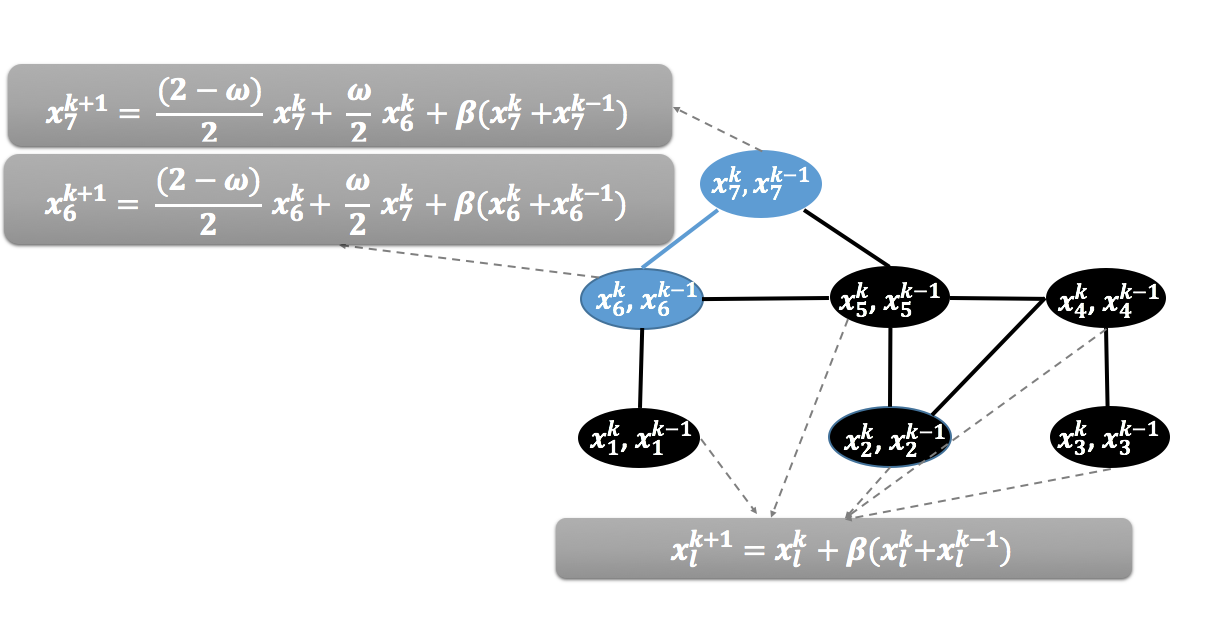}}
  \caption{\footnotesize Example of how mRK works as gossip algorithm. In the presented network the edge that connects nodes $6$ and $7$ is randomly selected. The pair of nodes exchange their information and update their values following the update rule of the Algorithm~\ref{RKmomentum} while the rest of the nodes, $\ell \in [5]$, update their values using only their own previous private values.}
  \label{fig:mRK}
\end{minipage}
\end{figure}

\subsubsection{Connections with existing fast randomized gossip algorithms}
\label{connectionOfAcceleratedMethods}

In the randomized gossip literature there is one particular method closely related to our approach. It was first proposed in \cite{cao2006accelerated} and its analysis under strong conditions was presented in \cite{liu2013analysis}. In this work local memory is exploited by installing shift registers at each agent. In particular we are interested in the case of two registers where the first stores the agent's current value and the second the agent's value before the latest update. The algorithm can be described as follows. Suppose that edge $e=(i,j)$ is chosen at time $k$. Then,
\begin{itemize}
\item Node $i$: $x_i^{k+1}= \omega(\frac{x_i^k+x_j^k}{2})+(1-\omega)x_i^{k-1},$
\item Node $j$: $x_i^{k+1}=  \omega(\frac{x_i^k+x_j^k}{2})+(1-\omega)x_j^{k-1},$
\item Any other node $\ell$: $x_\ell^{k+1}=x_\ell^k,$
\end{itemize}
where $\omega \in [1,2)$. The method was analyzed in \cite{liu2013analysis} under a strong assumption on the probabilities of choosing the pair of nodes, that as the authors mentioned, is unrealistic in practical scenarios, and for networks like the random geometric graphs. At this point we should highlight that the results presented in \cite{loizou2017momentum} hold for essentially any distribution $\cD$ \footnote{The only restriction is the exactness condition to be satisfied. See Theorem~\ref{okams}.} and as a result in the proposed gossip variants with heavy ball momentum such problem cannot occur.

Note that, in the special case that we choose $\beta=\omega-1$ in the update rule of Algorithm~\ref{RKmomentum} is simplified to:
\begin{itemize}
\item Node $i$: $x_i^{k+1}= \omega(\frac{x_i^k+x_j^k}{2})+(1-\omega)x_i^{k-1},$
\item Node $j$: $x_i^{k+1}=  \omega(\frac{x_i^k+x_j^k}{2})+(1-\omega)x_j^{k-1},$
\item Any other node $\ell$: $x_\ell^{k+1}=\omega x_\ell^k+(1-\omega)x_\ell^{k-1}.$
\end{itemize}

In order to apply Theorem~\ref{okams}, we need to assume that $0< \omega < 2$ and $\beta=\omega-1 \geq 0$ which also means that $\omega \in [1,2)$. Thus for $\omega \in [1,2)$ and momentum parameter $\beta=\omega-1$ it is easy to see that our approach is very similar to the shift-register algorithm. Both methods update the selected pair of nodes in the same way. However, in Algorithm~\ref{RKmomentum} the not selected nodes of the network do not remain idle but instead update their values using their own previous information.

By defining the momentum matrix $\bM=\text{\textbf{Diag}}(\beta_{1},\beta_{2},\dots, \beta_{n})$, the above closely related algorithms can be expressed, in vector form, as:
\begin{equation}
\label{updatewithB}
x^{k+1} = x^k - \frac{\omega}{2} (x_i^k-x_j^k)(e_i-e_j) + \bM (x^k-x^{k-1}).
\end{equation}

In particular, in mRK every diagonal element of matrix $\bM$ is equal to $\omega-1$, while in the algorithm of \cite{cao2006accelerated, liu2013analysis} all the diagonal elements are zeros except the two values that correspond to nodes $i$ and $j$ that are equal to $\beta_{i}=\beta_{j}=\omega-1$.

\begin{rem}
The shift register algorithm of \cite{liu2013analysis} and Algorithm~\ref{RKmomentum} of this work can be seen as the two limit cases of the update rule \eqref{updatewithB}. As we mentioned, the shift register method \cite{liu2013analysis} uses only two non-zero diagonal elements in $\bM$, while our method has a full diagonal. We believe that further methods can be developed in the future by exploring the cases where more than two but not all elements of the diagonal matrix $\bM$ are non-zero. It might be possible to obtain better convergence if one carefully chooses these values based on the network topology. We leave this as an open problem for future research.
\end{rem}

\subsubsection{Randomized block Kaczmarz gossip with heavy ball momentum}
Recall that Theorem~\ref{TheoremRBK} explains how RBK (with no momentum and no relaxation) can be interpreted as a gossip algorithm. In this subsection by using this result we explain how relaxed RBK with momentum works. Note that the update rule of RBK with momentum can be rewritten as follows:
\begin{equation}
\label{updateRBK2}
x^{k+1} \overset{\eqref{momentumupdateb0},\eqref{nacsklals}}{=} \omega \left(\bI- \bA_{C:}^\top (\bA_{C:}\bA_{C:}^\top)^\dagger \bA_{C:} \right) x^k+(1-\omega)x^k +\beta(x^k-x^{k-1}),
\end{equation}
and recall that $x^{k+1} =\left(\bI - \bA_{C:}^\top (\bA_{C:}\bA_{C:}^\top)^\dagger \bA_{C:} \right) x^k$ is the update rule of the standard RBK  \eqref{RBKgossip}.

Thus, in analogy to the standard RBK, in the $k^{th}$ step, a random set of edges is selected and $q \leq n$ connected components are formed as a result. This includes the connected components that belong to both sub-graph $\cG_k$ and also the singleton connected components (nodes outside the $\cG_k$). Let us define the set of the nodes that belong in the $r \in [q]$ connected component at the $k^{th}$ step $\cV_r^k$, such that $\cV= \cup_{r\in [q]} \cV_r^k$ and $|\cV|=\sum_{r=1}^{q} |\cV_r^k|$ for any $k>0$. 

Using the update rule \eqref{updateRBK2}, Algorithm~\ref{RBKmomentum} shows how mRBK is updating the private values of the nodes of the network (see also Figure~\ref{fig:mRBK} for the graphical interpretation).

\begin{algorithm}[t!]
	\caption{mRBK: Randomized Block Kaczmarz Gossip with momentum}
	\label{RBKmomentum}
	\small \small
	\begin{algorithmic}[1]
		\State {\bf Parameters:} Distribution $\mathcal{D}$ from which method samples matrices;  stepsize/relaxation parameter $\omega \in \R$;  heavy ball/momentum parameter $\beta$.
		\State {\bf Initialize:} $x^0,x^1 \in \R^n$
		\For{$k=1,2,...$} 
		\State Select a random set of edges $\cS \subseteq \cE$
		\State Form subgraph $\cG_k$ of $\cG$ from the selected  edges 
		\State Node values  are updated as follows:
\begin{itemize}
\item For each connected component $\cV_r^k$ of $\cG_k$, replace the values of its nodes with: 
\begin{equation}
\label{updateruelblock}
x_i^{k+1}=\omega \frac{\sum_{j \in \cV_r^k} x_j^{k}}{|\cV_r^k|} +(1-\omega)x_i^k+\beta (x_i^k-x_i^{k-1}).
\end{equation}
\item Any other node $\ell$: $x_\ell^{k+1}=x_\ell^k+\beta (x_\ell^k - x_\ell^{k-1})$
\end{itemize}
		\EndFor
		\State {\bf Output:} The last iterate $x^k$
	\end{algorithmic}
\end{algorithm}

Note that in the update rule of mRBK the nodes that are not attached to a selected edge (do not belong in the sub-graph $\cG_k$) update their values via $x_\ell^{k+1}=x_\ell^k+\beta (x_\ell^k - x_\ell^{k-1})$. By considering these nodes as singleton connected components their update rule is exactly the same with the nodes of sub-graph $\cG_k$. This is easy to see as follows:
\begin{eqnarray}
x_\ell^{k+1}&\overset{\eqref{updateruelblock}}{=}&\omega \frac{\sum_{j \in \cV_r^k} x_j^{k}}{|\cV_r^k|} +(1-\omega)x_\ell^k+\beta (x_\ell^k-x_\ell^{k-1})\notag\\
&\overset{|\cV_r^k|=1}{=}&\omega x_\ell^k +(1-\omega)x_\ell^k+\beta (x_\ell^k-x_\ell^{k-1})\notag\\
&=& x_\ell^k+\beta (x_\ell^k - x_\ell^{k-1}).
\end{eqnarray}

\begin{rem}
In the special case that only one edge is selected in each iteration ($\bS_k \in \R^{m \times 1}$)  the update rule of mRBK is simplified to the update rule of mRK. In this case the sub-graph $\cG_k$ is the pair of the two selected edges.  
\end{rem}

\begin{rem}
In previous section we explained how several existing gossip protocols for solving the average consensus problem are special cases of the RBK (Theorem~\ref{TheoremRBK}). For example two gossip algorithms that can be cast as special cases of the standard RBK are the path averaging proposed in \cite{benezit2010order} and the clique gossiping \cite{liu2017clique}. In path averaging, in each iteration a path of nodes is selected and its nodes update their values to their exact average ($\omega=1$). In clique gossiping, the network is already divided into cliques and through a random procedure a clique is activated and the nodes of it update their values to their exact average ($\omega=1$). Since mRBK contains the standard RBK as a special case (when $\beta=0$), we expect that these special protocols can also be accelerated with the addition of momentum parameter $\beta \in (0,1)$.
\end{rem}

\begin{figure}[t!]
\begin{minipage}[b]{1.0\linewidth}
  \centering
  \vspace{6pt}
  \centerline{\includegraphics[scale=0.6]{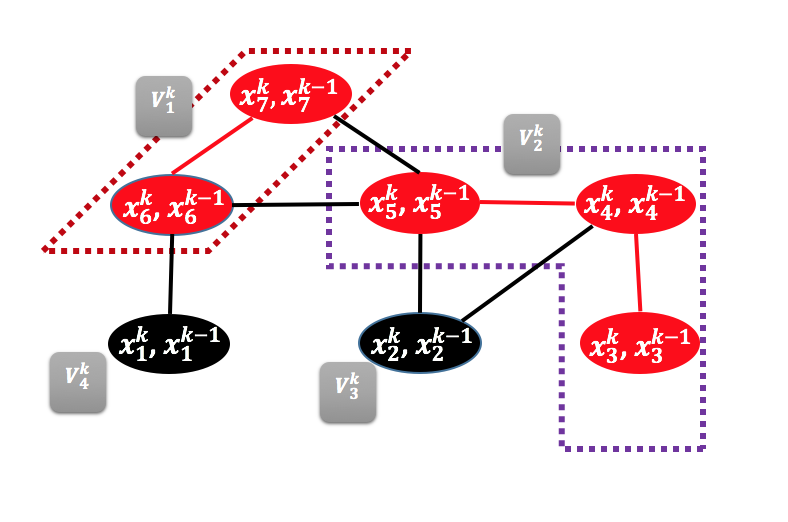}}
  \caption{\footnotesize Example of how the mRBK method works as gossip algorithm. In the presented network the red edges are randomly chosen in the $k^{th}$ iteration, and they form subgraph $\cG_k$ and four connected component. In this figure $V_1^k$ and $V_2^k$ are the two connected components that belong in the subgraph $\cG_k$ while $V_3^k$ and $V_4^k$ are the singleton connected components. Then the nodes update their values by communicate with the other nodes of their connected component using the update rule \eqref{updateruelblock}. For example the node number 5 that belongs in the connected component $V_2^k$ will update its value using the values of node 4 and 3 that also belong in the same component as follows:
$x_5^{k+1}=\omega \frac{x_3^k+x_4^k+x_5^k}{3} +(1-\omega)x_5^k+\beta (x_5^k-x_5^{k-1})$.}
  \label{fig:mRBK}
\end{minipage}
\end{figure}

\subsubsection{Mass preservation}
One of the key properties of some of the most efficient randomized gossip algorithms is mass preservation. That is, the sum (and as a result the average) of the private values of the nodes remains fixed during the iterative procedure ($\textstyle \sum_{i=1}^{n}x_i^{k}=\sum_{i=1}^{n}x_i^{0}, \quad \forall k \geq 1$). The original pairwise gossip algorithm proposed in \cite{boyd2006randomized} satisfied the mass preservation property, while exisiting fast gossip algorithms \cite{cao2006accelerated,liu2013analysis}  preserving a scaled sum.  
In this subsection we show that mRK and mRBK gossip protocols presented above satisfy the mass preservation property. In particular, we prove mass preservation for the case of the block randomized gossip protocol (Algorithm~\ref{RBKmomentum}) with momentum. This is sufficient since the randomized Kaczmarz gossip with momentum (mRK), Algorithm~\ref{RKmomentum} can be cast as special case.

\begin{thm}
Assume that $x^0=x^1=c$. That is, the two registers of each node have the same initial value.  Then for the Algorithms~\ref{RKmomentum} and \ref{RBKmomentum} we have $\sum_{i=1}^{n}x_i^k=\sum_{i=1}^{n}c_i$ for any $k\geq 0$ and as a result, $\frac{1}{n}\sum_{i=1}^{n}x_i^k=\bar{c}$. 
\end{thm}
\begin{proof}
We prove the result for the more general Algorithm~\ref{RBKmomentum}. Assume that in the $k^{th}$ step of the method $q$ connected components are formed.  Let the set of the nodes of each connected component be $\cV_r^k$ so that $\cV= \cup_{r=\{1,2,...q\}} \cV_r^k$ and $|\cV|=\sum_{{r}=1}^{q} |\cV_r^k|$ for any $k>0$.  Thus:
\begin{equation}
\label{generalsum}
\textstyle \sum_{i=1}^{n}x_i^{k+1}=\sum_{i \in \cV_1^k} x_i^{k+1} +\dots + \sum_{i \in \cV_q^k} x_i^{k+1}.
\end{equation}
Let us first focus, without loss of generality, on  connected component $r \in [q]$ and simplify the expression for the sum of its nodes:
\begin{eqnarray}
 \sum_{i\in \cV_r^k} x_i^{k+1}
&\overset{\eqref{updateruelblock}}=& \textstyle \sum_{i \in \cV_r^k} \omega \frac{\sum_{j \in \cV_r^k} x_j^{k}}{|\cV_r^k|} +   (1-\omega) \sum_{i \in \cV_r^k} x_i^k  +\beta \sum_{i \in \cV_r^k}  (x_i^k-x_i^{k-1})\notag\\
 &=&|\cV_r^k| \frac{\omega \sum_{j \in \cV_r^k} x_j^{k}}{|\cV_r^k|}+ (1-\omega) \sum_{i \in \cV_r^k} x_i^k 
 +\beta \sum_{i \in \cV_r^k}  (x_i^k-x_i^{k-1})\notag\\
 &=&(1+\beta) \sum_{i \in \cV_r^k}x_i^k-\beta \sum_{i \in \cV_r^k}x_i^{k-1}.
\end{eqnarray}
 By substituting this for all $r \in [q]$ into the right hand side of \eqref{generalsum} and from the fact that $\cV= \cup_{r\in [q]}
 \cV_r^k$, we obtain:
$$ \sum_{i=1}^{n}x_i^{k+1}= (1+\beta) \sum_{i=1}^{n}x_i^k-\beta \sum_{i=1}^{n} x_i^{k-1}.$$
Since $x^0=x^1$, we have $\sum_{i=1}^{n}x_i^{0}=\sum_{i=1}^{n}x_i^{1}$, and as a result $
\sum_{i=1}^{n}x_i^{k} = \sum_{i=1}^{n}x_i^{0}$ for all $ k \geq 0$.
\end{proof}
\subsection{Provably accelerated randomized gossip algorithms}
\label{accSubsection}

In the results of this subsection we focus on one specific case of the Sketch and Project framework, the RK method \eqref{RK}. We present two accelerated variants of the randomized Kaczmarz (RK) where the Nesterov's momentum is used, for solving consistent linear systems and we describe their theoretical convergence results. Based on these methods we propose two provably accelerated gossip protocols, along with some remarks on their implementation.

\subsubsection{Accelerated Kaczmarz methods using Nesterov's momentum}
\label{AcceleratedVariants}
There are two different but very similar ways to provably accelerate the randomized Kaczmarz method using Nesterov's acceleration. The first paper that proves \emph{asymptotic convergence }with an accelerated linear rate is \cite{liu2016accelerated}. The proof technique is similar to the framework developed by Nesterov in \cite{nesterov2012efficiency} for the acceleration of coordinate descent methods.  In \cite{tu2017breaking,gower2018accelerated} a modified version for the selection of the parameters was proposed and a \emph{non-asymptotic} accelerated linear rate was established. In Algorithm~\ref{alg:AccKaczmarz}, pseudocode of the Accelerated Kaczmarz method (AccRK) is presented where both variants can be cast as special cases, by choosing the parameters with the correct way. 
\begin{algorithm}[!h]
\begin{algorithmic}[1]
\State Data: Matrix $\mA\in \R^{m\times n}$; vector $b\in \R^m$
\State Choose $x^0\in \R^n$ and set $v^0 = x^0$
\State Parameters: 
Evaluate the sequences of the scalars $\alpha_k, \beta_k ,\gamma_k$ following one of two possible options.
\For {$k = 0, 1, 2, \dots, K$}
 \State $y^k = \alpha_k v^k + (1-\alpha_k) x^k$ 
\State Draw a fresh sample $i_k \in [m]$ with equal probability 
\State $x^{k+1} = y^k - \frac{\bA_{i_k :} y^k -b_{i_k}}{\|\bA_{i_k :}\|^2} \bA_{i_k :}^ \top.$ 
\State $v^{k+1} = \beta_k v^k + (1-\beta_k) y^k -\gamma_k \frac{\bA_{i_k :} y^k -b_{i_k}}{\|\bA_{i_k :}\|^2} \bA_{i_k :}^ \top.$
\EndFor
\end{algorithmic}
\caption{Accelerated Randomized Kaczmarz Method (AccRK)}
\label{alg:AccKaczmarz}
\end{algorithm}

There are two options for selecting the parameters of the AccRK for solving consistent linear systems with normalized matrices, which we describe next. 
\begin{enumerate}
\item From \cite{liu2016accelerated}: 
Choose $\lambda \in [0,\lambda_{\min}^+(\bA^\top \bA)]$ and set $\gamma_{-1}=0$.
Generate the sequence $\{\gamma_k: k=0,1,\dots, K+1\}$ by choosing $\gamma_k$ to be the largest root of $$\gamma_k^2-\frac{\gamma_k}{m}=(1-\frac{\gamma_k}\lambda{m})\gamma_{k-1}^2,$$ and generate the sequences $\{\alpha	_k :  k=0,1,\dots,K+1\}$ and $\{\beta_k :  k=0,1,\dots,K+1\}$ by setting $$\alpha_k=\frac{m-\gamma_k\lambda}{\gamma_k(m^2-\lambda)}, \quad \beta_k=1-\frac{\gamma_k \lambda}{m}.$$
\item From \cite{gower2018accelerated}: Let 
\begin{equation}
\label{thenu}
\nu= \max_{u \in \range{\mA^\top}   } \frac{ u^\top \left[\sum_{i=1}^m  \mA_{i:}^\top \mA_{i:} (\mA^\top \mA)^\dagger \mA_{i:}^\top \mA_{i:} \right]u }{ u^\top \frac{\bA^\top \bA }{m}u }.
\end{equation}
Choose the three sequences to be fixed constants as follows:
$\beta_k =\beta = 1-\sqrt{ \frac{\lambda_{\min}^+(\bW)}{\nu}} $, \;$\gamma_k=\gamma = \sqrt{ \frac{1}{\lambda_{\min}^+(\bW) \nu}} $, \; $\alpha_k=\alpha =  \frac{1}{1+\gamma \nu} \in (0,1)$ where $\bW=\frac{\bA^\top \bA}{m}$.
\end{enumerate}
\subsubsection{Theoretical guarantees of AccRK}
The two variants (Option 1 and Option 2) of AccRK are closely related, however their convergence analyses are different. Below we present the theoretical guarantees of the two options as presented in \cite{liu2016accelerated} and \cite{gower2018accelerated}.
\begin{thm}[\cite{liu2016accelerated}]
\label{causjoakls}
Let $\{x^k\}_{k=0}^\infty$ be the sequence of random iterates produced by Algorithm~\ref{alg:AccKaczmarz} with the Option 1 for the parameters. Let $\bA$ be normalized matrix and let $\lambda \in [0,\lambda_{\min}^+(\bA^\top \bA)]$. Set $\sigma_1=1+\frac{\sqrt{\lambda}}{2m}$ and $\sigma_2=1-\frac{\sqrt{\lambda}}{2m}$. Then for any $k\geq 0$ we have that:
$$\Exp[\|x^{k}-x^*\|^2 ] \leq \frac{4 \lambda}{(\sigma_1^{k}-\sigma^{k}_2)^2}\|x^0-x^*\|^2_{(\bA^\top \bA)^\dagger}.$$
\end{thm}
\begin{cor}[\cite{liu2016accelerated}]
\label{basjda}
Note that as $k\rightarrow\infty$, we have that $\sigma^k_2\rightarrow 0$. This means that the decrease of the right hand side is governed mainly by the behavior of the term $\sigma_1$ in the denominator and as a result the method converge \emph{asymptotically} with a decrease factor per iteration:
$\sigma_1^{-2}=(1+\frac{\sqrt{\lambda}}{2m})^{-2}\approx 1-\frac{\sqrt{\lambda}}{m}.$
That is, as $k\rightarrow\infty$:
$$\Exp[\|x^{k}-x^*\|^2 ] \leq \left(1- \sqrt{\lambda}/m \right)^k 4 \lambda \|x^0-x^*\|^2_{(\bA^\top \bA)^\dagger}$$
\end{cor}

Thus, by choosing $\lambda=\lambda_{\min}^+$ and for the case that $\lambda_{\min}^+$ is small, Algorithm~\ref{alg:AccKaczmarz} will have significantly faster convergence rate than RK. Note that the above convergence results hold only for normalized matrices $\bA \in \R^{m \times n}$, that is matrices that have $\|\bA_{i:}\|=1$ for any $i \in m$. 

Using Corollary~\ref{basjda}, Algorithm~\ref{alg:AccKaczmarz} with the first choice of the parameters converges linearly with rate $\left(1- \sqrt{\lambda}/m \right)$. That is, it requires $O \left( m/ \sqrt{\lambda} \log(1/\epsilon) \right)$ iterations to obtain accuracy $\Exp[\|x^{k}-x^*\|^2 ]\leq   \epsilon 4 \lambda \|x^0-x^*\|^2_{(\bA^\top \bA)^\dagger} $.

\begin{thm}[\cite{ gower2018accelerated}]
\label{theorem2}
Let $\bW=\frac{\bA^\top \bA}{m}$ and let assume exactness\footnote{Note that in this setting $\bB=\bI$, which means that $\bW=\Exp[\bZ]$, and the exactness assumption takes the form ${\rm Null}(\mW) = {\rm Null}(\mA)$.}. Let $\{x^k,y^k,v^k\}$ be the iterates  of Algorithm~\ref{alg:AccKaczmarz} with the Option 2 for the parameters.  Then 
$$\Psi^k \leq \left(1-\sqrt{\lambda_{\min}^+(\bW)/\nu}\right)^k \Psi^0,$$
where $\Psi^k =\Exp \left[\| v^k - x^*\|_{\mW^\dagger}^2+  \frac{1}{\mu  } \|x^{k}-x^*\|^2 \right]$.
\end{thm}

The above result implies that Algorithm~\ref{alg:AccKaczmarz} converges linearly with rate $1-\sqrt{\lambda_{\min}^+(\bW)/\nu}$, which translates to a total of $O\left(\sqrt{\nu/\lambda_{\min}^+(\bW)}\log(1/\epsilon)\right)$ iterations to bring the quantity $\Psi^k$ below $\epsilon>0$. It can be shown that $1 \leq \nu \leq 1/\lambda_{\min}^+(\bW)$, (Lemma 2 in \cite{gower2018accelerated}) where $\nu$ is as defined in \eqref{thenu}. Thus,
$\sqrt{\frac{1}{\lambda_{\min}^+(\bW)}} \leq \sqrt{\frac{\nu}{\lambda_{\min}^+(\bW)}} \leq \frac{1}{\lambda_{\min}^+(\bW)},$
which means that the rate of AccRK (Option 2) is always better than that of the RK with unit stepsize which is equal to $O\left(\frac{1}{\lambda_{\min}^+(\bW)} \log(1/\epsilon)\right)$ (see Theorem~\ref{ConvergenceSketchProject}). 

In \cite{ gower2018accelerated}, Theorem~\ref{theorem2} has been proposed for solving more general consistent linear systems (the matrix $\bA$ of the system is not assumed to be normalized). In this case $\bW=\Exp[\bZ]$ and the parameter $\nu$ is slightly more complicated than the one of equation \eqref{thenu}. We refer the interested reader to \cite{ gower2018accelerated} for more details.

\paragraph{Comparison of the convergence rates:}
Before describe the distributed nature of the AccRK and explain how it can be interpreted as a gossip algorithm, let us compare the convergence rates of the two options of the parameters for the case of general normalized consistent linear systems ($\|\bA_{i:}\|=1$ for any $i \in [m]$).

Using Theorems~\ref{causjoakls} and \ref{theorem2}, it is clear that the iteration complexity of AccRK is 
\begin{equation}
\label{IterCompleOption1}
O \left(\frac{m}{\sqrt{\lambda}} \log(1/\epsilon) \right)\overset{\lambda=\lambda_{\min}^+(\bA^\top \bA)}{=} O \left( \frac{m}{\sqrt{\lambda_{\min}^+(\bA^\top \bA)}} \log(1/\epsilon) \right),
\end{equation} and 
\begin{equation}
\label{IterCompleOption2}
O\left(\sqrt{\frac{\nu m}{\lambda_{\min}^+(\bA^\top \bA)}}\log(1/\epsilon)\right),
\end{equation} for the Option 1 and Option 2 for the parameters, respectively.

In the following derivation we compare the iteration complexity of the two methods.

\begin{lem}
\label{naoiskma}
Let matrices $\bC \in \R^{n \times n}$  and $\bC_i \in \R^{n \times n}$ where $i \in [m]$ be positive semidefinite, and satisfying $\sum_{i=1}^m \bC_i =\bC$. Then 
$$\sum_{i=1}^m \bC_i  \bC^\dagger \bC_i \preceq \bC.$$
\end{lem}

\begin{proof}
From the definition of the matrices it holds that $\bC_i \preceq \bC$ for any $i \in [m]$. Using the properties of Moore-Penrose pseudoinverse,  this implies that
\begin{equation}
\label{cnaosklasdpoad}
\bC_i^\dagger \succeq \bC^\dagger.
\end{equation}
Therefore
\begin{equation}
\label{cbaiusjkalks}
\bC_i= \bC_i \bC_i^\dagger \bC_i \overset{\eqref{cnaosklasdpoad}}{\succeq} \bC_i \bC^\dagger \bC_i.
\end{equation} 
From the definition of the matrices by taking the sum over all $i \in [m]$ we obtain:
$$\bC= \sum_{i=1}^m \bC_i \overset{\eqref{cbaiusjkalks}}{\succeq} \sum_{i=1}^m \bC_i \bC^\dagger \bC_i,$$
which completes the proof.
\end{proof}

Let us now choose $\bC_i = \mA_{i:}^\top \mA_{i:}$ and $\bC=\mA^\top \mA$. Note that from their definition the matrices are positive semidefinite and satisfy $\sum_{i=1}^m \mA_{i:}^\top \mA_{i:}= \mA^\top \mA$.
Using Lemma~\ref{naoiskma} it is clear that:
$$\sum_{i=1}^m  \mA_{i:}^\top \mA_{i:} (\mA^\top \mA)^\dagger \mA_{i:}^\top \mA_{i:} \preceq \bA^\top \bA,$$
or in other words, for any vector $v \notin {\rm Null}{(\bA)}$ we set the inequality
$$ \frac{v^\top \left[\sum_{i=1}^m  \mA_{i:}^\top \mA_{i:} (\mA^\top \mA)^\dagger \mA_{i:}^\top \mA_{i:}\right] v}{v^\top [\bA^\top \bA] v} \leq 1.$$
Multiplying both sides by $m$, we set:
$$ \frac{v^\top \left[\sum_{i=1}^m  \mA_{i:}^\top \mA_{i:} (\mA^\top \mA)^\dagger \mA_{i:}^\top \mA_{i:}\right] v}{v^\top [\frac{\bA^\top \bA}{m}] v} \leq m.$$

Using the above derivation, it is clear from the definition of the parameter $\nu$ \eqref{thenu}, that $\nu \leq m.$
By combining our finding with the bounds already obtained in  \cite{ gower2018accelerated} for the parameter $\nu$, we have that:
\begin{equation}
\label{acsnklasda}
1 \leq \nu \leq \min \left\{ m, \frac{1}{\lambda_{\min}^+(\bW)} \right\}.
\end{equation}
Thus, by comparing the two iteration complexities of equations \eqref{IterCompleOption1} and \eqref{IterCompleOption2} it is clear that Option 2 for the parameters \cite{ gower2018accelerated} is always faster in theory than Option 1 \cite{liu2016accelerated}. To the best of our knowledge, such comparison of the two choices of the parameters for the AccRK was never presented before.

\subsubsection{Accelerated randomized gossip algorithms}
\label{sec:AccGossip}
Having presented the complexity analysis guarantees of AccRK for solving consistent linear systems with normalized matrices, let us now explain how the two options of AccRK behave as gossip algorithms when they are used to solve the linear system $\bA x=0$ where $\bA \in \R^{|\cE| \times n}$ is the normalized incidence matrix of the network. That is, each row $e=(i,j)$ of $\bA$ can be represented as $ (\bA_{e:})^\top= \frac{1}{\sqrt{2}}(e_i -e_j)$ where $e_i$ (resp.$e_j$) is the $i^{th}$ (resp. $j^{th}$) unit coordinate vector in $\R^{n}$.

By using this particular linear system, the expression $\frac{\bA_{i :} y^k -b_{i}}{\|\bA_{i :}\|^2} \bA_{i :}^ \top$ that appears in steps 7 and 8 of AccRK takes the following form when the row $e=(i,j) \in \cE$ is sampled:
$$\frac{\bA_{e :} y^k -b_{i}}{\|\bA_{e :}\|^2} \bA_{e :}^ \top \overset{b=0}{=} \frac{\bA_{e :} y^k}{\|\bA_{e :}\|^2} \bA_{e :}^ \top \overset{\text{form of A}}{=} \frac{ y_i^k - y_j^k}{2}(e_i-e_j).$$

Recall that with $\bL$ we denote the Laplacian matrix of the network. For solving the above AC system (see Definition~\ref{defACsystem}), the standard RK requires $O\left( \left(\frac{2m}{\lambda_{\min}^+(\bL)}\right)\log(1/\epsilon)\right)$ iterations to achieve expected accuracy $\epsilon>0$. To understand the acceleration in the gossip framework this should be compared to the $$O\left(m \sqrt{\frac{2}{\lambda_{\min}^+(\bL)}} \log(1/\epsilon)\right)$$ of AccRK (Option 1) and the $$O\left(\sqrt{\frac{2m\nu}{\lambda_{\min}^+(\bL)} } \log(1/\epsilon)\right)$$ of AccRK (Option 2).

Algorithm~\ref{alg:acceleratedNew} describes in a single framework how the two variants of AccRK of Section~\ref{AcceleratedVariants} behave as gossip algorithms when are used to solve the above linear system. Note that each node $\ell \in \cV$ of the network has two local registers to save the quantities $v^k_\ell$ and $x^k_\ell$. In each step using these two values every node $\ell \in \cV$ of the network (activated or not) computes the quantity $y^k_\ell =\alpha_k v^k_\ell + (1-\alpha_k) x^k_\ell$. Then in the $k^{th}$ iteration the activated nodes $i$ and $j$ of the randomly selected edge $e=(i,j)$ exchange their values $y^k_i$ and $y^k_j$ and update the values of $x^k_i$,  $x^k_j$ and $v^k_i$, $v^k_j$ as shown in Algorithm~\ref{alg:acceleratedNew}. The rest of the nodes use only their own $y^k_\ell$ to update the values of $v^k_i$ and $x^k_i$ without communicate with any other node.

The parameter $\lambda^+_{\min}(\bL)$ can be estimated by all nodes in a decentralized manner using the method described in~\cite{charalambous2016distributed}. In order to implement this algorithm, we assume that all nodes have synchronized clocks and that they know the rate at which gossip updates are performed, so that inactive nodes also update their local values. This may not be feasible in all applications, but when it is possible (e.g., if nodes are equipped with inexpensive GPS receivers, or have reliable clocks) then they can benefit from the significant speedup achieved.

\begin{algorithm}[!h]
\begin{algorithmic}[1]
\State \textbf{Data:} Matrix $\mA\in \R^{m\times n}$ (normalized incidence matrix); vector $b=0\in \R^m$
\State Choose $x^0\in \R^n$ and set $v^0 = x^0$
\State\textbf{ Parameters:} 
Evaluate the sequences of the scalars $\alpha_k, \beta_k ,\gamma_k$ following one of two possible options.
\For {$k = 0, 1, 2, \dots, K$}
\State Each node $\ell \in \cV$ evaluate $y^k_\ell = \alpha_k v^k_\ell + (1-\alpha_k) x^k_\ell$.
\State Pick an edge $e=(i,j)$ uniformly at random.
\State Then the nodes update their values as follows:
\begin{itemize}
\item The selected node $i$ and node $j$: 
$$x^{k+1}_i = x^{k+1}_j  = (y^k_i+y^k_j)/2$$ 
$$v^{k+1}_i = \beta_k v^k_i + (1-\beta_k) y^k_i -\gamma_k (y^k_i-y^k_j)/2$$
$$v^{k+1}_j = \beta_k v^k_j + (1-\beta_k) y^k_j -\gamma_k (y^k_j-y^k_i)/2$$
\item Any other node $\ell \in \cV$:
$$x^{k+1}_\ell = y^k_\ell \quad,\quad v^{k+1}_\ell = \beta_k v^k_\ell + (1-\beta_k) y^k_\ell$$
\end{itemize}
\EndFor
\end{algorithmic}
\caption{Accelerated Randomized Gossip Algorithm (AccGossip)}
\label{alg:acceleratedNew}
\end{algorithm}

\section{Dual Randomized Gossip Algorithms}
\label{DualBlock}
An important tool in optimization literature is duality. In our setting, instead of solving the original minimization problem (primal problem) one may try to develop dual in nature methods that have as a goal to directly solve the dual maximization problem.  Then the primal solution can be recovered through the use of optimality conditions and the development of an affine mapping between the two spaces (primal and dual).

Like in most optimization problems (see \cite{necoara2014rate, SDCA, qu2015quartz, duchi2012dual}), dual methods have been also developed in the literature of randomized algorithms for solving linear systems \cite{wright2015coordinate,gower2015stochastic,loizou2017momentum}. In this section, using existing dual methods and the connection already established between the two areas of research (methods for linear systems and gossip algorithms), we present a different viewpoint that allows the development of novel dual randomized gossip algorithms. 

Without loss of generality we focus on the case of $\bB=\bI$ (no weighted average consensus). For simplicity, we formulate the AC system as the one with the incidence matrix of the network ($\bA=\bQ$) and focus on presenting the distributed nature of dual randomized gossip algorithms with no momentum. While we focus only on no-momentum protocols, we note that accelerated variants of the dual methods could be easily obtained using tools from Section~\ref{AccelerateGossip}.

\subsection{Dual problem and SDSA}
The Lagrangian dual of the best approximation problem \eqref{best approximation} is the (bounded) unconstrained concave quadratic maximization problem \cite{gower2015stochastic}:
\begin{equation}\label{eq:Dual}
\max_{y\in \R^m} D(y) \eqdef (b-\bA x^0)^\top y - \frac{1}{2}\|\bA^\top y\|_{\bB^{-1}}^2.
\end{equation}

A direct method for solving the dual problem -- Stochastic Dual Subspace Accent (SDSA)-- was first proposed in \cite{gower2015stochastic}. SDSA is a randomized iterative algorithm for solving \eqref{eq:Dual}, which updates the dual vectors $y^k$ as follows: 
$$y^{k+1} = y^k + \omega \bS_k \lambda^k,$$
where $\bS_k$ is a matrix chosen in an i.i.d.\ fashion throughout the iterative process from an arbitrary but fixed distribution (which is a parameter of the method)  and $\lambda^k$ is a vector chosen {\em afterwards} so that $D(y^k + \bS_k \lambda)$ is maximized in $\lambda$. In particular, SDSA proceeds by moving in random subspaces spanned by the random columns of $\mS_k$.

In general, the maximizer in $\lambda$ is not unique. In SDSA, $\lambda^k$ is chosen to be the least-norm maximizer, which leads to the iteration
\begin{equation}
\label{alg:dual}
 y^{k+1}= y^k + \omega \bS_k \left( \bS_k^\top \bA \bB^{-1} \bA^\top  \bS_k \right)^\dagger \bS_k^\top \left( b-\bA(x^0+ \bB^{-1} \bA^\top y^k) \right).
\end{equation}

It can be shown that the iterates $\{x^k\}_{k\geq0}$ of the sketch and project method (Algorithm~\ref{FullSkecth}) can be arised as affine images of the iterates $\{y^k\}_{k\geq0}$ of the dual method \eqref{alg:dual} as follows~\cite{gower2015stochastic,loizou2017momentum}: 
\begin{equation}
\label{eq:dual-corresp}
x^{k} = x(y^k) =  x^0 + \bB^{-1} \mA^\top y^k.
\end{equation}

In \cite{gower2015stochastic} the dual method was analyzed for the case of unit stepsize ($\omega=1$). Later in \cite{loizou2017momentum} the analysis extended to capture the cases of $\omega \in (0,2)$. 

It can be shown,\cite{gower2015stochastic,loizou2017momentum}, that if $\bS_k$ is chosen randomly from the set of unit coordinate/basis vectors in $\R^m$,  then the dual method \eqref{alg:dual} is randomized coordinate descent \cite{leventhal2010randomized,  richtarik2014iteration}, and the corresponding primal method is RK \eqref{RK}.  More generally, if $\bS_k$ is a random column submatrix of the $m \times m$ identity matrix, the dual method is the randomized Newton method \cite{qu2015sdna}, and the corresponding primal method is RBK \eqref{RBK}. In Section~\ref{RNMdual} we shall describe the more general block case in more detail.

The basic convergence result for the dual iterative process \eqref{alg:dual} is presented in the following theorem.  We set $y^0=0$ so that $x^0 = c$ and the affine mapping of equation \eqref{eq:dual-corresp} is satisfied. Note that, SDSA for solving the dual problem \eqref{eq:Dual} and sketch and project method for solving the best approximation problem \eqref{best approximation} (primal problem) converge exactly with the same convergence rate. 

 \begin{thm}[Complexity of SDSA \cite{gower2015stochastic, loizou2017momentum}]
 \label{SDSATheorem}
Assume exactness and let $y^0=0$.  Let $\{y^k\}$ be the iterates produced by SDSA (equation \eqref{alg:dual}) with step-size $\omega \in (0,2)$ . Then,
\begin{equation}
\Exp[D(y^*)-D(y^k)] \leq \rho^k \left[D(y^*)-D(y^0)\right],
\end{equation}
where $\rho \eqdef 1 - \omega (2-\omega) \lambda_{\min}^+ \left( \bW \right) \in (0,1)$.
 \end{thm}

\subsection{Randomized Newton method as a dual gossip algorithm}
\label{RNMdual}
In this subsection we bring a new insight into the randomized gossip framework by presenting how the dual iterative process that is associated to RBK method solves the AC problem with $\bA=\bQ$ (incidence matrix). Recall that the right hand side of the linear system is $b=0$. For simplicity, we focus on the case of $\bB=\bI$ and $\omega=1$.

Under this setting ($\bA=\bQ$, $\bB=\bI$ and $\omega=1$) the dual iterative process \eqref{alg:dual} takes the form:
\begin{equation}
\label{SDSA}
 y^{k+1}= y^k - \bI_{C:}(\bI_{C:}^\top \bQ \bQ^\top \bI_{C:})^\dagger \bQ(x^0+\bQ^\top y^k),
\end{equation}
and from Theorem~\ref{SDSATheorem} converges to a solution of the dual problem as follows:
$$\Exp \left[D(y^*)-D(y^k) \right] \leq \left[1 - \lambda_{\min}^+ \left( \Exp\left[\bQ_{C:}^\top (\bQ_{C:}\bQ_{C:}^\top)^\dagger \bQ_{C:}\right] \right) \right]^k \left[D(y^*)-D(y^0)\right].$$
Note that the convergence rate is exactly the same with the rate of the RBK under the same assumptions (see \eqref{anisojxalksda}).

This algorithm is a randomized variant of the Newton method applied to the problem of maximizing the quadratic function $D(y)$ defined in \eqref{eq:Dual}. Indeed, in each iteration we perform the update $y^{k+1} = y^k +\bI_{C:} \lambda^k$, where $\lambda^k$ is chosen greedily so that $D(y^{k+1})$ is maximized. In doing so, we invert a random principal submatrix of the Hessian of $D$, whence the name.  

 \textit{Randomized Newton Method} (RNM) was first proposed by Qu et al.\ \cite{qu2015sdna}. RNM was first analyzed as an algorithm for minimizing \emph{smooth strongly convex functions}. In \cite{gower2015stochastic} it was also extended to the case of a \emph{smooth but weakly convex quadratics}.  This method was not previously associated with any gossip algorithm.

The most important distinction of RNM compared to  existing gossip algorithms is that it operates with values that are associated to the {\em edges} of the network. To the best of our knowledge, it is the first {\em randomized dual gossip method}\footnote{Since our first paper \cite{LoizouRichtarik} appeared online, papers \cite{hanzely2017privacy} and \cite{hendrikx2018accelerated} propose dual randomized gossip algorithms as well. In \cite{hanzely2017privacy} the authors focus on presenting privacy preserving dual randomized gossip algorithms and \cite{hendrikx2018accelerated} solve the dual problem \eqref{eq:Dual} using accelerated coordinate descent.}. In particular, instead of iterating over values  stored at  the nodes, RNM uses these values to update ``dual weights'' $y^k \in \R^m$ that correspond to the edges $\cE$ of the network. However, deterministic dual  distributed averaging algorithms were proposed before  \cite{rabbat2005generalized, ghadimi2014admm}. Edge-based methods have also been proposed before; in particular  in \cite{wei20131} an asynchronous distributed ADMM algorithm presented for solving the more general consensus optimization problem with convex functions. 

\textbf{Natural Interpretation.} In  iteration $k$, \emph{RNM} (Algorithm~\eqref{SDSA}) executes the following steps:
1) Select a random set of edges $\cS_k \subseteq \cE$, 2)  Form a subgraph $\cG_k$ of $\cG$ from the selected edges, 3) The values of the edges in each connected component of $\cG_k$ are updated: their new values are a linear combination of the private values of the nodes belonging to the connected component and of the adjacent edges of their connected components. (see also example of Figure~\ref{figureRNM}).

\textbf{Dual Variables as Advice.} The weights $y^k$ of the edges  have a natural interpretation as \textit{advice} that each selected node receives from the network in order to update its value (to one that will eventually converge to the desired average).

Consider RNM performing the $k^{th}$ iteration and let $\cV_r$ denote the set of nodes of the selected connected component that node $i$ belongs to. Then, from Theorem~\ref{TheoremRBK} we know that $x_i^{k+1}=\sum_{i \in \cV_r} x_i^k / |\cV_r|$. Hence, by using \eqref{eq:dual-corresp}, we obtain the following identity:
\begin{equation}
\label{advice}
\textstyle
 (\bA^\top y^{k+1})_i=\frac{1}{|\cV_r|} \sum_{i \in \cV_{r}}(c_i+(\bA^\top y^{k})_i)- c_i.
\end{equation}
Thus in each step $(\bA^\top y^{k+1})_i$ represents the term (advice) that must be added to the initial value $c_i$ of node $i$ in order to update its value to the average of the values of the nodes of the connected component  $i$ belongs to.

\textbf{Importance of the dual perspective:} It was shown  in  \cite{qu2015sdna} that when RNM (and as a result, RBK, through the affine mapping \eqref{eq:dual-corresp}) is viewed as a family of methods indexed by the size $\tau = |\cS|$ (we choose $\cS$ of fixed size in the experiments),   then $\tau \to 1/(1-\rho)$, where $\rho$ is defined in \eqref{RateRho}, decreases {\em superlinearly} fast in $\tau$. That is, as $\tau$ increases by some factor, the iteration complexity  drops by a factor that is at least as large. Through preliminary numerical experiments in Section~\ref{cakslas} we experimentally show that this is true for the case of AC systems as well.
\begin{figure}[htb]
\begin{minipage}[b]{1.0\linewidth}
  \centering
  \centerline{\includegraphics[scale=0.4]{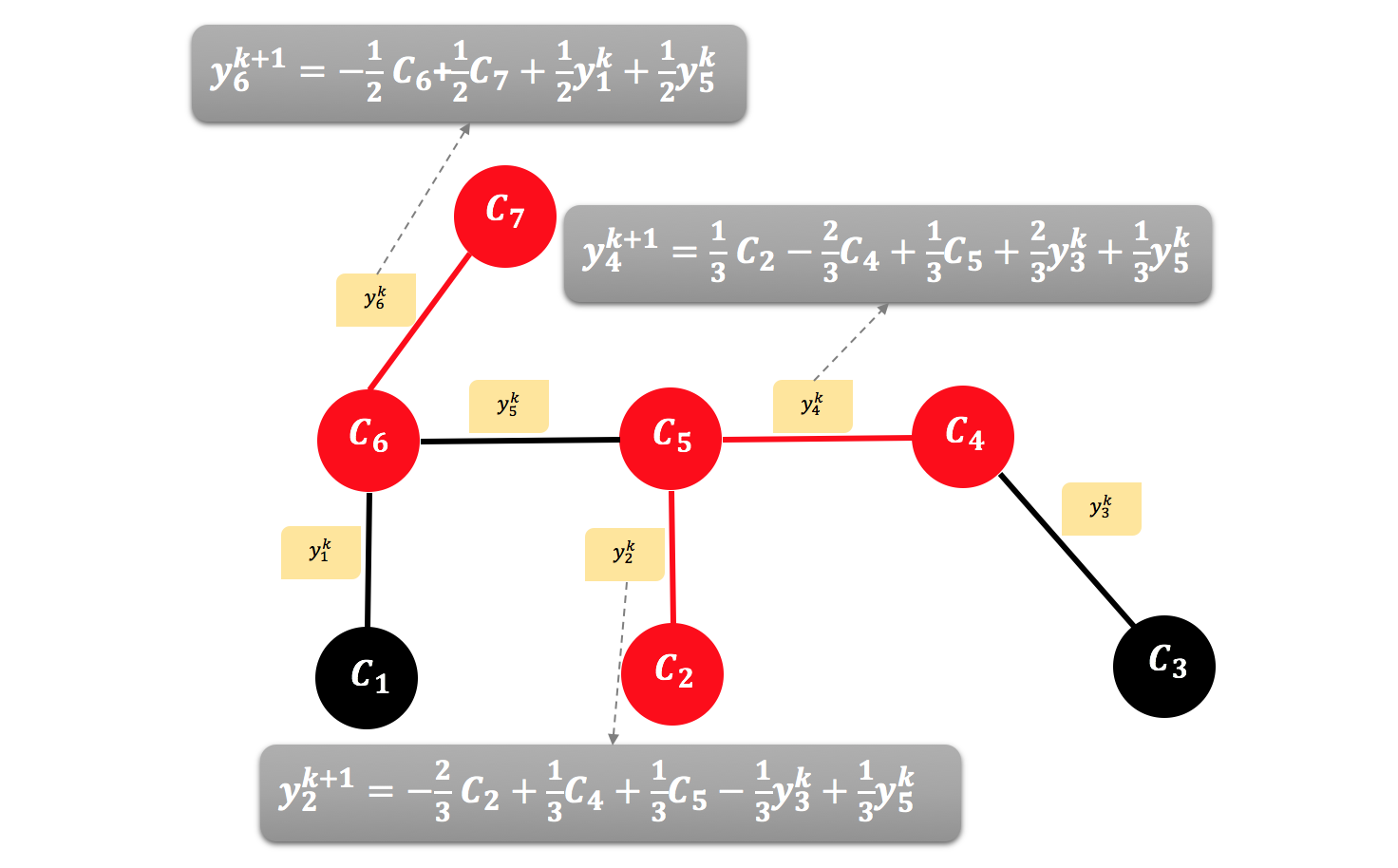}}
  \caption{\footnotesize Example of how the RNM method works as gossip algorithm. In this specific case 3 edges are selected and form a sub-graph with two connected components. Then the values at the edges update their values using the private values of the nodes belonging to their connected component and the values associate to the adjacent edges of their connected components.}
  \label{figureRNM}
\end{minipage}
\end{figure}

\section{Further Connections Between Methods for Solving Linear Systems and Gossip Algorithms}
\label{FurtherConnections}
In this section we highlight some further interesting connections between linear systems solvers and gossip protocols for average consensus:
\begin{itemize}
\item \textbf{Eavesdrop gossip as special case of Kaczmarz-Motzkin method.}
In \cite{ustebay2008greedy} greedy gossip with eavesdropping (GGE), a novel randomized gossip algorithm for distributed computation of the average consensus problem was proposed and analyzed. 
In particular it was shown that that greedy updates of GGE lead to rapid convergence. In this protocol, the greedy updates are made possible by exploiting the broadcast nature of wireless communications. During the operation of GGE, when a node decides to gossip, instead of choosing one of its neighbors at random, it makes a greedy selection, choosing the node which has the value most different from its own. In particular the method behaves as follows:

At the $k^{th}$ iteration of GGE, a node $i_k$ is chosen uniformly at random from $[n]$. Then, $i_k$ identifies a neighboring node $j_k \in \cN_i$ satisfying:
$$j_k \in \max_{j \in \cN_i} \left\{ \frac{1}{2} (x_i^k-x_j^k)^2 \right\}$$
which means that the selected node $i_k$ identifies a neighbor that currently has the most different value from its own. This choice is possible because each node $i\in \cV$  maintains not only its own local variable $x_i^k$, but also a copy of the current values at its neighbors $x_j^k$ for $j \in \cN_i$. In the case that node $i_k$ has multiple neighbors whose values are all equally (and maximally) different from  its current value, it chooses one of these neighbors at random. Then node $i_k$ and $j_k$ update their values to:
$$x_i^{k+1}=x_j^{k+1}=\frac{1}{2} (x_i^k+x_j^k).$$

In the area of randomized methods for solving large linear system there is one particular method,  the Kaczmarz-Motzkin algorithm \cite{de2017sampling, haddock2018motzkin} that can work as gossip algorithm with the same update as the GGE when is use to solve the homogeneous linear system with matrix the Incidence matrix of the network.

Update rule of Kaczmarz-Motzkin algorithm (KMA) \cite{de2017sampling, haddock2018motzkin}:
\begin{enumerate}
\item Choose sample of $d_k$ constraints, $P_k$, uniformly at random from among the rows of matrix $\bA$.
\item From among these $d_k$ constraints, choose $t_k=\text{argmax}_{i \in P_k} \bA_{i:} x^k-b_i.$
\item Update the value: $x^{k+1}=x^k - \frac{\bA_{t_k :} x^k -b_{i}}{\|\bA_{t_k :}\|^2} \bA_{t_k :}^ \top.$
\end{enumerate}

It is easy to verify that when the Kaczmarz-Motzkin algorithm is used for solving the AC system with $\bA=\bQ$ (incidence matrix) and in each step of the method the chosen constraints $d_k$ of the linear system correspond to edges attached to one node it behaves exactly like the GGE. From numerical analysis viewpoint an easy way to choose the constraints $d_k$ that are compatible to the desired edges is in each iteration to find the indexes of the non-zeros of a uniformly at random selected column (node) and then select the rows corresponding to these indexes.

Therefore, since GGE \cite{ustebay2008greedy} is a special case of the KMA (when the later applied to special AC system with Incidence matrix) it means that we can obtain the convergence rate of GGE by simply use the tight conergence analysis presented in \cite{de2017sampling, haddock2018motzkin} \footnote{Note that the convergence theorems of \cite{de2017sampling, haddock2018motzkin} use $d_k=d$. However, with a small modification in the original proof the theorem can capture the case of different $d_k$.}. In \cite{ustebay2008greedy} it was mentioned that analyzing the convergence behavior of GGE is non-trivial and not an easy task. By establishing the above connection the convergence rates of GGE can be easily obtained as special case of the theorems presented in \cite{de2017sampling}. 

In Section~\ref{AccelerateGossip} we presented provably accelerated variants of the pairwise gossip algorithm and of its block variant. Following the same approach one can easily develop accelerated variants of the GGE using the recently proposed analysis for the accelerated Kaczmarz-Motzkin algorithm presented in \cite{morshed2019accelerated}. 

\item \textbf{Inexact Sketch and Project Methods:}

Recently in \cite{loizou2019Inexact}, several inexact variants of the sketch and project method \eqref{FullSkecth} have been proposed. As we have already mentioned the sketch and project method is a two step procedure algorithm where first the sketched system is formulated and then the last iterate $x^k$ is \textit{exactly} projected into the solution set of the sketched system. In \cite{loizou2019Inexact} the authors replace the exact projection with an inexact variant and suggested to run a different algorithm (this can be the sketch and project method itself) in the sketched system to obtain an approximate solution.  It was shown that in terms of time the inexact updates can be faster than their exact variants.

In the setting of randomized gossip algorithms for the AC system with Incidence matrix ($\bA=\bQ$) , $\bB=\bI$ and $\omega =1$ a variant of the inexact sketch and project method will work as follows (similar to the update proved in Theorem~\ref{TheoremRBK}):
\begin{enumerate}
\item Select a random set of edges $C \subseteq \cE$.
\item Form subgraph $\cG_k$ of $\cG$ from the selected edges.
\item Run the pairwise gossip algorithm of \cite{boyd2006randomized} (or any variant of the sketch and project method) on the subgraph $\cG_k$ until an accuracy $\epsilon$ is achieved (reach a neighborhood of the exact average).
\end{enumerate}

\item \textbf{Non-randomized gossip algorithms as special cases of Kaczmarz methods:}

In the gossip algorithms literature there are efficient protocols that are not randomized\cite{mou2010deterministic, he2011periodic, liu2011deterministic, yu2017distributed}. Typically, in these algorithms the pairwise exchanges between nodes it happens in a deterministic, such as predefined cyclic, order. For example, $T$-periodic gossiping is a protocol which stipulates that each node must interact with each of its neighbours exactly once every $T$ time units. It was shown that under suitable connectivity assumptions of the network $\mathcal{G}$, the $T$-periodic gossip sequence will converge at a rate determined by the magnitude of the second largest eigenvalue of the stochastic matrix determined by the sequence of pairwise exchanges which occurs over a period. It has been shown that if the underlying graph is a tree, the mentioned eigenvalue is constant for all possible $T$-periodic gossip protocols.

In this work we focus only on randomized gossip protocols. However we speculate that the above non-randomized gossip algorithms would be able to express as special cases of popular non-randomized projection methods for solving linear systems   \cite{popa2017convergence, nutini2016convergence, dutight}.  Establishing connections like that is an interesting future direction of research and can possibly lead to the development of novel block and accelerated variants of many non-randomized gossip algorithms, similar to the protocols we present in Sections~\ref{skecthsection} and \ref{AccelerateGossip}.

\end{itemize}

\section{Numerical Evaluation}
\label{experiments}
In this section, we empirically validate our theoretical results and evaluate the performance of the proposed randomized gossip algorithms. The section is divided into four main parts, in each of which we highlight a different aspect of our contributions. 

In the first experiment, we numerically verify the linear convergence of the Scaled RK algorithm (see equation \eqref{scaledRK})  for solving the weighted average consensus problem presented in Section~\ref{weightedAC}. In the second part, we explain the benefit of using block variants in the gossip protocols where more than two nodes update their values in each iteration (protocols presented in Section~\ref{BlockGossip}).  In the third part, we explore the performance of the faster and provably accelerated gossip algorithms proposed in Section~\ref{AccelerateGossip}. In the last experiment, we numerically show that relaxed variants of the pairwise randomized gossip algorithm converge faster than the standard randomized pairwise gossip with unit stepsize (no relaxation). This gives a specific setting where the phenomenon of over-relaxation of iterative methods for solving linear systems is beneficial.

In the comparison of all gossip algorithms we use the relative error measure $\|x^k-x^*\|_{\bB}^2 / \|x^0-x^*\|_{\bB}^2 $ where $x^0 =c \in \R^n$ is the starting vector of the values of the nodes and matrix $\bB$ is the positive definite diagonal matrix with weights in its diagonal (recall that in the case of standard average consensus this can be simply $\bB=\bI$). Depending on the experiment, we choose the values of the starting vector $c \in \R^n$ to follow either a Gaussian distribution or uniform distribution or to be integer values such that $c_i=i \in \R$.  In the plots, the horizontal axis represents the number of iterations except in the figures of subsection~\ref{cakslas}, where the horizontal axis represents the block size.

In our implementations we use three popular graph topologies from the area of wireless sensor networks. These are the cycle (ring graph), the 2-dimension grid and the random geometric graph (RGG) with radius $r=\sqrt{\log(n)/n}$. In all experiments we formulate the average consensus problem (or its weighted variant) using the incidence matrix. That is, $\bA=\bQ$ is used as the AC system. Code was written in Julia 0.6.3.


\subsection{Convergence on weighted average consensus}
\label{Weightexperiments}
As we explained in Section~\ref{skecthsection}, the sketch and project method (Algorithm~\ref{FullSkecth}) can solve the more general weighted AC problem. In this first experiment we numerically verify the linear convergence of the Scaled RK algorithm \eqref{scaledRK} for solving this problem in the case of $\bB=\bD$.  That is, the matrix $\bB$ of the weights is the degree matrix $\bD$ of the graph ($\bB_{ii}=d_i$,  $\forall i \in [n]$). In this setting the exact update rule of the method is given in equation \eqref{ScRKwithQ}, where in order to have convergence to the weighted average the chosen nodes are required to share not only their private values but also their weight $\bB_{ii}$ (in our experiment this is equal to the degree of the node $d_i$).  In this experiment the starting vector of values $x^0=c \in \R^n$ is a Gaussian vector. The linear convergence of the algorithm is clear in Figure~\ref{scaledRKFigure}.

\begin{figure}[t]
\centering
\begin{subfigure}{.3\textwidth}
  \centering
  \includegraphics[width=1\linewidth]{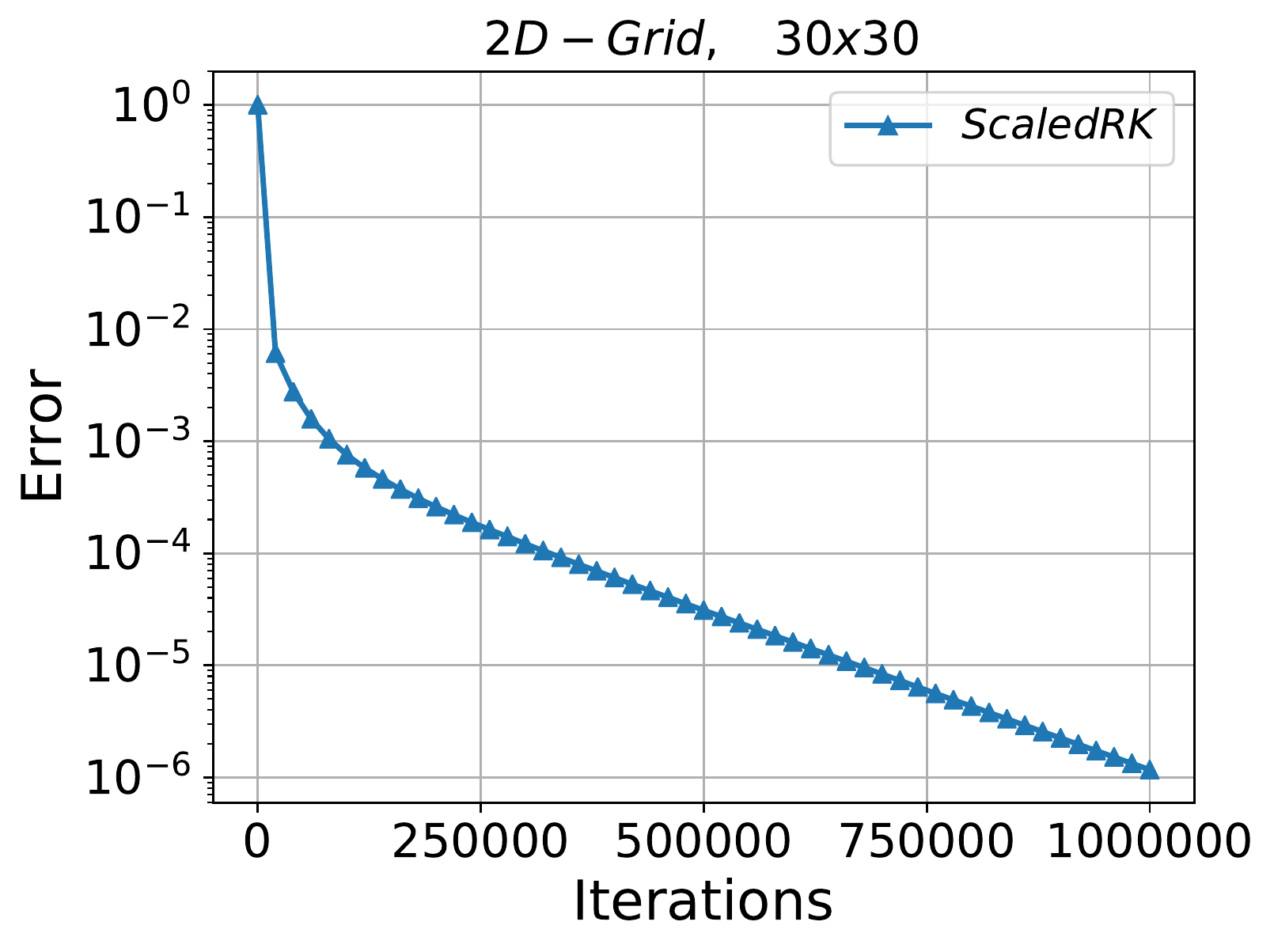}
\end{subfigure}%
\begin{subfigure}{.3\textwidth}
  \centering
  \includegraphics[width=1\linewidth]{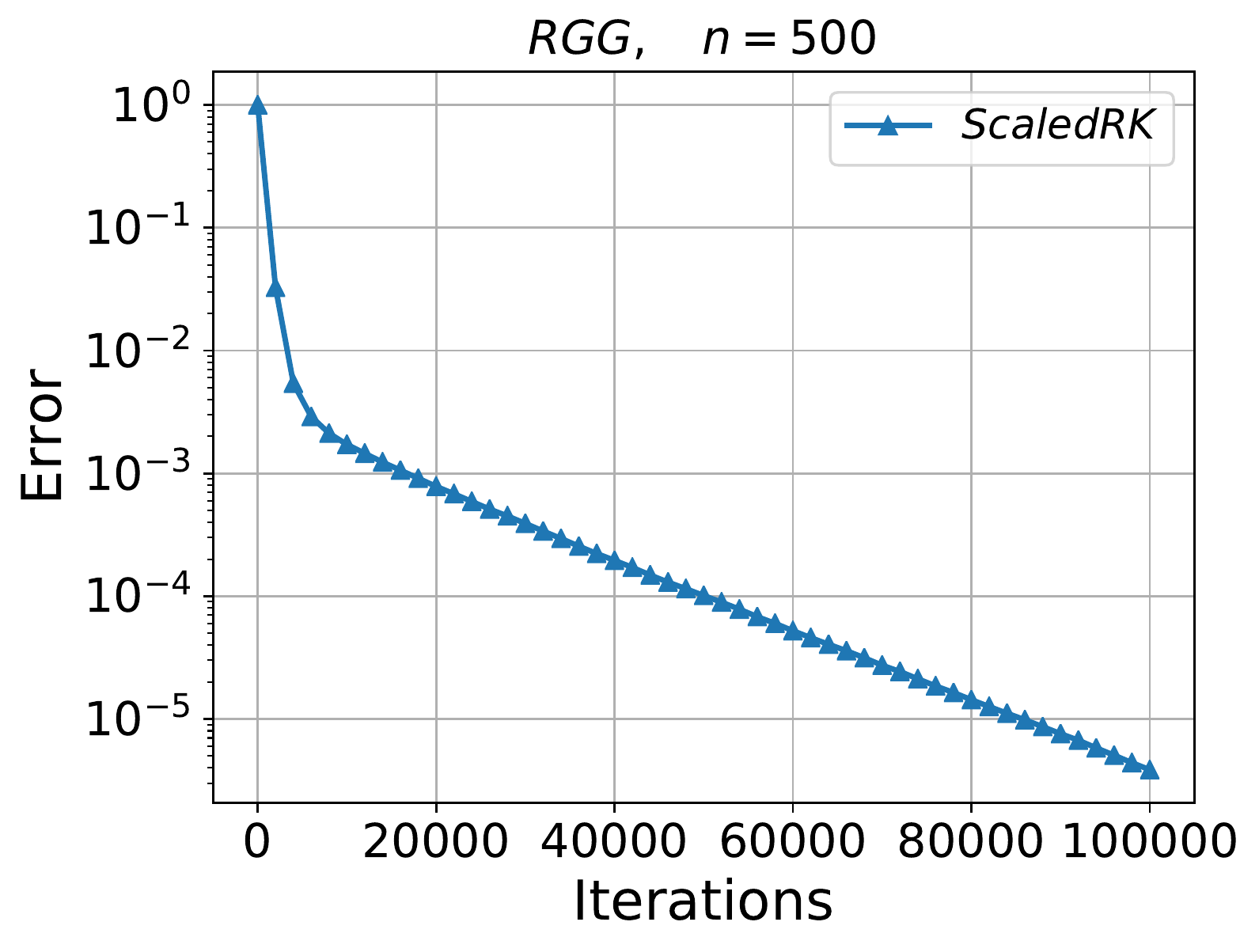}
\end{subfigure}
\begin{subfigure}{.3\textwidth}
  \centering
  \includegraphics[width=1\linewidth]{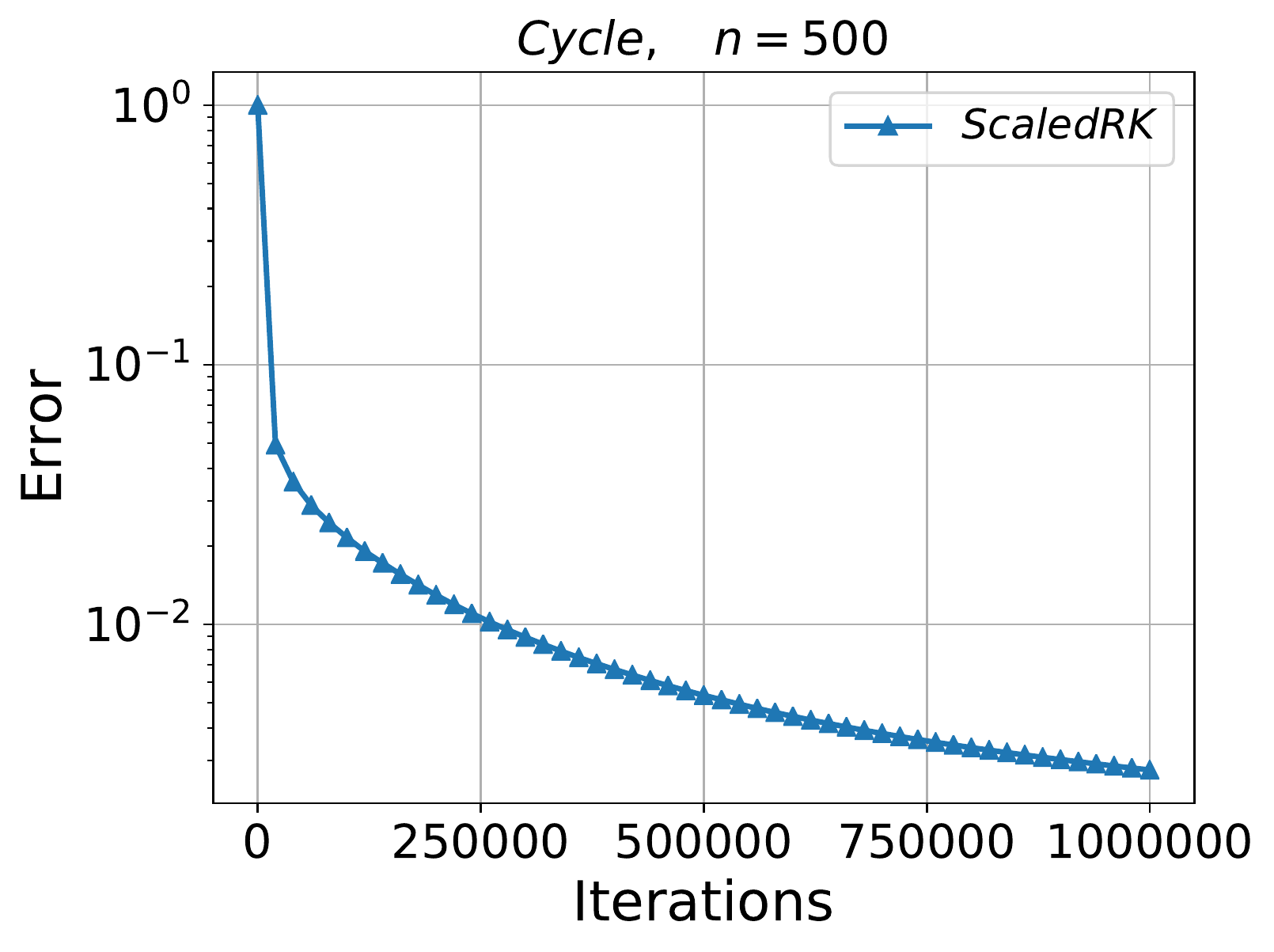}
\end{subfigure}\\
\caption{\footnotesize Performance of ScaledRK  in a 2-dimension grid, random geometric graph (RGG) and a cycle graph for solving the weighted average consensus problem. The weight matrix is chosen to be $\bB=\bD$, the degree matrix of the network. The $n$ in the title of each plot indicates the number of nodes of the network. For the grid graph this is $n \times n$.}
\label{scaledRKFigure}
\end{figure}

\subsection{Benefit of block variants}
\label{cakslas}
We devote this experiment to evaluate the performance of the randomized block gossip algorithms presented in Sections \ref{BlockGossip} and \ref{DualBlock}.  In particular, we would like to highlight the benefit of using larger block size in the update rule of randomized Kaczmarz method and as a result through our established connection of the randomized pairwise gossip algorithm \cite{boyd2006randomized} (see equation \eqref{pairwiseUpdate}).

Recall that in Section~\ref{DualBlock} we show that both RBK and RNM converge to the solution of the primal and dual problems respectively with the same rate and that their iterates are related via a simple affine transform \eqref{eq:dual-corresp}. In addition note that an interesting feature of the RNM \cite{qu2015sdna}, is that when the method viewed as algorithm indexed by the size $ \tau=|C|$, it enjoys superlinear speedup in $\tau$. That is, as $\tau$ (block size) increases by some factor, the iteration complexity drops by a factor that is at least as large (see Section~\ref{RNMdual}). Since RBK and RNM share the same rates this property naturally holds for RBK as well.

We show that for a  connected network $\cG$, the complexity improves superlinearly in $\tau = |C|$, where $C$ is chosen as a subset of $\cE$ of size $\tau$, uniformly at random (recall the in the update rule of RBK the random matrix is $\bS=\bI_{:C}$). Similar to the rest of this section in comparing the number of iterations for different values of $\tau$, we use the relative error  $\varepsilon=\|x^k-x^*\|^2 / \|x^0-x^*\|^2$. We let $x^0_i=c_i=i$ for each node $i \in \cV$ (vector of integers). We run RBK until the relative error becomes smaller than $0.01$.  The blue solid line in the figures denotes the actual number of iterations (after running the code) needed in order to achieve $\varepsilon\leq 10 ^{-2}$ for different values of $\tau$. The green dotted line represents the function $f(\tau)\eqdef \frac{\ell}{\tau}$, where $\ell$ is the number of iterations of RBK with $\tau=1$ (i.e., the pairwise gossip algorithm).  The green line depicts linear speedup;  the fact that the blue line (obtained through  experiments) is below the green line points to superlinear speedup.  In this experiment we use the Cycle graph with $n=30$ and $n=100$ nodes (Figure~\ref{fig:test3}) and  the $4 \times 4$ two dimension grid graph (Figure~\ref{fig:test4}). Note that, when $|C|=m$ the convergence rate of the  method becomes $\rho=0$ and as a result it converges in one step.

\begin{figure}[!h]
\label{RingGraph}
\centering
\begin{subfigure}{.3\textwidth}
  \centering
  \includegraphics[width=1\linewidth]{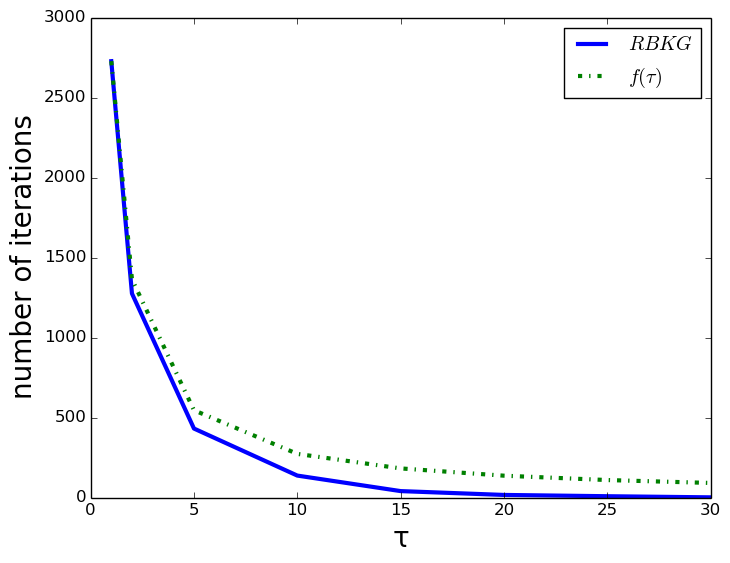}
  \caption{\footnotesize Cycle, $n=30$}
  \label{fig:sub1}
\end{subfigure}%
\begin{subfigure}{.3\textwidth}
  \centering
  \includegraphics[width=1\linewidth]{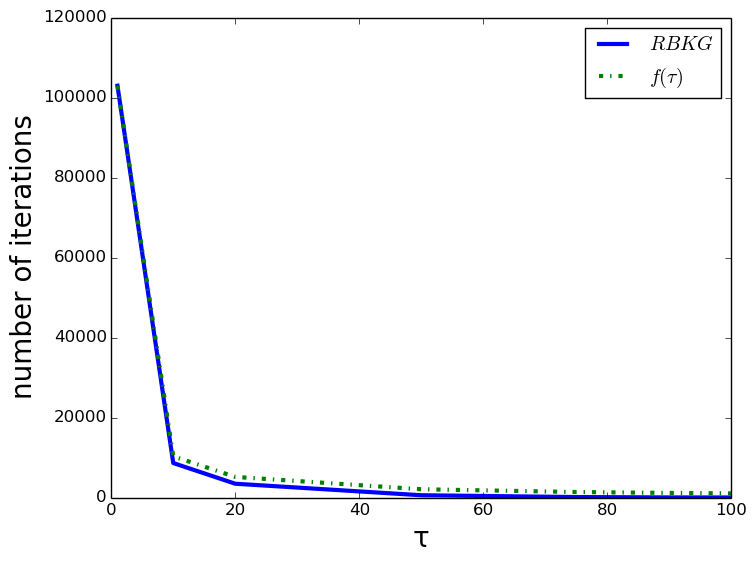}
  \caption{\footnotesize Cycle, $n=100$}
  \label{fig:sub2}
\end{subfigure}
\caption{\footnotesize Superlinear speedup of RBK on cycle graphs.}
\label{fig:test3}
\end{figure}
\begin{figure}[!h]
\label{gridGraph}
\centering
\begin{subfigure}{.3\textwidth}
  \centering
  \includegraphics[width=1\linewidth]{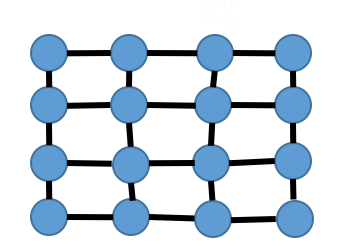}
  \caption{\footnotesize 2D-Grid, $4 \times 4$}
  \label{fig:sub3}
\end{subfigure}%
\begin{subfigure}{.3\textwidth}
  \centering
  \includegraphics[width=1\linewidth]{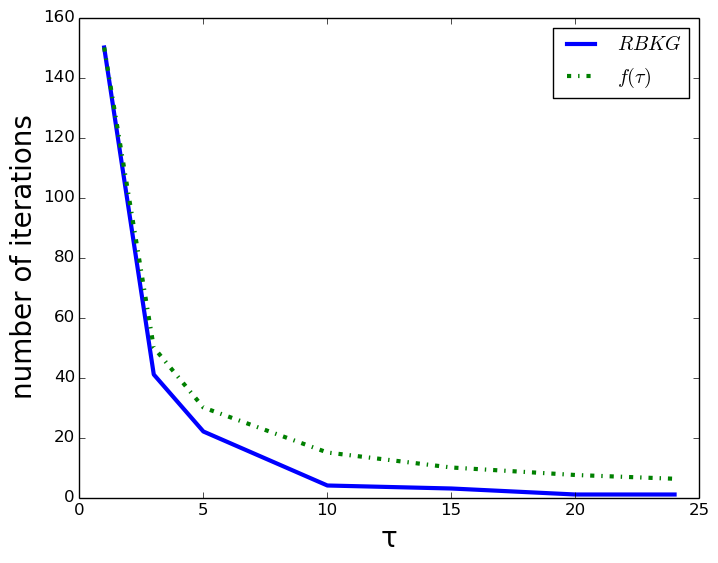}
  \caption{\footnotesize Speedup in $\tau$}
  \label{fig:sub4}
\end{subfigure}
\caption{\footnotesize Superlinear speedup of RBK on a $4 \times 4$ two dimension grid graph.}
\label{fig:test4}
\end{figure}

\subsection{Accelerated gossip algorithms}
We devote this subsection to experimentally evaluate the performance of the proposed accelerated gossip algorithms: mRK (Algorithm~\ref{RKmomentum}),  mRBK (Algorithm~\ref{RBKmomentum}) and AccGossip with the two options of the parameters (Algorithm~\ref{alg:acceleratedNew}). In particular we perform four experiments. In the first two we focus on the performance of the mRK and how the choice of stepsize (relaxation parameter) $\omega$ and heavy ball momentum parameter $\beta$ affect the performance of the method. In the next experiment we show that the addition of heavy ball momentum can be also beneficial for the performance of the block variant mRBK. In the last experiment we compare the standard pairwise gossip algorithm (baseline method) from \cite{boyd2006randomized}, the mRK and the AccGossip and show that the probably accelerated gossip algorithm, AccGossip outperforms the other algorithms and converge as predicted from the theory with an accelerated linear rate.

\subsubsection{Impact of momentum parameter on mRK}
As we have already presented in the standard pairwise gossip algorithm (equation \eqref{pairwiseUpdate}) the two selected nodes that exchange information update their values to their exact average while all the other nodes remain idle. In our framework this update can be cast as special case of mRK when $\beta=0$ and $\omega=1$.

In this experiment we keep the stepsize fixed and equal to $\omega=1$ which means that the pair of the chosen nodes update their values to their exact average and we show that by choosing a suitable momentum parameter $\beta \in (0,1)$ we can obtain faster convergence to the consensus for all networks under study. The momentum parameter $\beta$ is chosen following the suggestions made in \cite{loizou2017momentum} for solving general consistent linear systems. See Figure~\ref{mRKomega1} for more details. It is worth to point out that for all networks under study the addition of a heavy ball momentum term is beneficial in the performance of the method.

\begin{figure}[t]
\centering
\begin{subfigure}{.3\textwidth}
  \centering
  \includegraphics[width=1\linewidth]{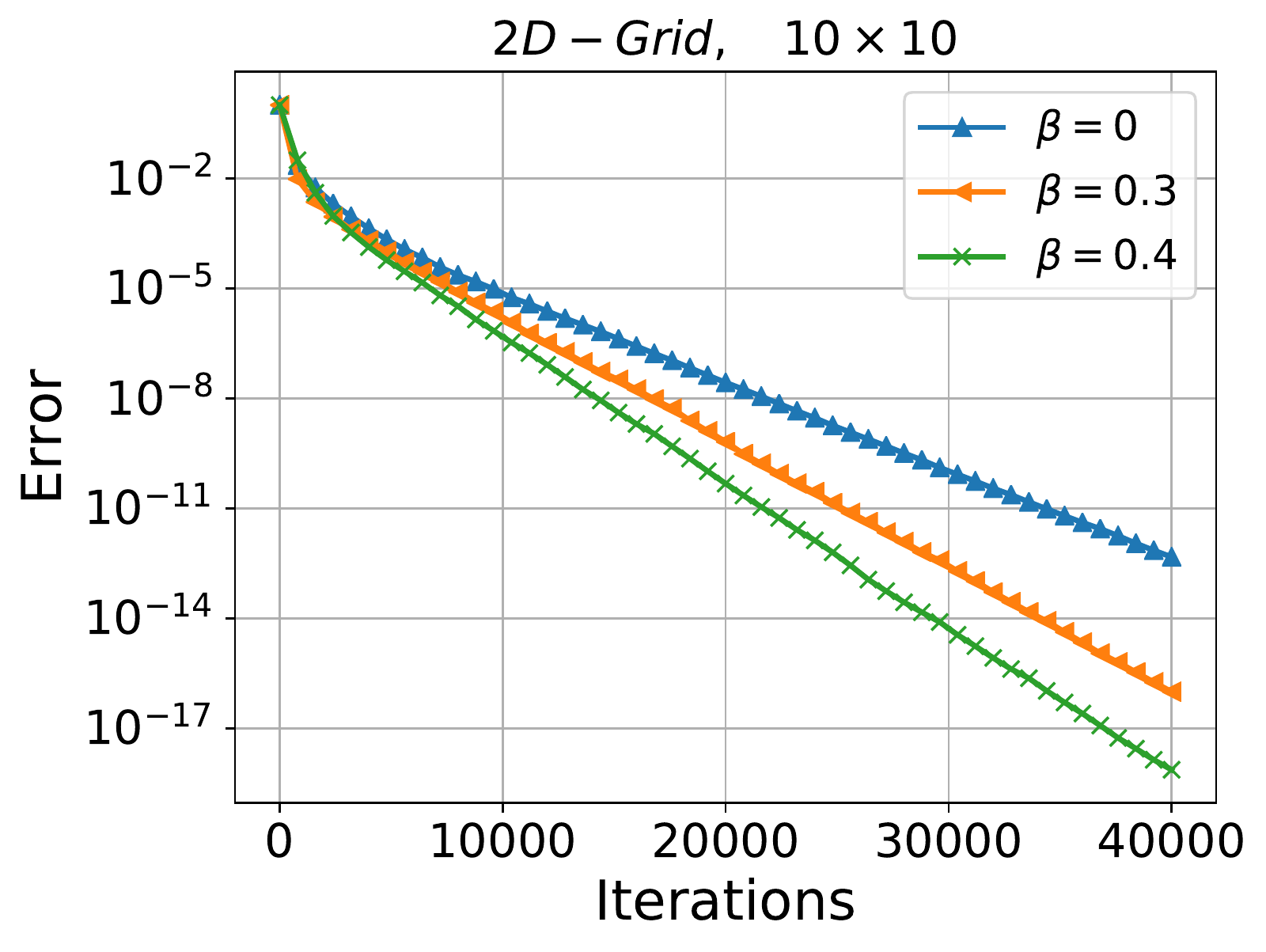}
\end{subfigure}%
\begin{subfigure}{.3\textwidth}
  \centering
  \includegraphics[width=1\linewidth]{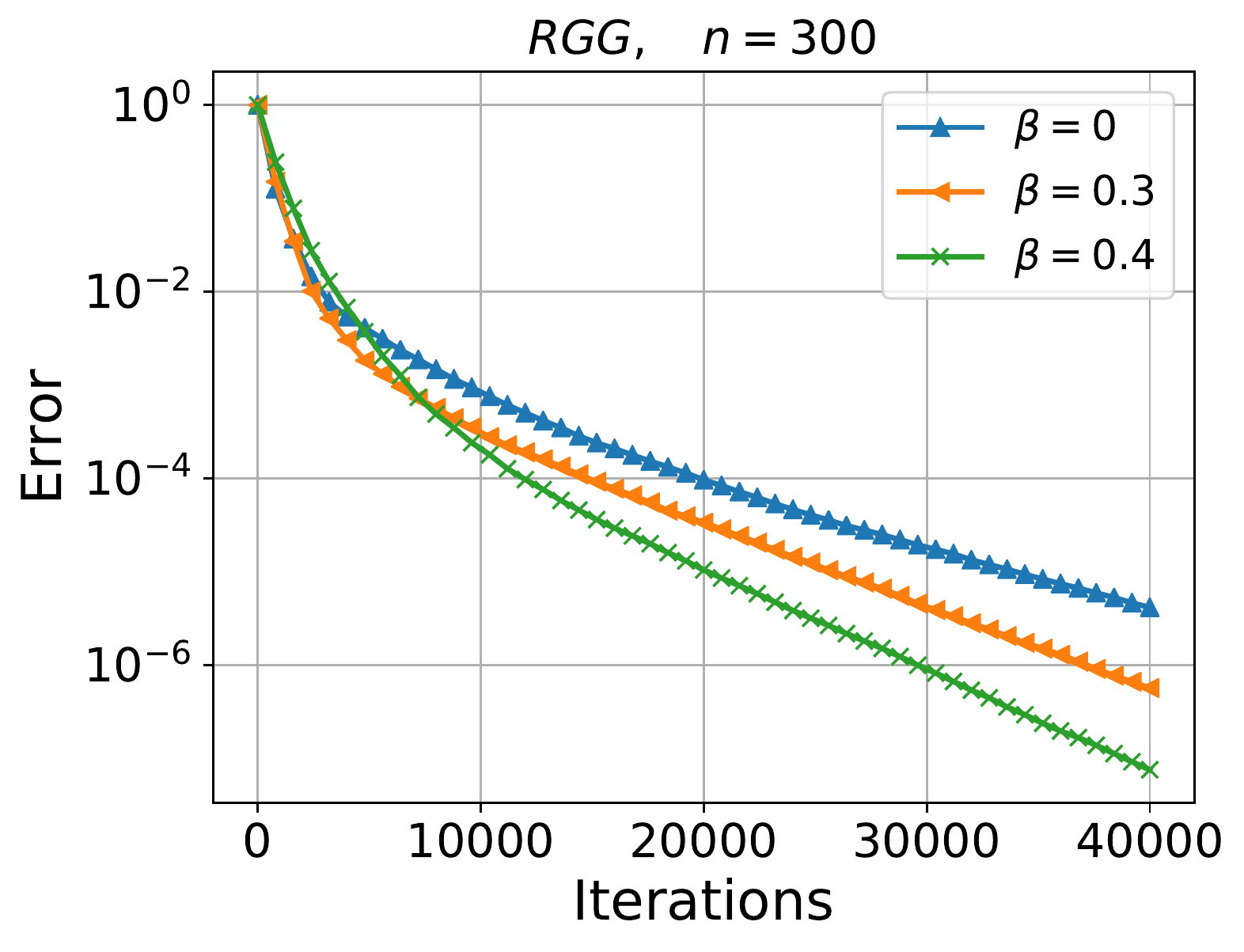}
\end{subfigure}%
\begin{subfigure}{.3\textwidth}
  \centering
  \includegraphics[width=1\linewidth]{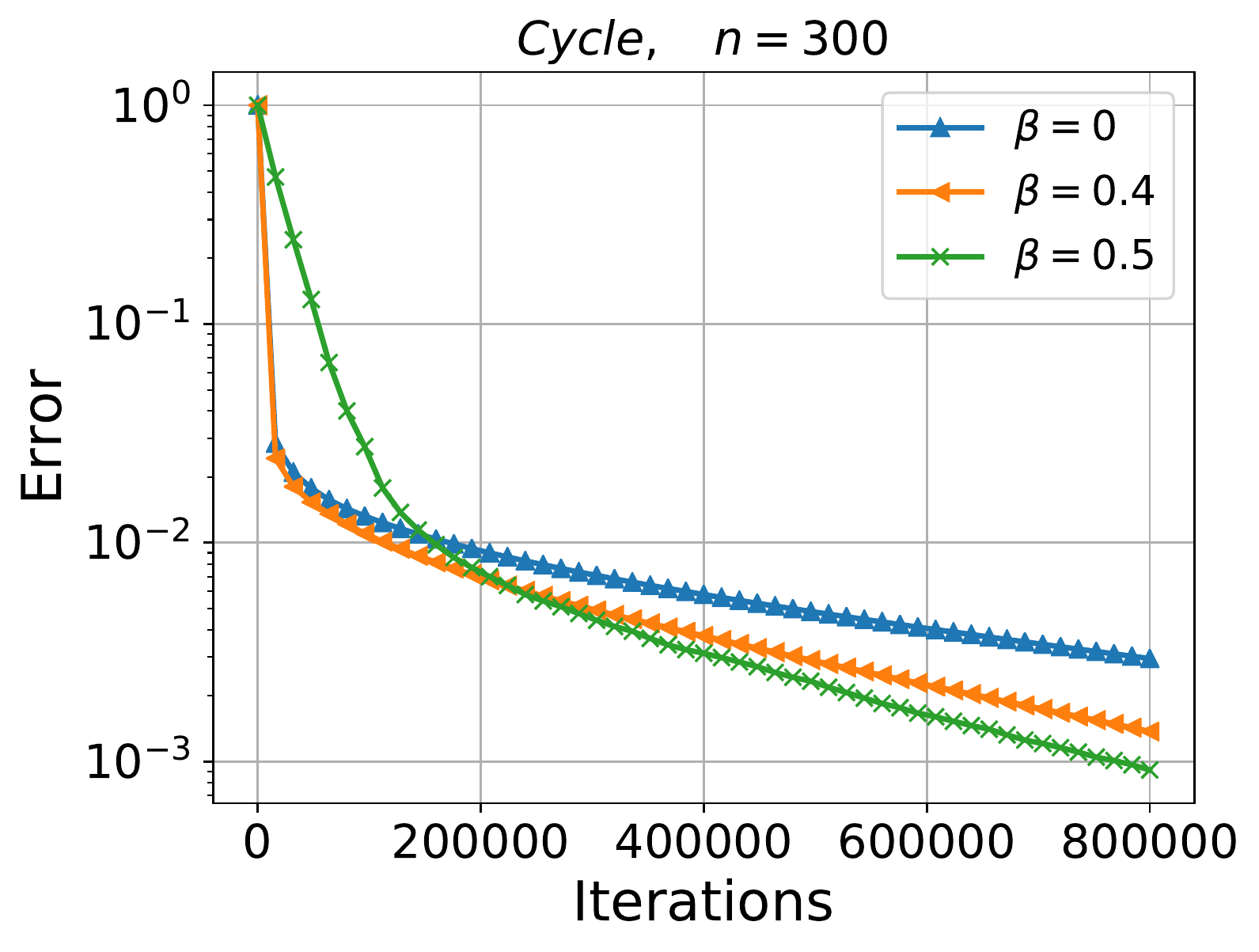}
\end{subfigure}\\
\begin{subfigure}{.3\textwidth}
  \centering
  \includegraphics[width=1\linewidth]{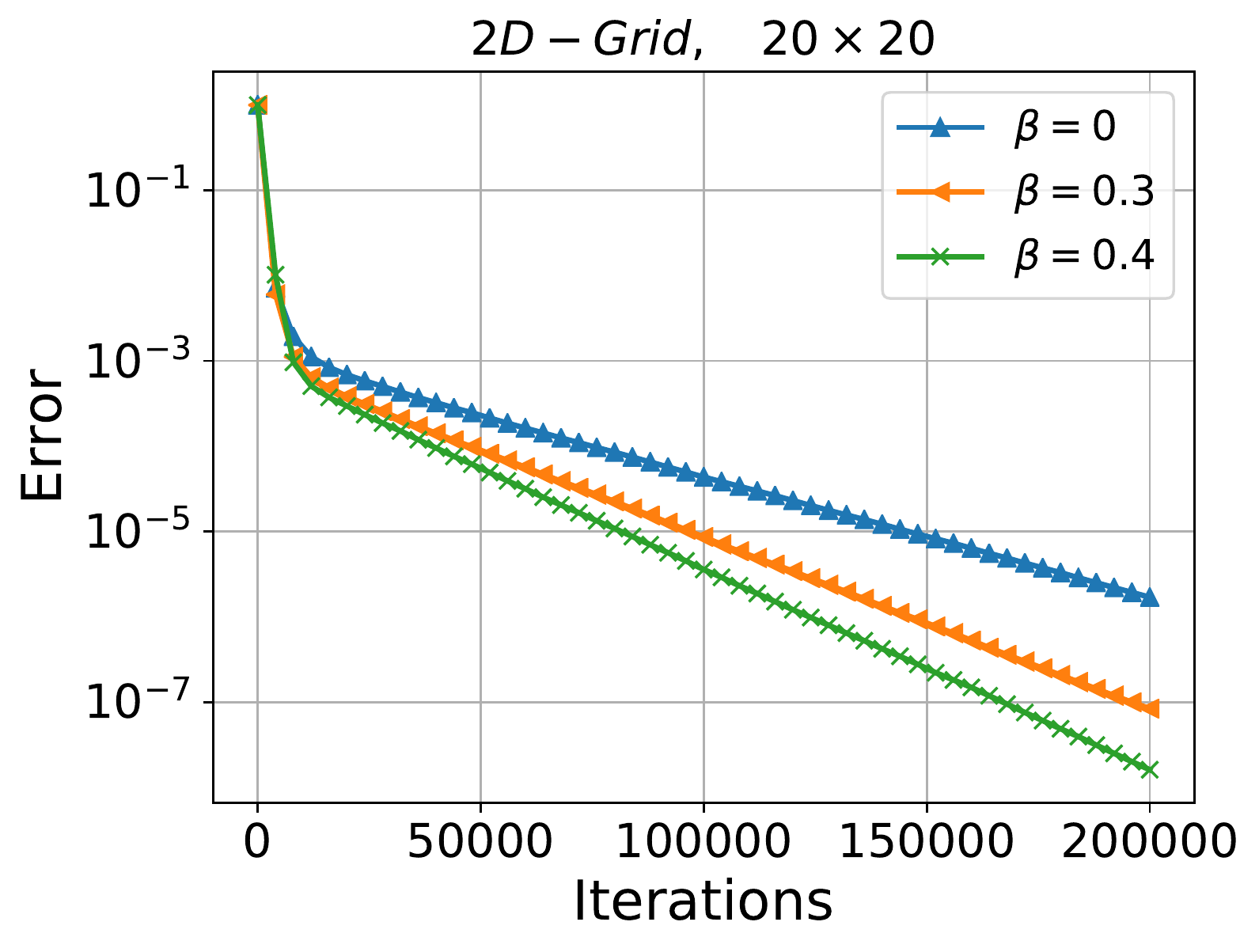}
\end{subfigure}
\begin{subfigure}{.3\textwidth}
  \centering
  \includegraphics[width=1\linewidth]{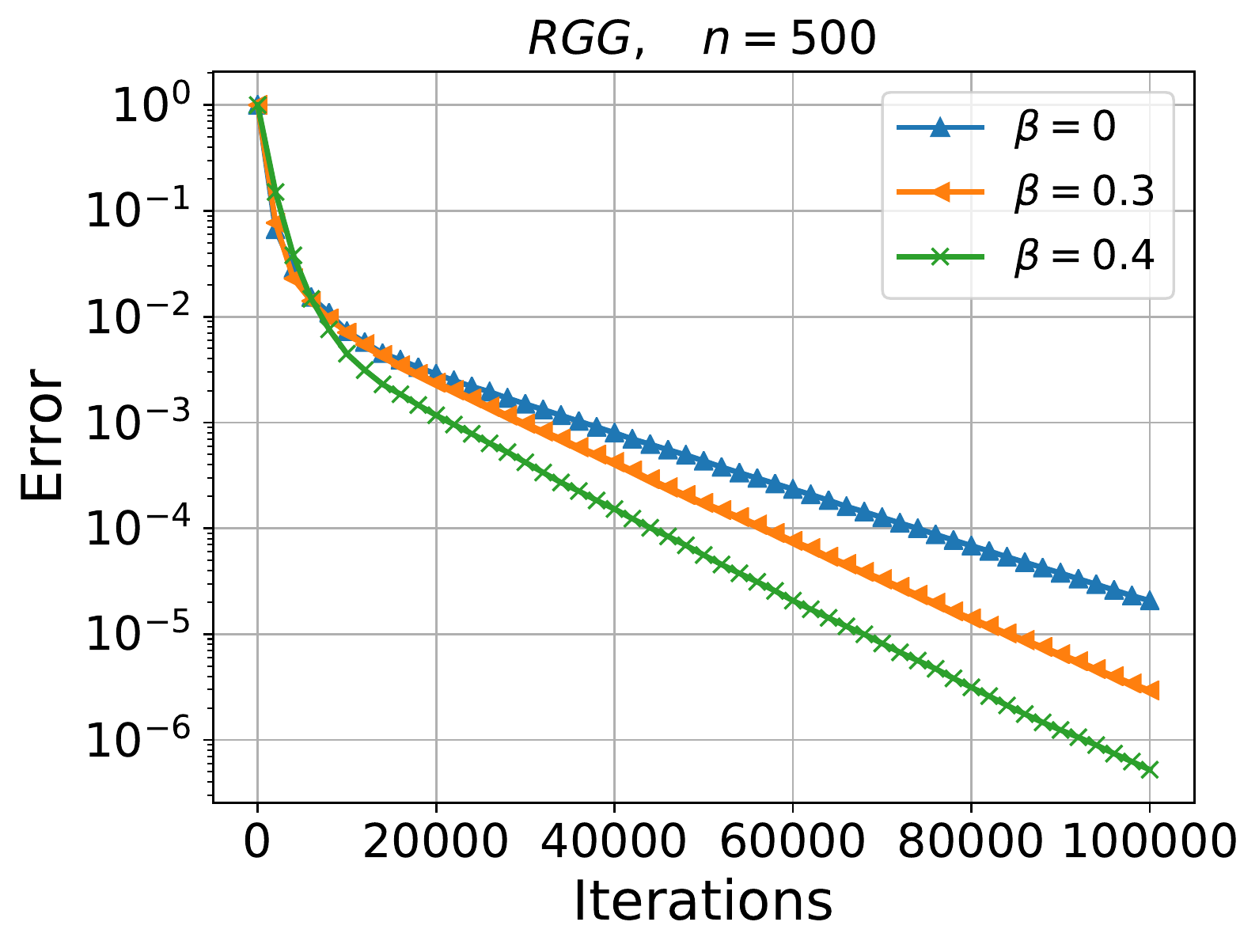}
\end{subfigure}
\begin{subfigure}{.3\textwidth}
  \centering
  \includegraphics[width=1\linewidth]{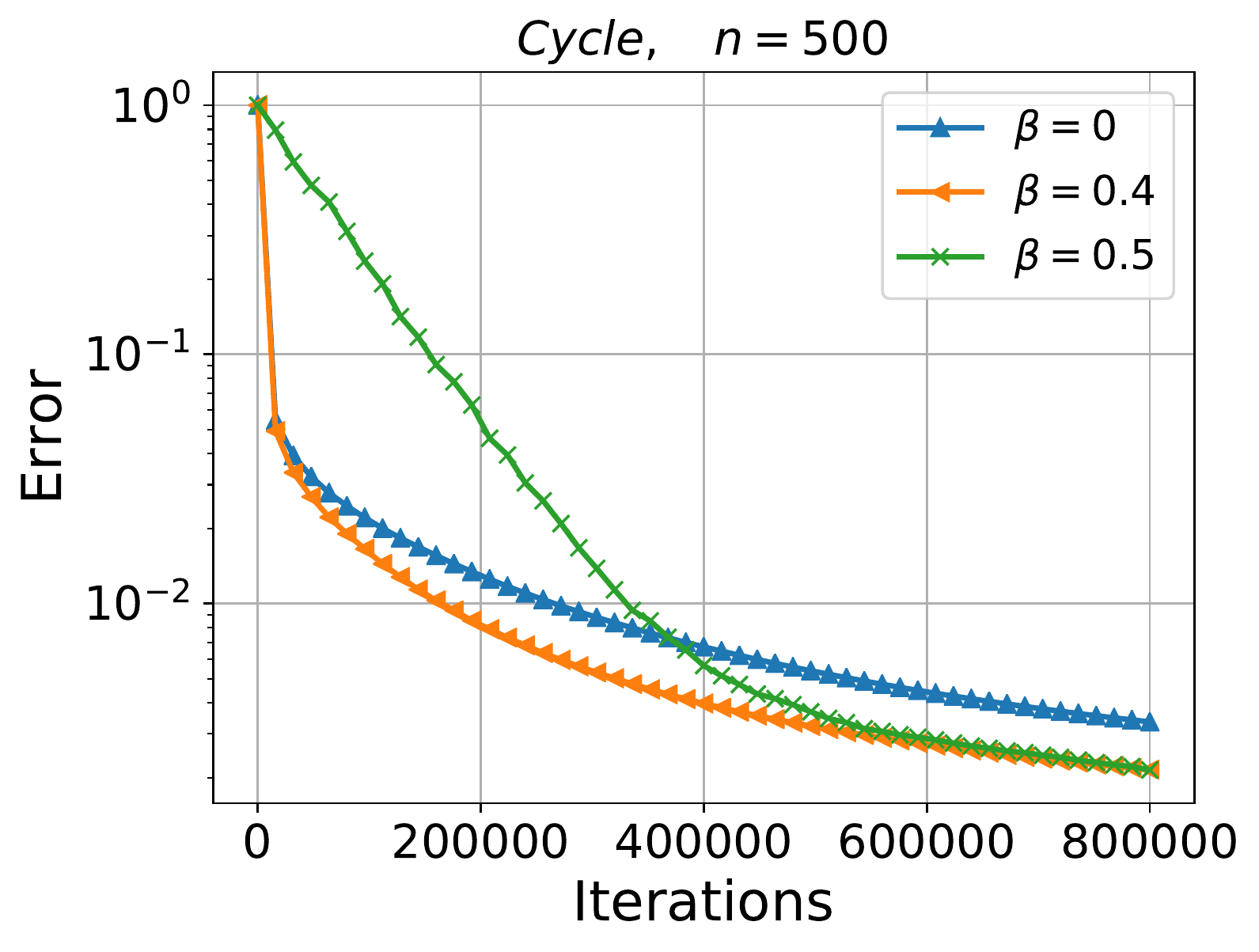}
\end{subfigure}\\
\caption{\footnotesize Performance of mRK for fixed step-size $\omega=1$ and several momentum parameters $\beta$ in a 2-dimension grid, random geoemtric graph (RGG) and a cycle graph. The choice $\beta=0$ corresponds to the randomized pairwise gossip algorithm proposed in \cite{boyd2006randomized}. The starting vector $x^0=c \in \R^n$ is a Gaussian vector. The $n$ in the title of each plot indicates the number of nodes of the network. For the grid graph this is $n \times n$. }
\label{mRKomega1}
\end{figure}

\subsubsection{Comparison of mRK and shift-Register algorithm \cite{liu2013analysis}}
In this experiment we compare mRK with the shift register gossip algorithm (pairwise momentum method, abbreviation: Pmom) analyzed in \cite{liu2013analysis}. We choose the parameters $\omega$ and $\beta$ of mRK in such a way in order to satisfy the connection established in Section~\ref{connectionOfAcceleratedMethods}. That is, we choose $\beta=\omega-1$ for any choice of $\omega \in (1,2)$. Observe that in all plots of Figure~\ref{shiftregister} mRK outperforms the corresponding shift-register algorithm. 

\begin{figure}[t]
\centering
\begin{subfigure}{.3\textwidth}
  \centering
  \includegraphics[width=1\linewidth]{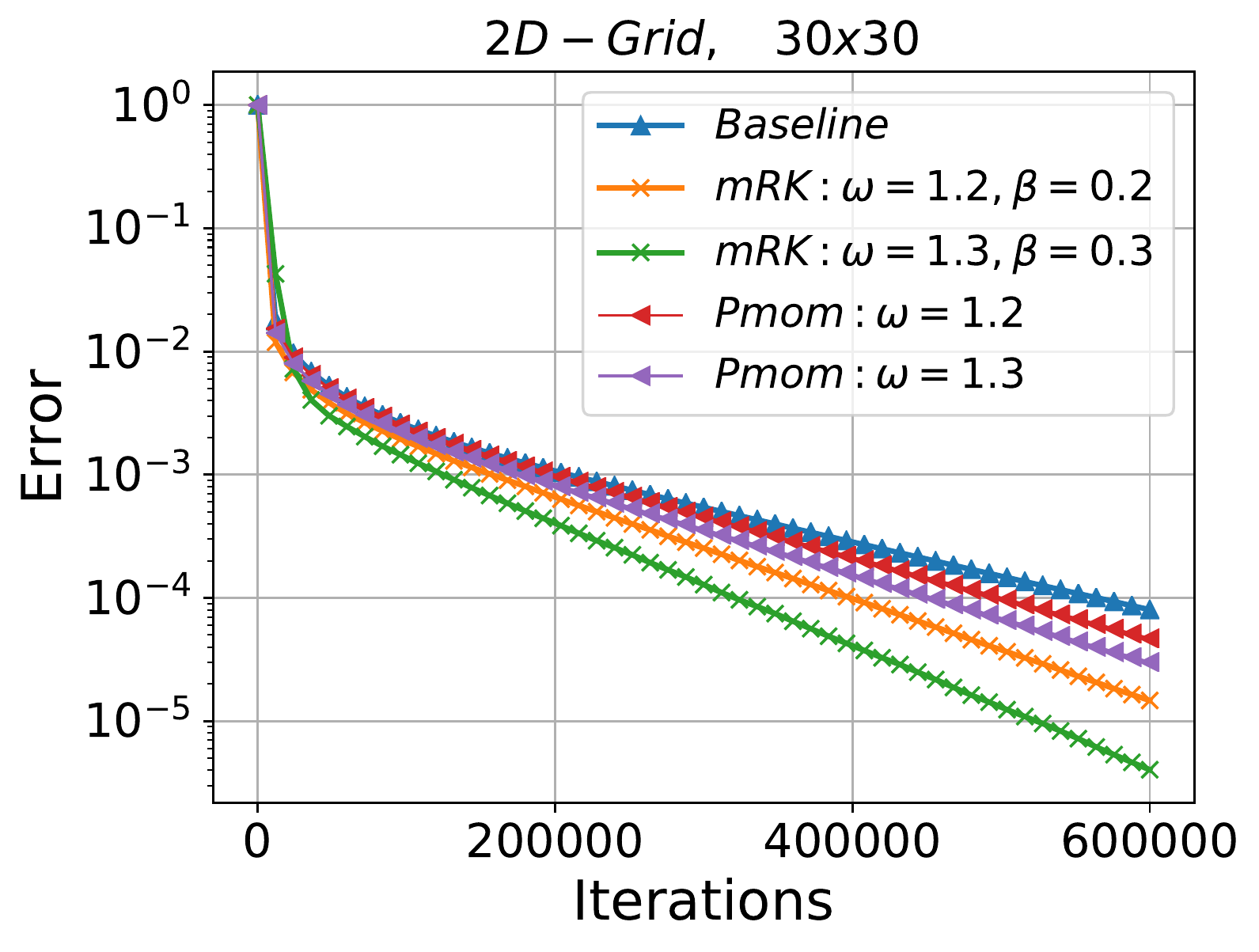}
\end{subfigure}%
\begin{subfigure}{.3\textwidth}
  \centering
  \includegraphics[width=1\linewidth]{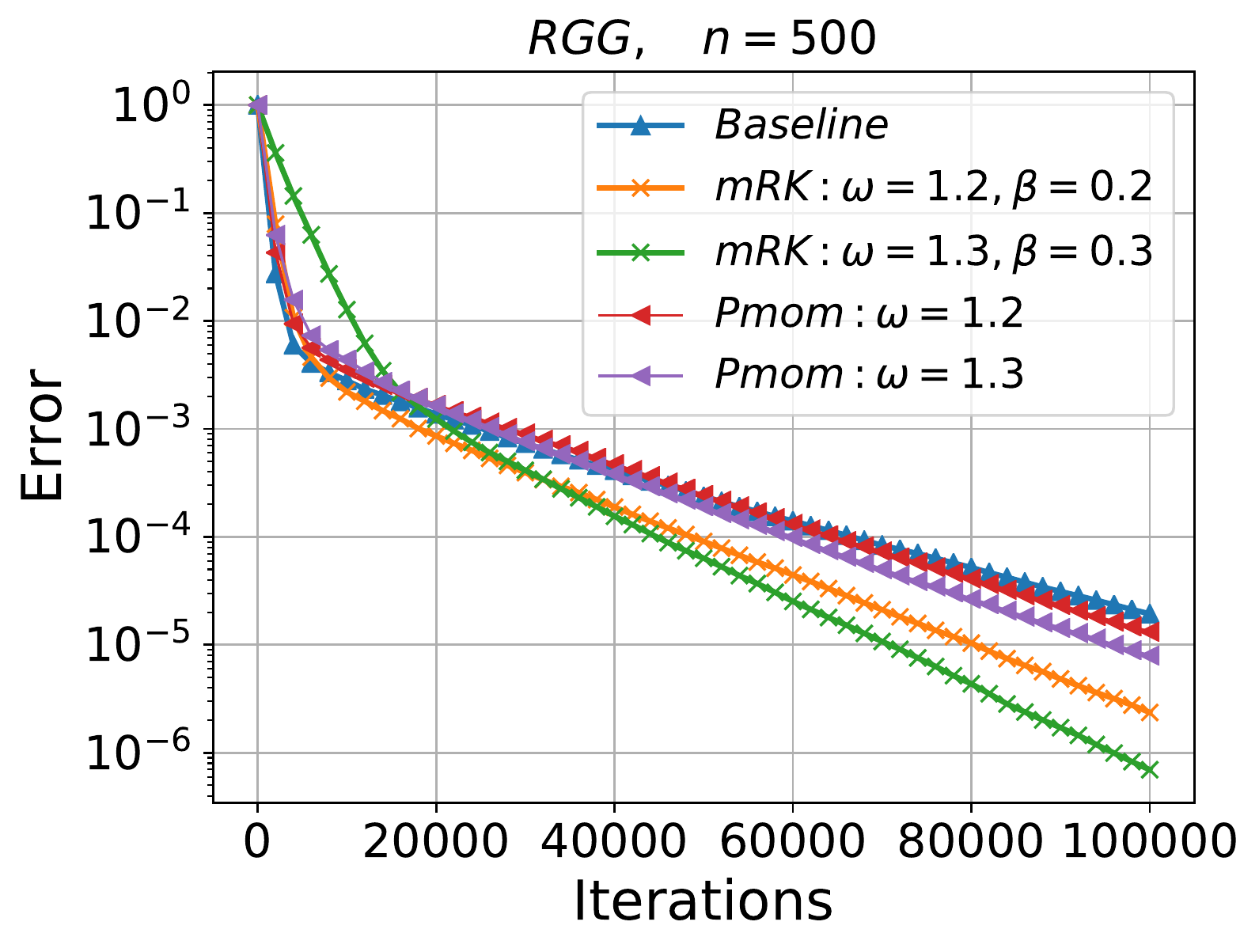}
\end{subfigure}%
\begin{subfigure}{.3\textwidth}
  \centering
  \includegraphics[width=1\linewidth]{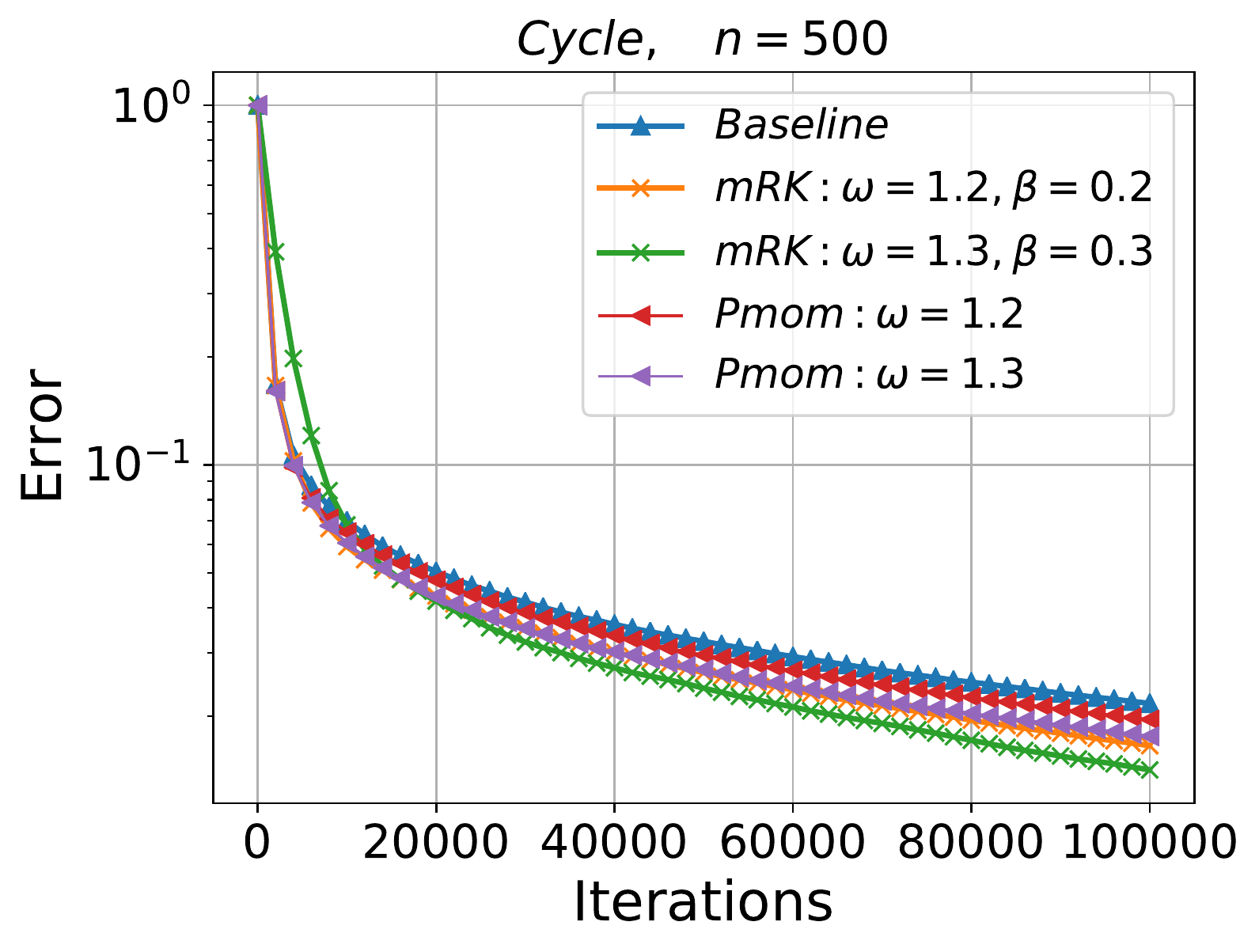}
\end{subfigure}\\
\begin{subfigure}{.3\textwidth}
  \centering
  \includegraphics[width=1\linewidth]{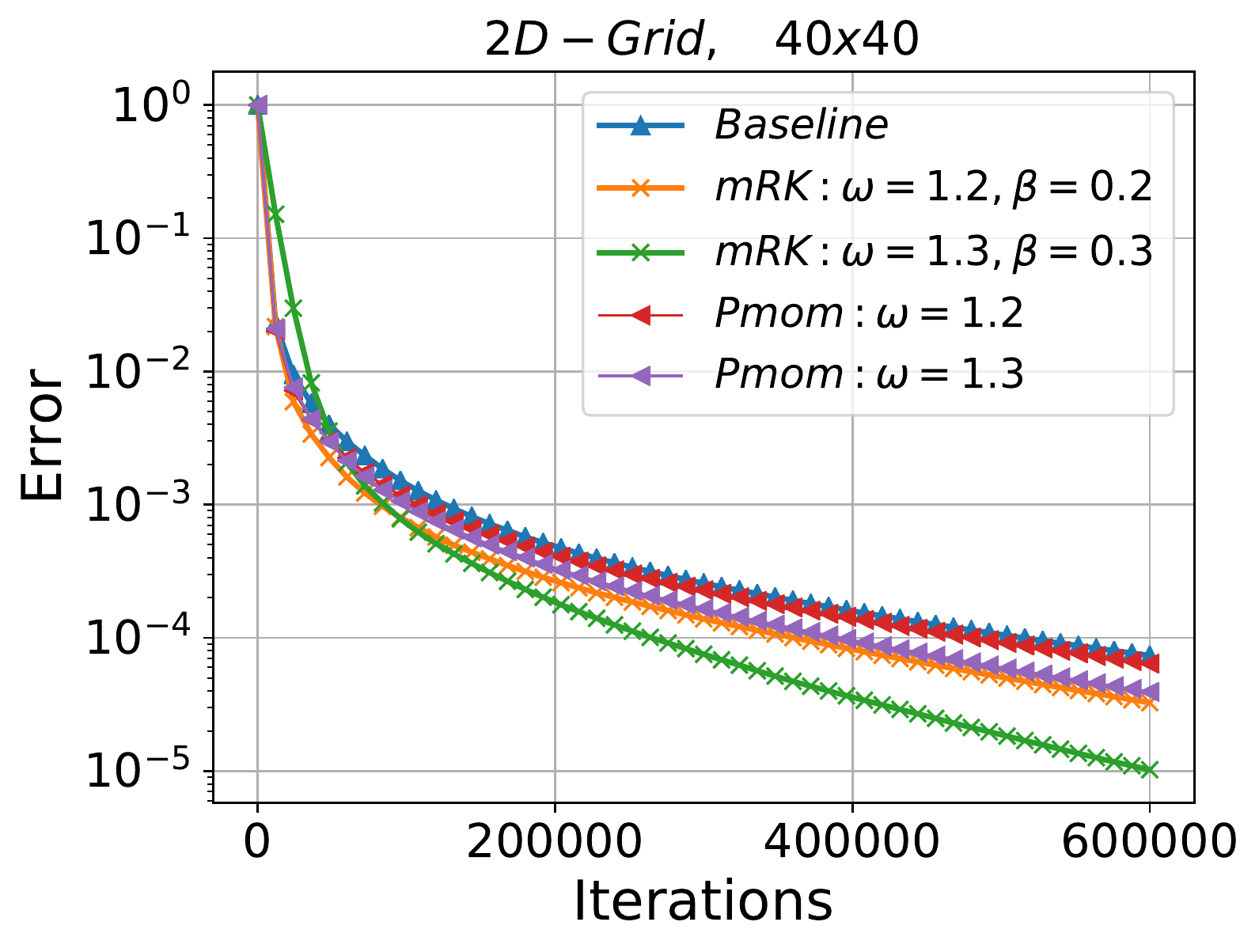}
\end{subfigure}
\begin{subfigure}{.3\textwidth}
  \centering
  \includegraphics[width=1\linewidth]{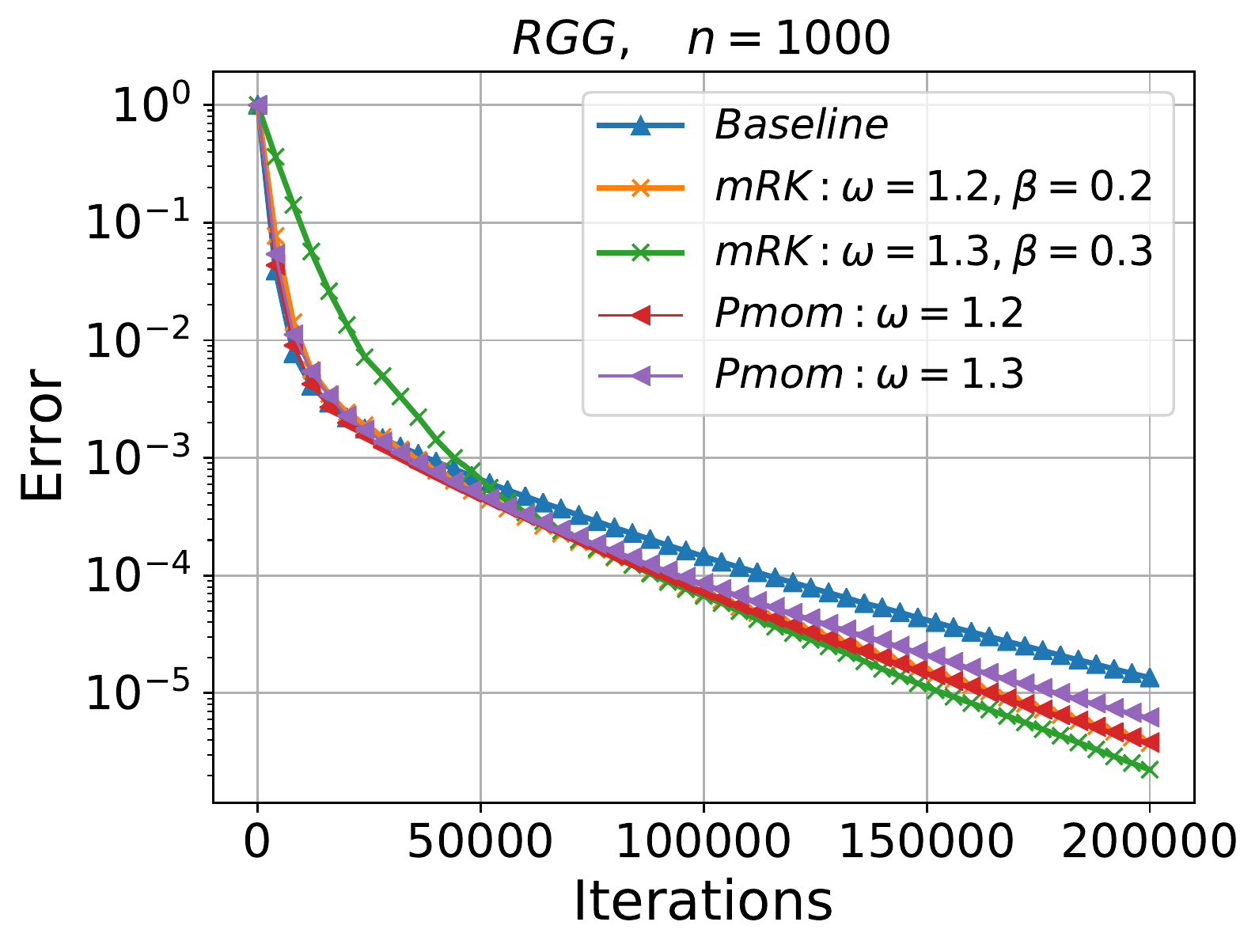}
\end{subfigure}
\begin{subfigure}{.3\textwidth}
  \centering
  \includegraphics[width=1\linewidth]{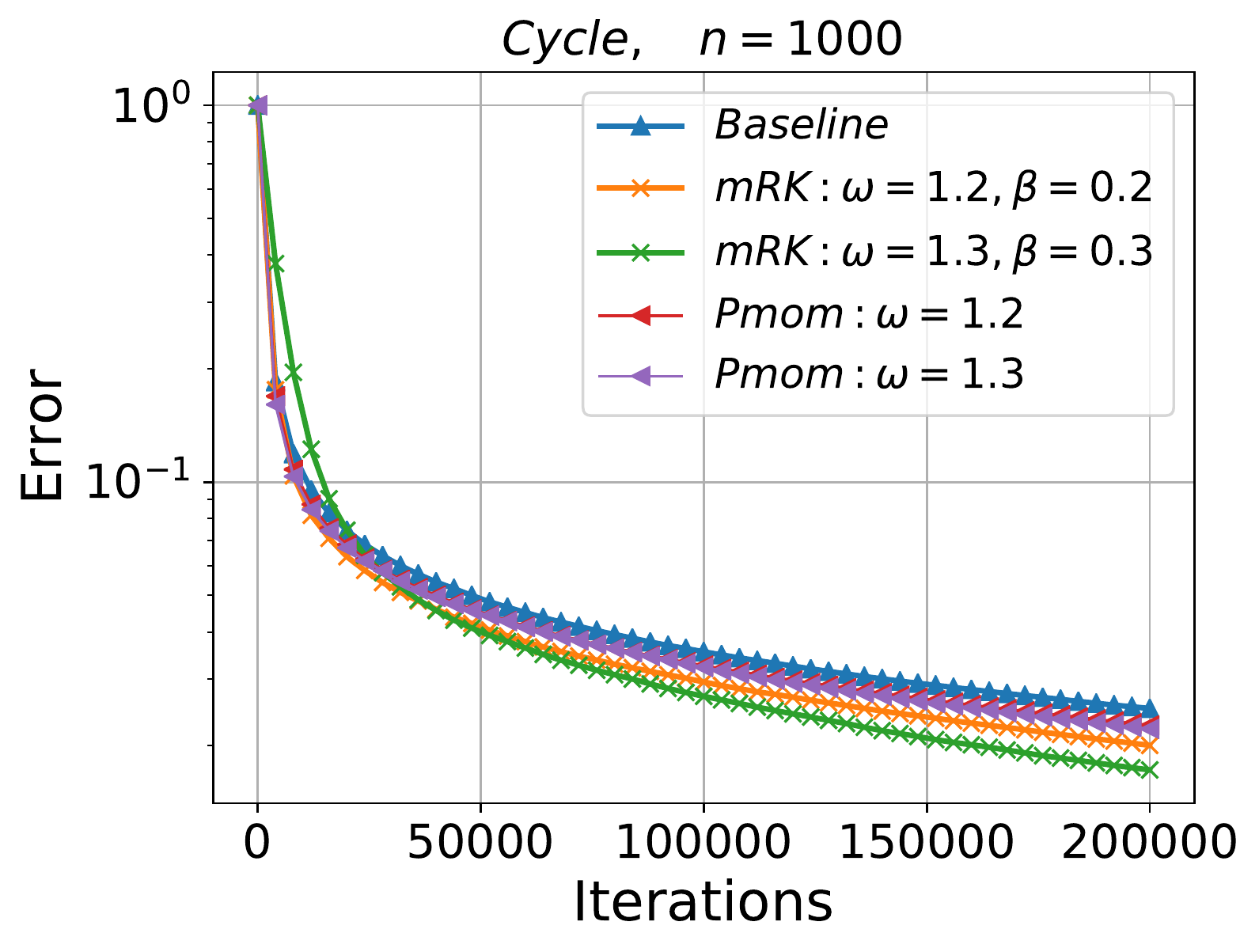}
\end{subfigure}\\
\caption{\footnotesize Comparison of mRK and the pairwise momentum method (Pmom), proposed in \cite{liu2013analysis} (shift-register algorithm of Section~\ref{connectionOfAcceleratedMethods}). Following the connection between mRK and Pmom established in Section~\ref{connectionOfAcceleratedMethods} the momentum parameter of mRK is chosen to be $\beta=\omega-1$ and the stepsizes are selected to be either $\omega= 1.2$ or $\omega=1.3$.  The baseline method is the standard randomized pairwise gossip algorithm from \cite{boyd2006randomized}. The starting vector $x^0=c \in \R^n$ is a Gaussian vector. The $n$ in the title of each plot indicates the number of nodes of the network. For the grid graph this is $n \times n$.}
\label{shiftregister}
\end{figure}

\subsubsection{Impact of momentum parameter on mRBK}
In this experiment our goal is to show that the addition of heavy ball momentum accelerates the RBK gossip algorithm presented in Section~\ref{BlockGossip}. Without loss of generality we choose the block size to be  equal to $\tau=5$. That is, the random matrix $\bS_k\sim \cD$ in the update rule of mRBK is a $m \times 5$ column submatrix of the indetity $m \times m$ matrix. Thus, in each iteration $5$ edges of the network are chosen to form the subgraph $\cG_k$ and the values of the nodes are updated according to Algorithm~\ref{RBKmomentum}. Note that similar plots can be obtained for any choice of block size. We run all algorithms with fixed stepsize $\omega=1$. From Figure~\ref{RBKfigures}, it is obvious that for all networks under study, choosing a suitable momentum parameter $\beta \in (0,1)$ gives faster convergence than having no momentum, $\beta =0$. 

\begin{figure}[t]
\centering
\begin{subfigure}{.3\textwidth}
  \centering
  \includegraphics[width=1\linewidth]{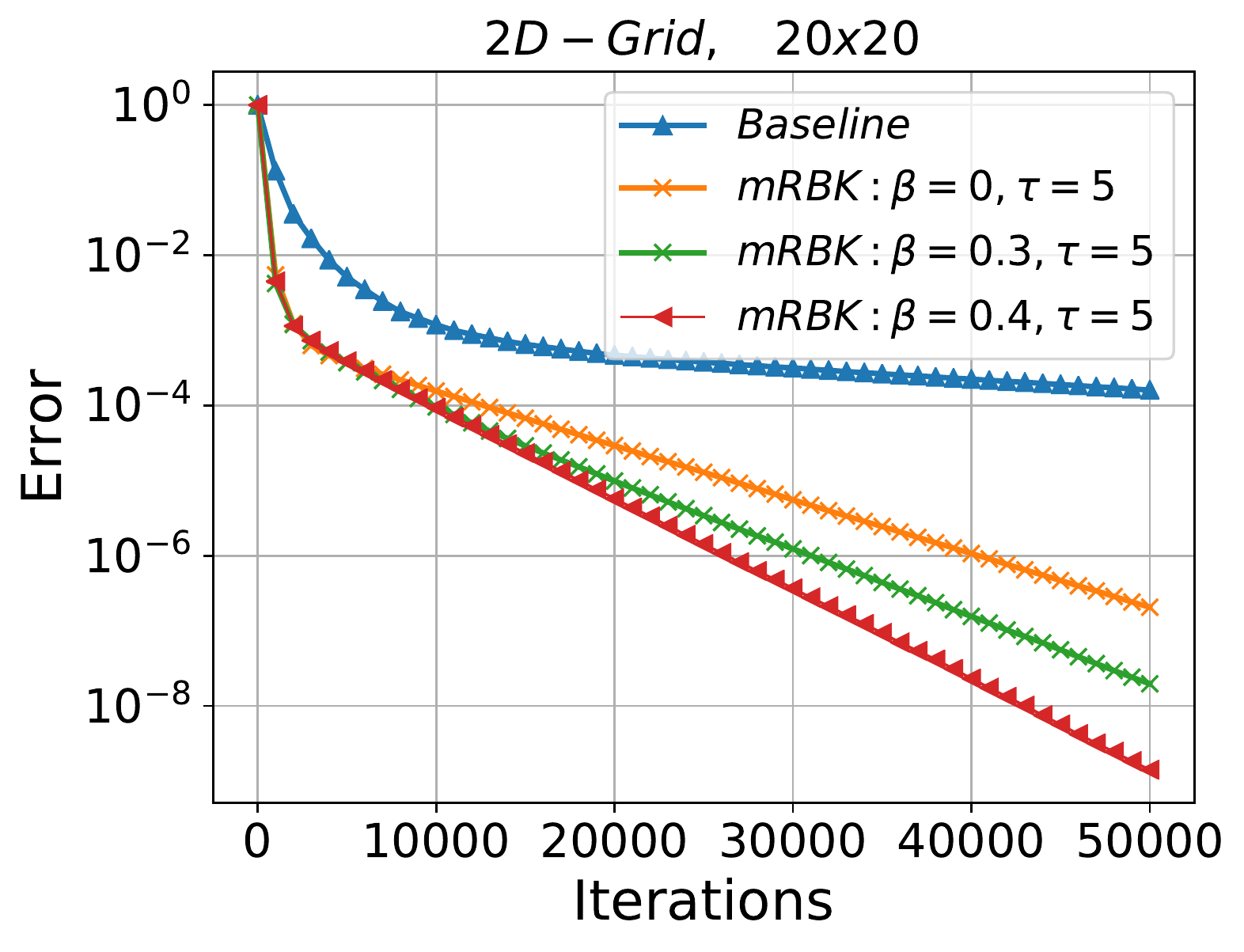}
\end{subfigure}%
\begin{subfigure}{.3\textwidth}
  \centering
  \includegraphics[width=1\linewidth]{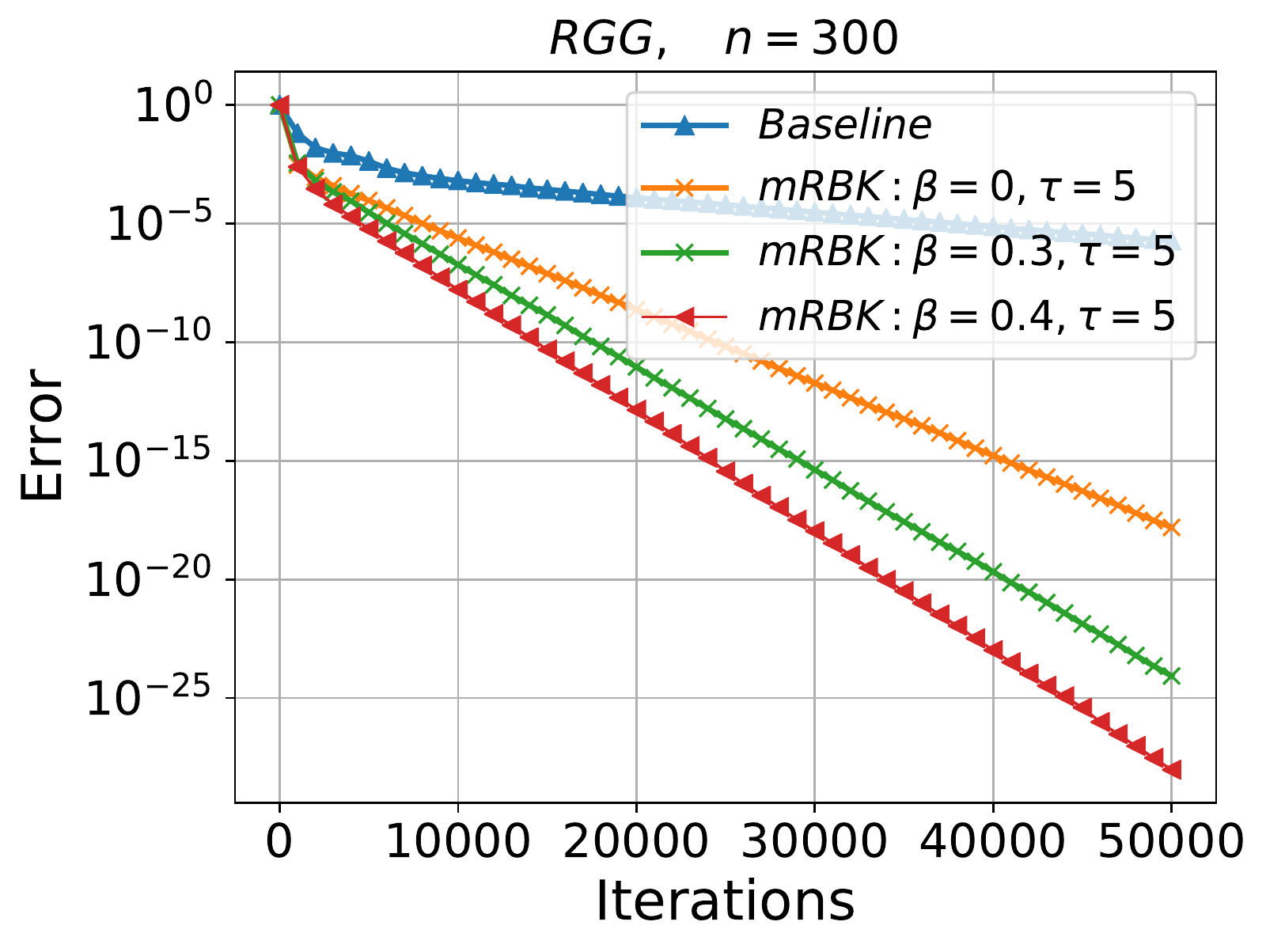}
\end{subfigure}%
\begin{subfigure}{.3\textwidth}
  \centering
  \includegraphics[width=1\linewidth]{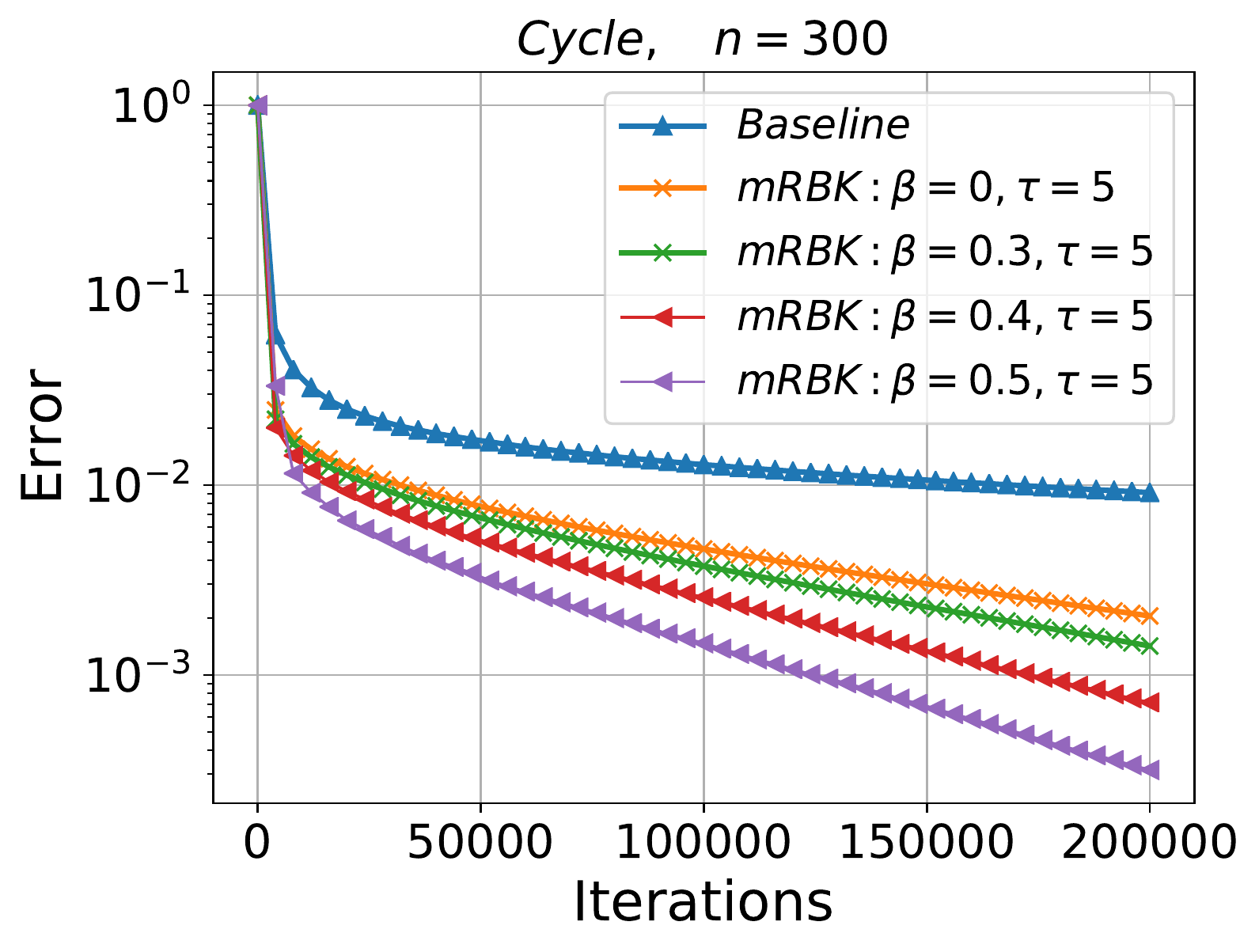}
\end{subfigure}\\
\begin{subfigure}{.3\textwidth}
  \centering
  \includegraphics[width=1\linewidth]{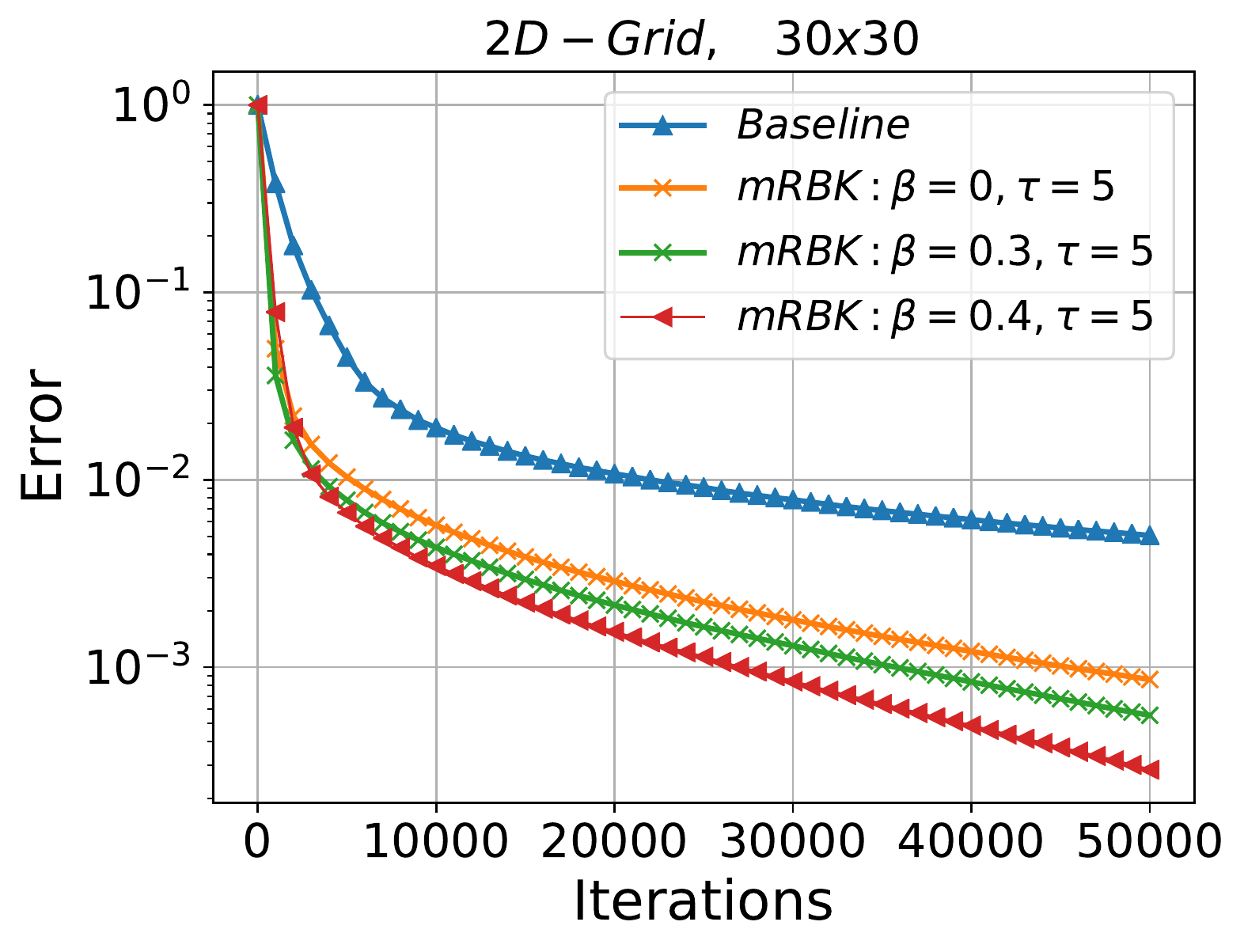}
\end{subfigure}
\begin{subfigure}{.3\textwidth}
  \centering
  \includegraphics[width=1\linewidth]{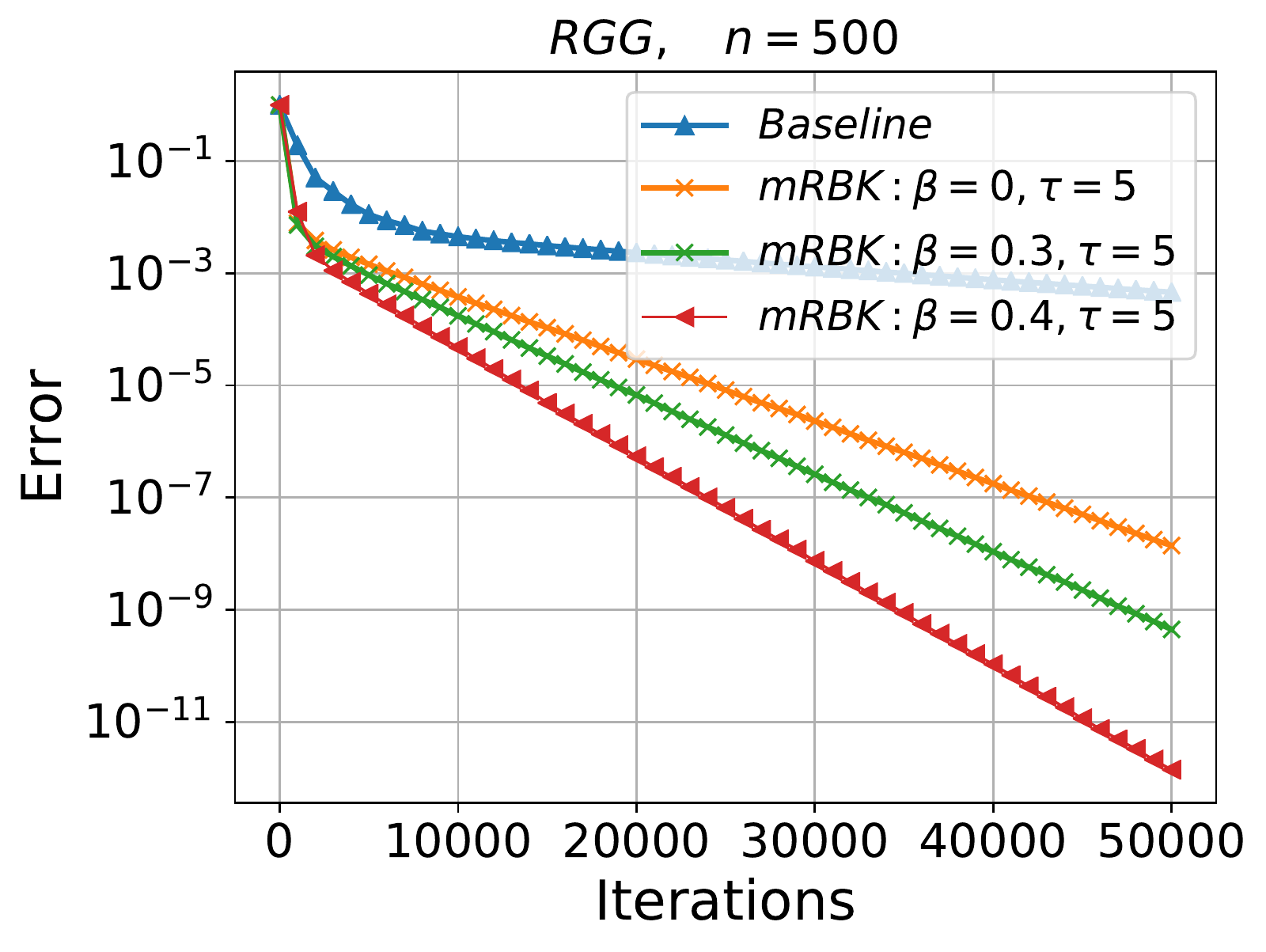}
\end{subfigure}
\begin{subfigure}{.3\textwidth}
  \centering
  \includegraphics[width=1\linewidth]{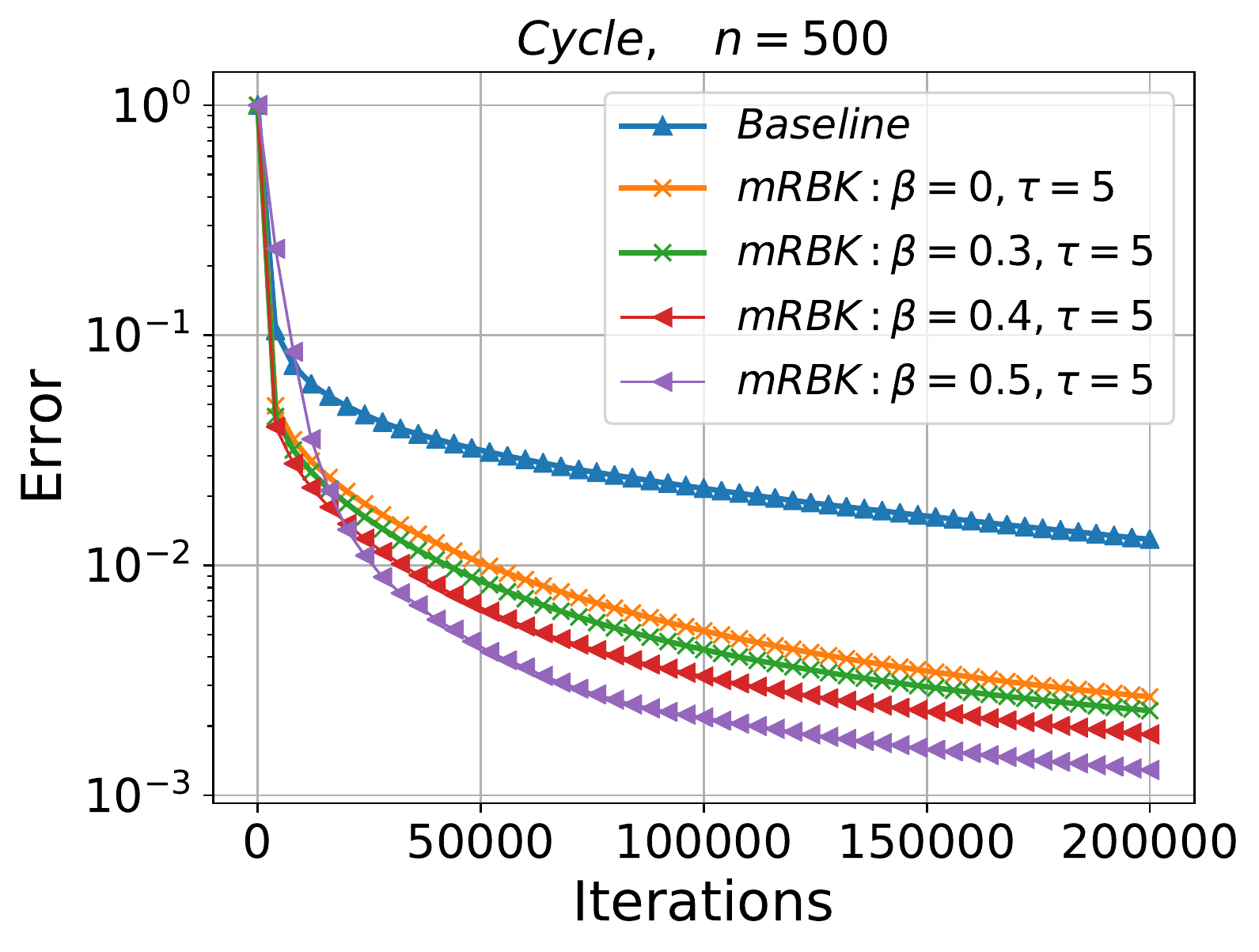}
\end{subfigure}\\
\caption{\footnotesize Comparison of mRBK with its no momentum variant RBK ($\beta=0$).  The stepsize for all methods is $\omega=1$ and the block size is $\tau=5$. The baseline method in the plots denotes the standard randomized pairwise gossip algorithm (block $\tau=1$) and is plotted to highlight the benefits of having larger block sizes (at least in terms of iterations). The starting vector $x^0=c \in \R^n$ is a Gaussian vector. The $n$ in the title of each plot indicates the number of nodes. For the grid graph this is $n \times n$.}
\label{RBKfigures}
\end{figure}

\subsubsection{Performance of AccGossip}
In the last experiment on faster gossip algorithms we evaluate the performance of the proposed provably accelerated gossip protocols of Section~\ref{accSubsection}. In particular we compare the standard RK (pairwise gossip algorithm of \cite{boyd2006randomized}) the mRK (Algorithm~\ref{RKmomentum}) and the AccGossip (Algorithm~\ref{alg:acceleratedNew}) with the two options for the selection of the parameters presented in Section~\ref{AcceleratedVariants}. 

The starting vector of values $x^0=c$ is taken to be a Gaussian vector.  For the implementation of mRK we use the same parameters with the ones suggested in the stochastic heavy ball (SGB) setting in \cite{loizou2017momentum}. For the AccRK (Option 1) we use $\lambda=\lambda_{\min}^+(\bA^\top\bA)$ and for AccRK (Option 2) we select $\nu=m$\footnote{For the networks under study we have
 $m < \frac{1}{\lambda_{\min}^+(\bW)}$. Thus, by choosing $\nu=m$ we select the pessimistic upper bound of the parameter \eqref{acsnklasda} and not its exact value \eqref{thenu}. As we can see from the experiments, the performance is still accelerated and almost identical to the performance of AccRK (Option 1) for this choice of $\nu$.}. From Figure~\ref{AccGossipPlots} it is clear that for all networks under study the two randomized gossip protocols with Nesterov momentum are faster than both the pairwise gossip algorithm of \cite{boyd2006randomized} and the mRK/SHB (Algorithm~\ref{RKmomentum}). To the best of our knowledge Algorithm~\ref{alg:acceleratedNew} (Option 1 and Option 2) is the first randomized gossip protocol that converges with provably accelerated linear rate and as we can see from our experiment its faster convergence is also obvious in practice.

\begin{figure}[t]
\centering
\begin{subfigure}{.3\textwidth}
  \centering
  \includegraphics[width=1\linewidth]{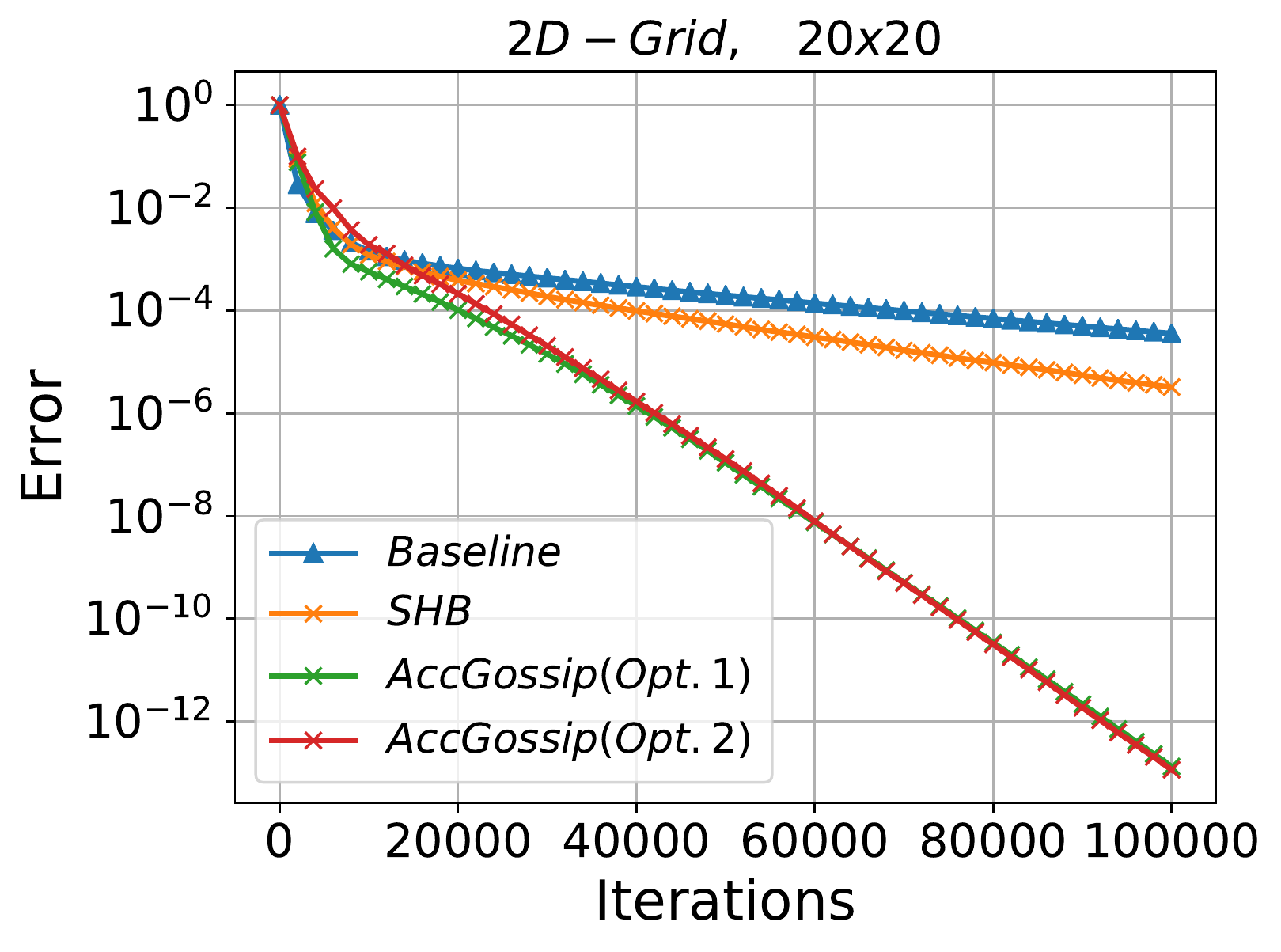}
\end{subfigure}%
\begin{subfigure}{.3\textwidth}
  \centering
  \includegraphics[width=1\linewidth]{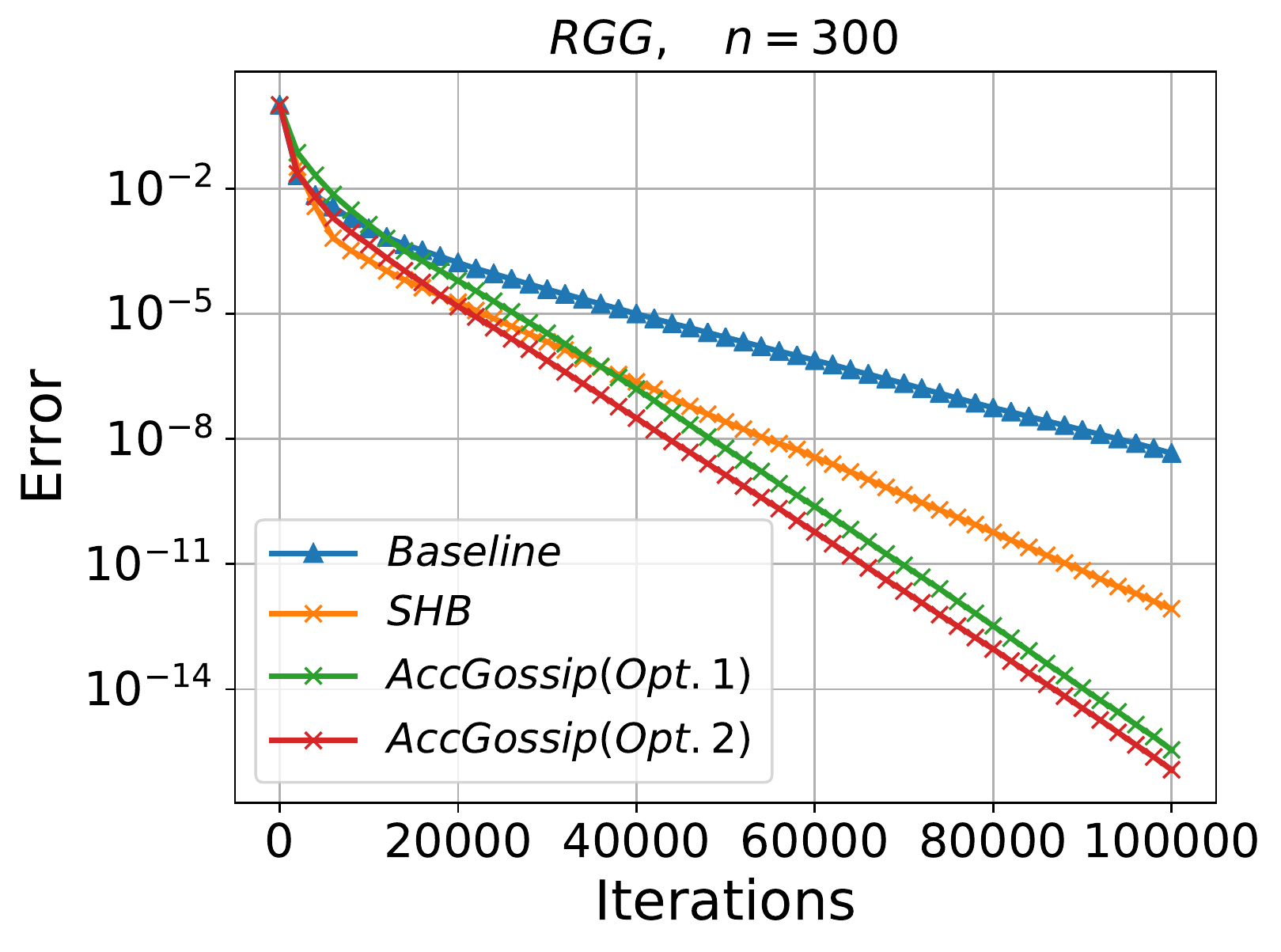}
\end{subfigure}%
\begin{subfigure}{.3\textwidth}
  \centering
  \includegraphics[width=1\linewidth]{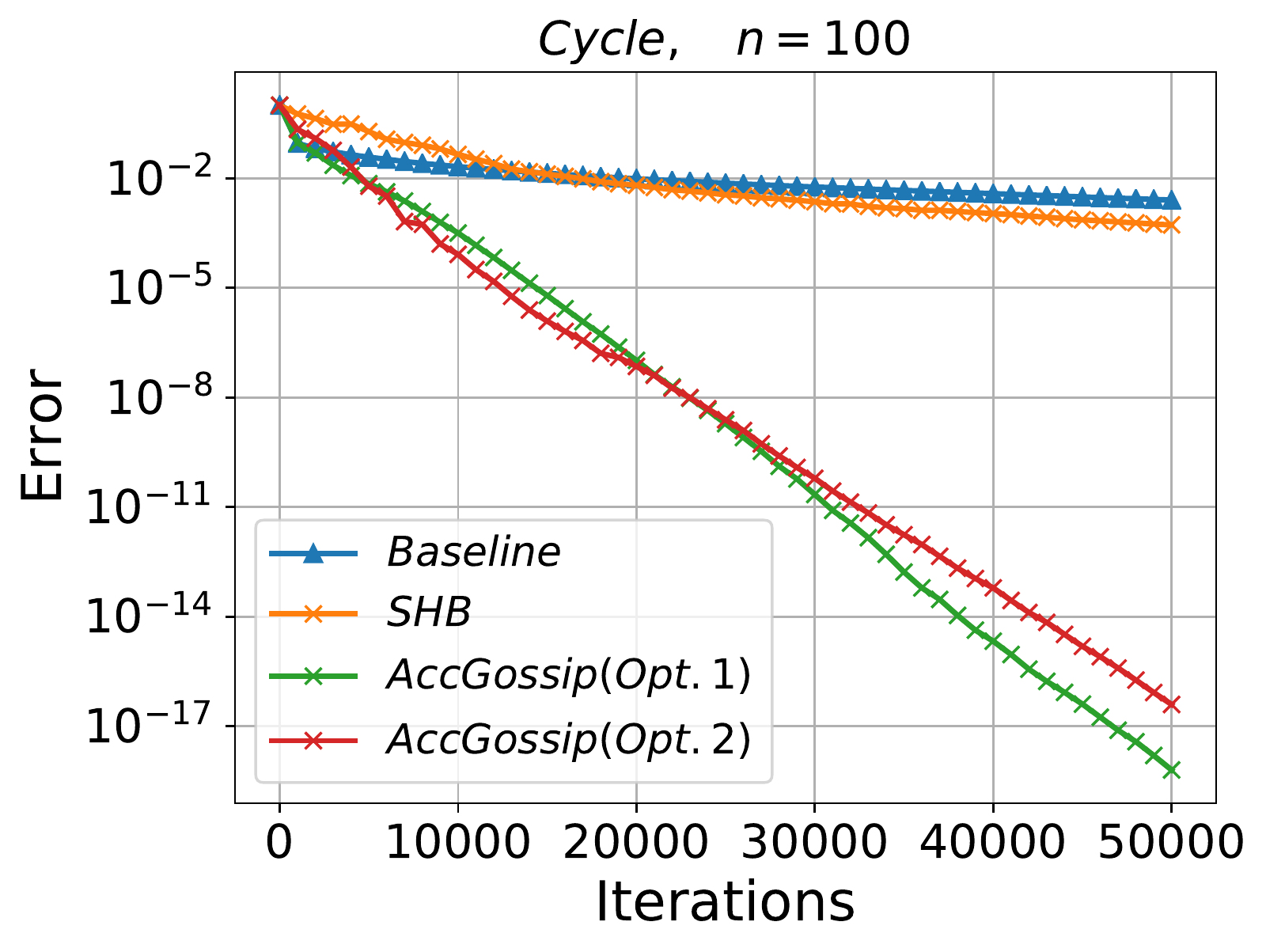}
\end{subfigure}\\
\begin{subfigure}{.3\textwidth}
  \centering
  \includegraphics[width=1\linewidth]{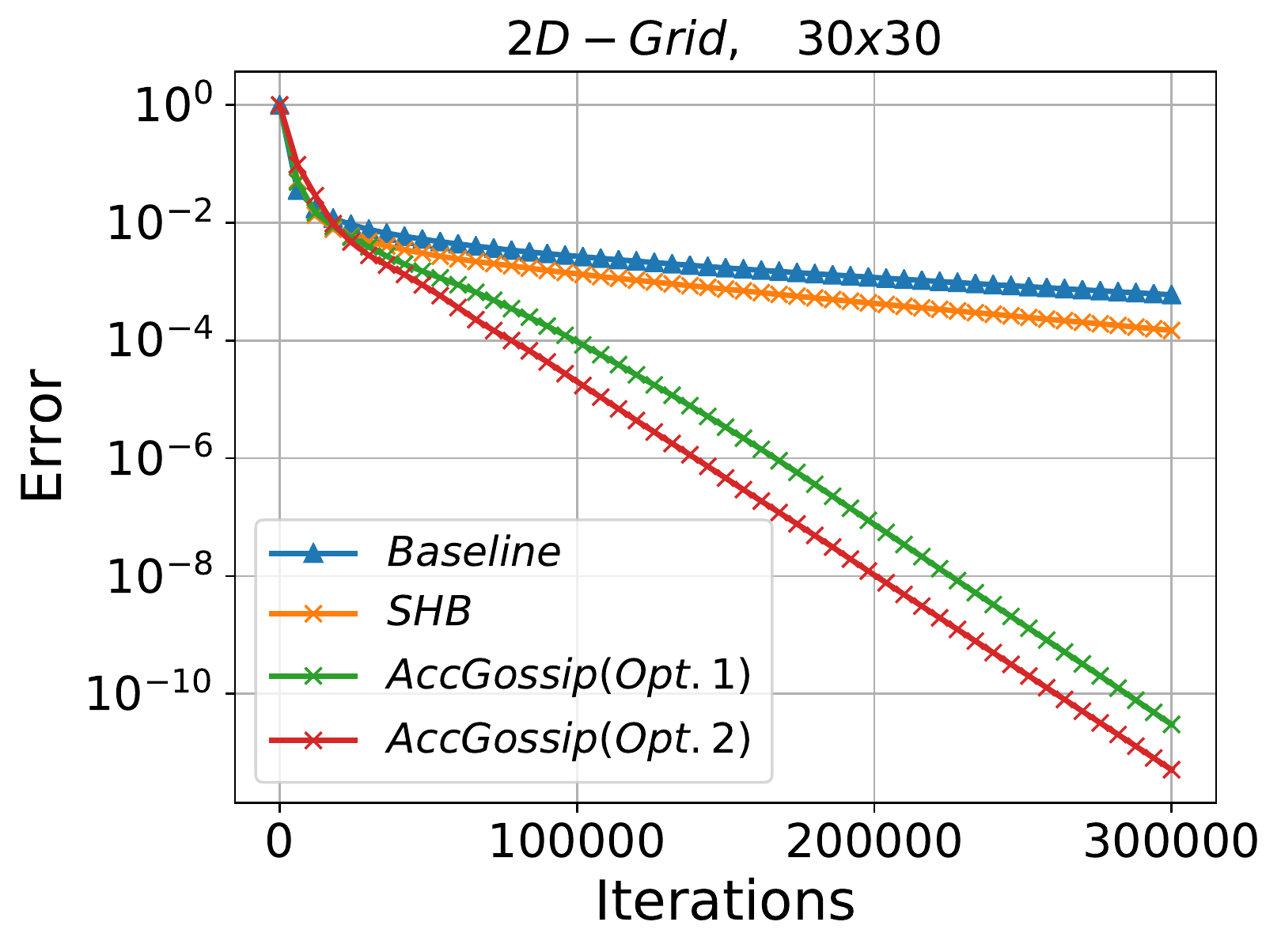}
\end{subfigure}
\begin{subfigure}{.3\textwidth}
  \centering
  \includegraphics[width=1\linewidth]{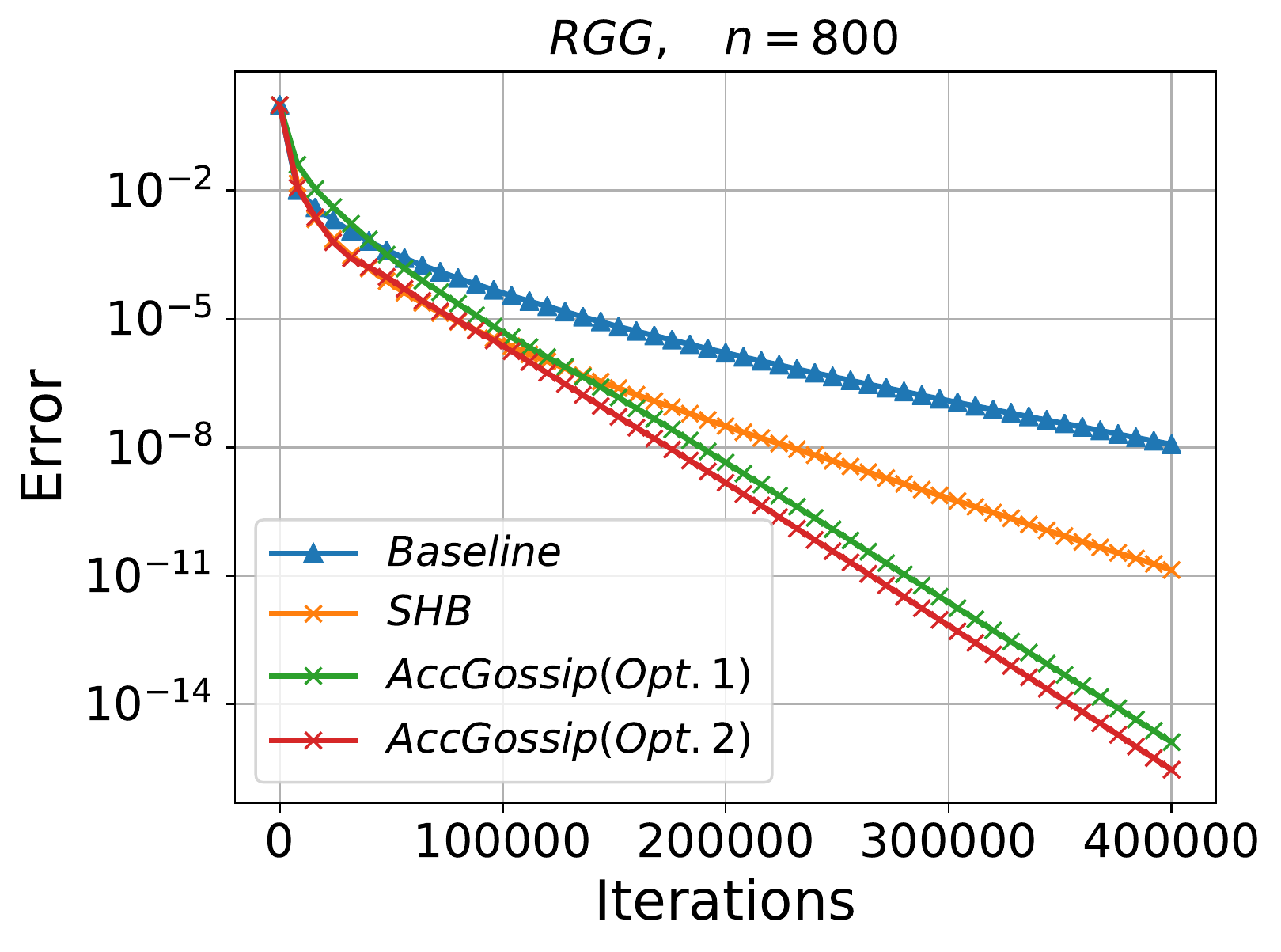}
\end{subfigure}
\begin{subfigure}{.3\textwidth}
  \centering
  \includegraphics[width=1\linewidth]{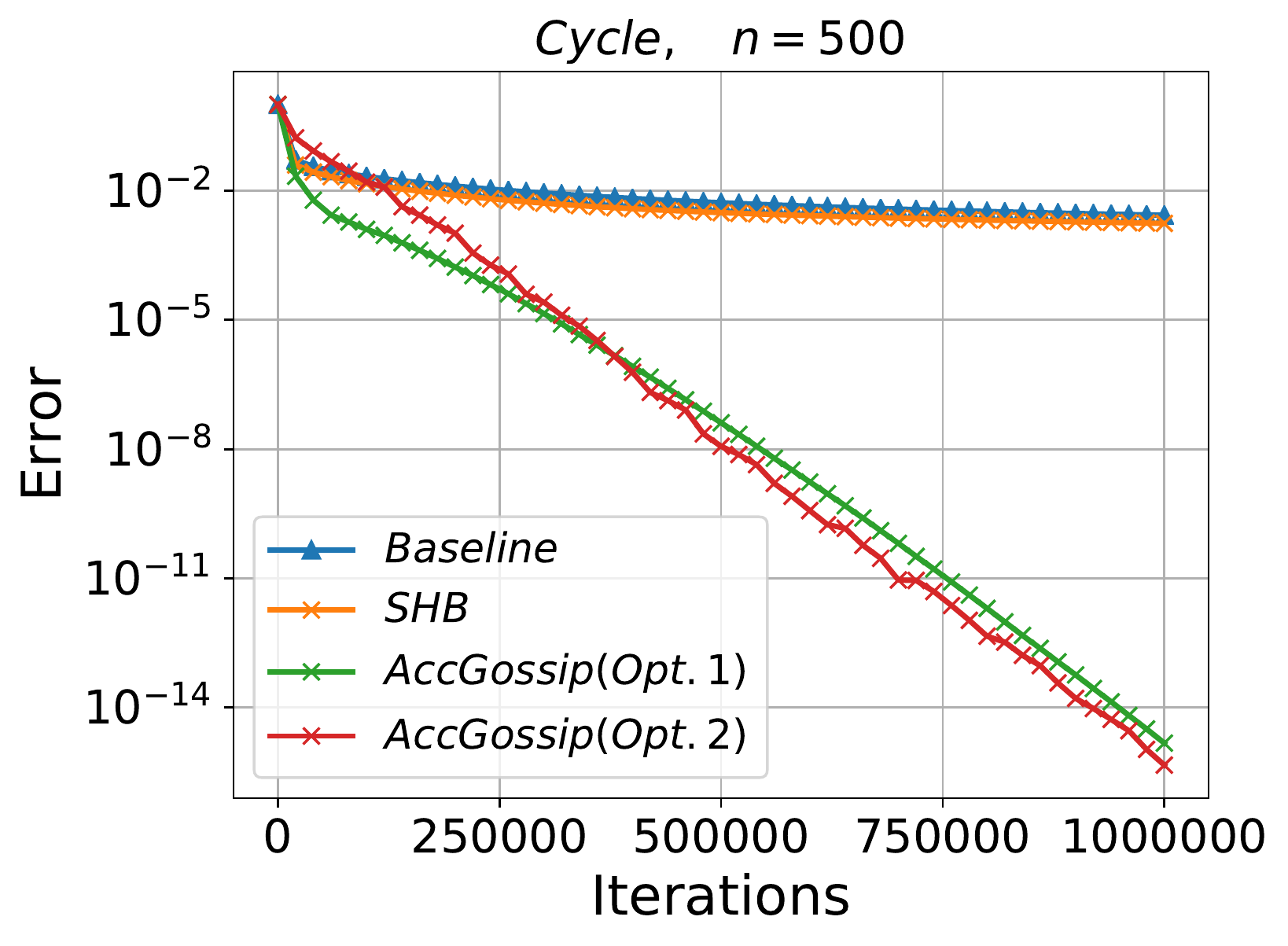}
\end{subfigure}\\
\caption{\footnotesize Performance of AccGossip (Option 1 and Option 2 for the parameters) in a 2-dimension grid, random geometric graph (RGG) and a cycle graph. The Baseline method corresponds to the randomized pairwise gossip algorithm proposed in \cite{boyd2006randomized} and the SHB represents the mRK (Algorithm~\ref{RKmomentum}) with the best choice of parameters as proposed in \cite{loizou2017momentum} ; The $n$ in the title of each plot indicates the number of nodes of the network. For the grid graph this is $n \times n$.}
\label{AccGossipPlots}
\end{figure}

\subsection{Relaxed randomized gossip without momentum}
In the area of randomized iterative methods for linear systems it is know that over-relaxation (using of larger step-sizes) can be particularly helpful in practical scenarios. However, to the best of our knowledge there is not theoretical justification of why this is happening.

In our last experiment we explore the performance of relaxed randomized gossip algorithms ($\omega \neq 1$) without momentum and show that in this setting having larger stepsize can be particularly beneficial.

As we mentioned before (see Theorem~\ref{ConvergenceSketchProject}) the sketch and project method (Algorithm~\ref{FullSkecth}) converges with linear rate when the step-size (relaxation parameter) of the method is $\omega \in (0,2)$ and the best theoretical rate is achieved when $\omega=1$.  In this experiment we explore the performance of the standard pairwise gossip algorithm when the step-size of the update rule is chosen in $(1,2)$. Since there is no theoretical proof of why over-relaxation can be helpful we perform the experiments using different starting values of the nodes. In particular we choose the values of vector $c \in R^n$ to follow (i) Gaussian distribution, (ii) Uniform Distribution and (iii) to be integers values such that $c_i=i \in R$. Our findings are presented in Figure~\ref{RelaxedGossipFigure}. Note that for all networks under study and for all choices of starting values having larger stepsize, $\omega \in (1,2)$ can lead to better performance. Interesting observation from Figure~\ref{RelaxedGossipFigure} is that the stepsizes $\omega =1.8$ and $\omega=1.9$ give the best performance (among the selected choices of stepsizes) for all networks and for all choices of starting vector $x^0=c$.

\begin{figure}[t]
\centering
\begin{subfigure}{.3\textwidth}
  \centering
  \includegraphics[width=1\linewidth]{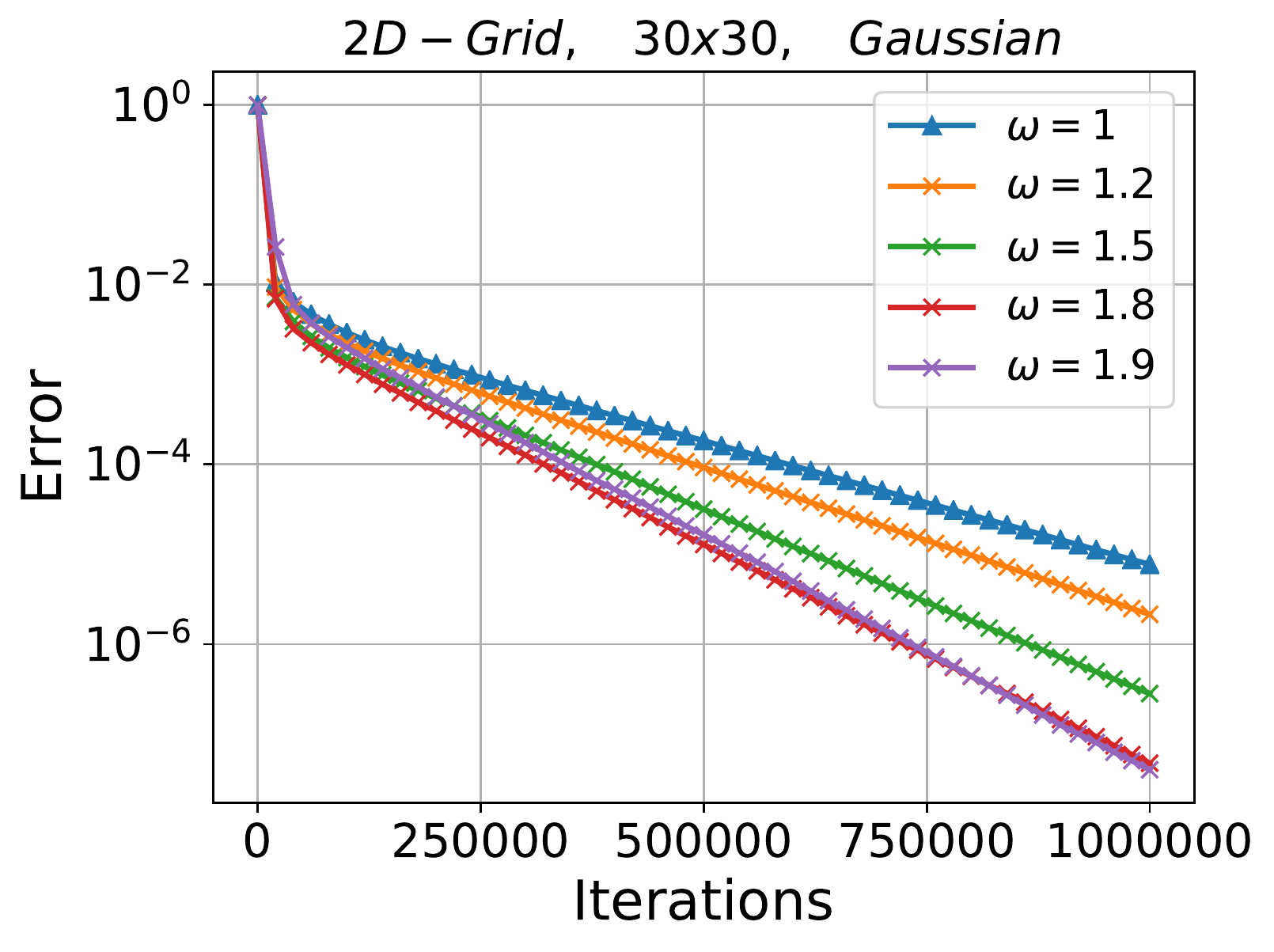}
\end{subfigure}%
\begin{subfigure}{.3\textwidth}
  \centering
  \includegraphics[width=1\linewidth]{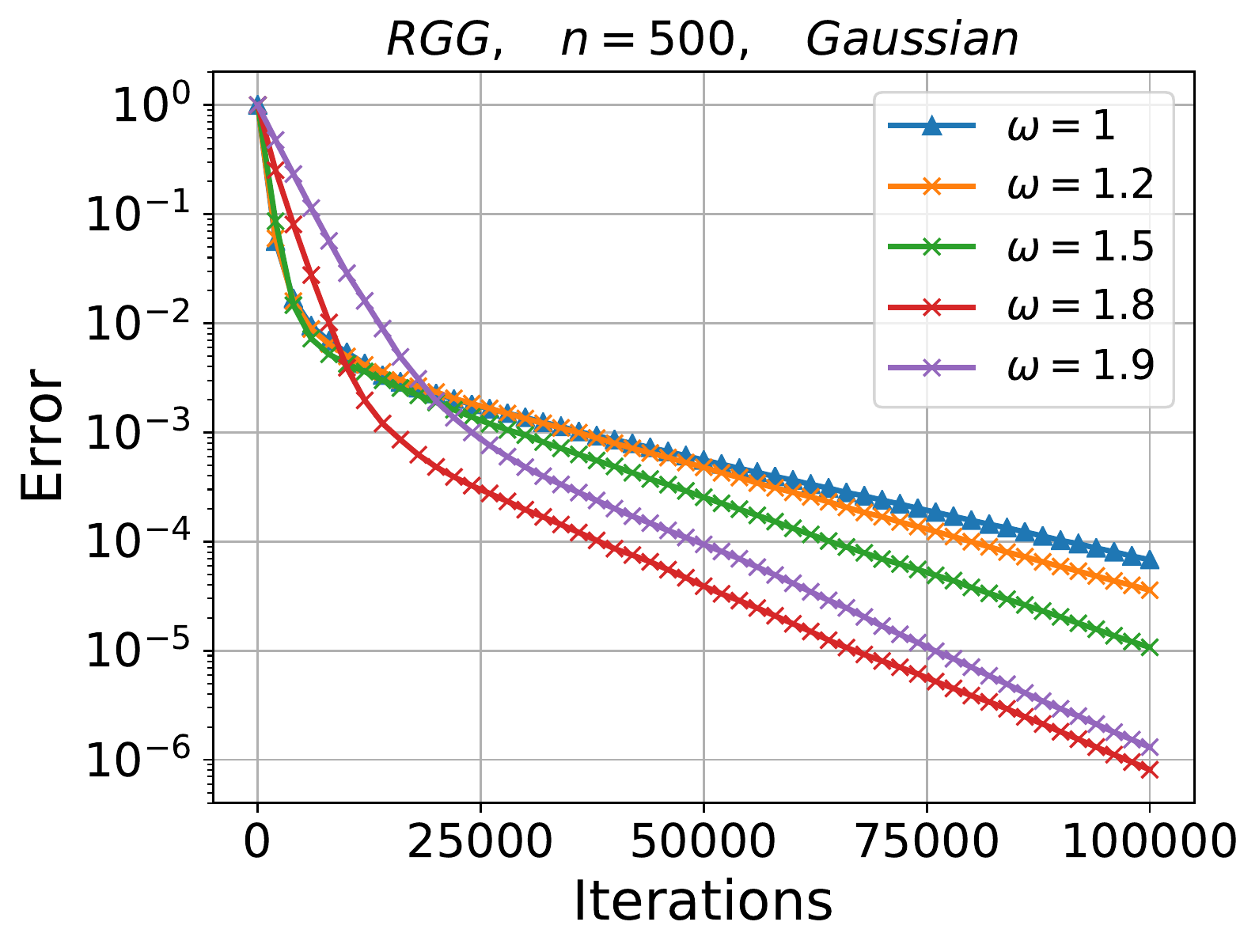}
\end{subfigure}
\begin{subfigure}{.3\textwidth}
  \centering
  \includegraphics[width=1\linewidth]{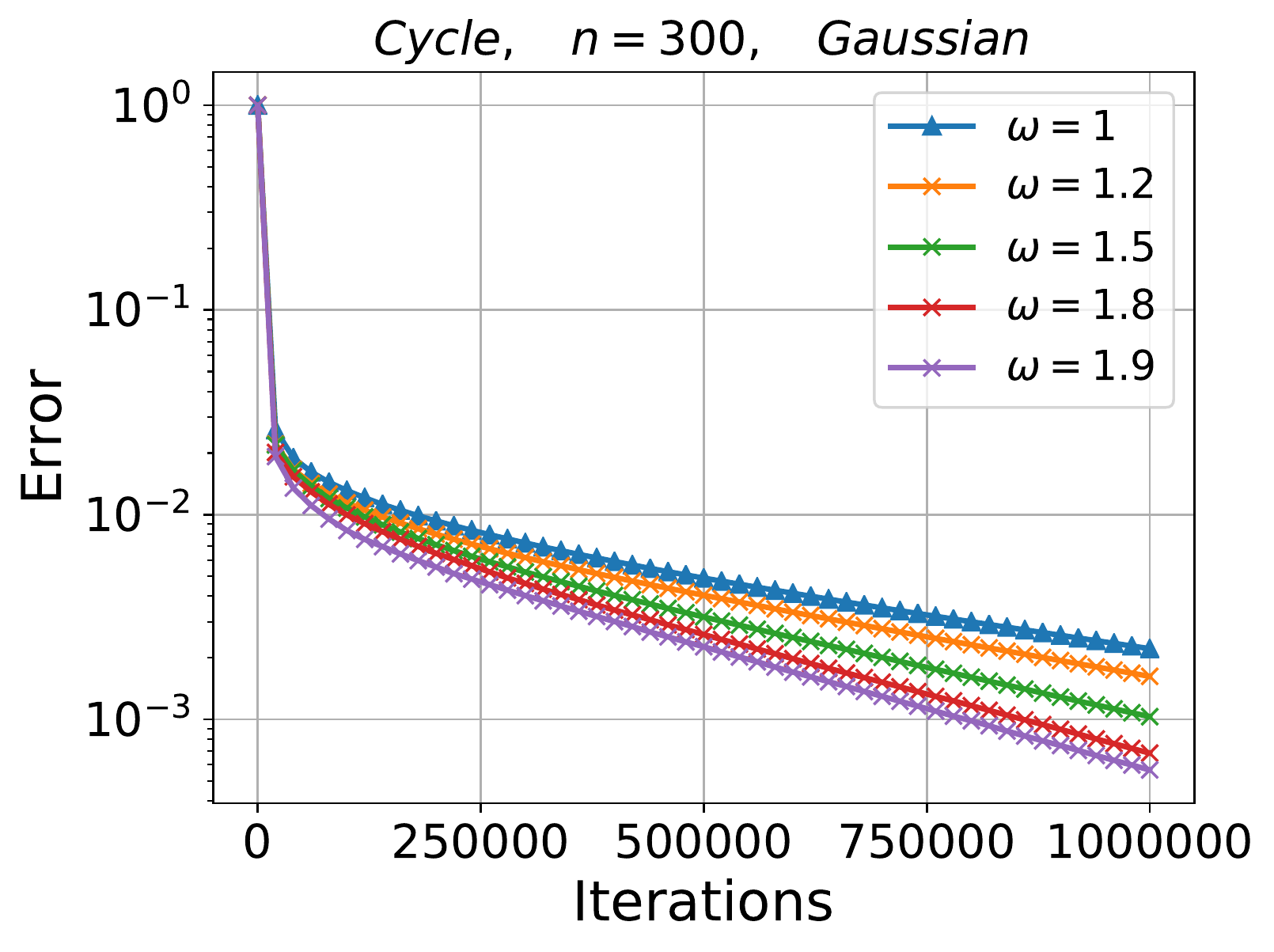}
\end{subfigure}\\
\begin{subfigure}{.3\textwidth}
  \centering
  \includegraphics[width=1\linewidth]{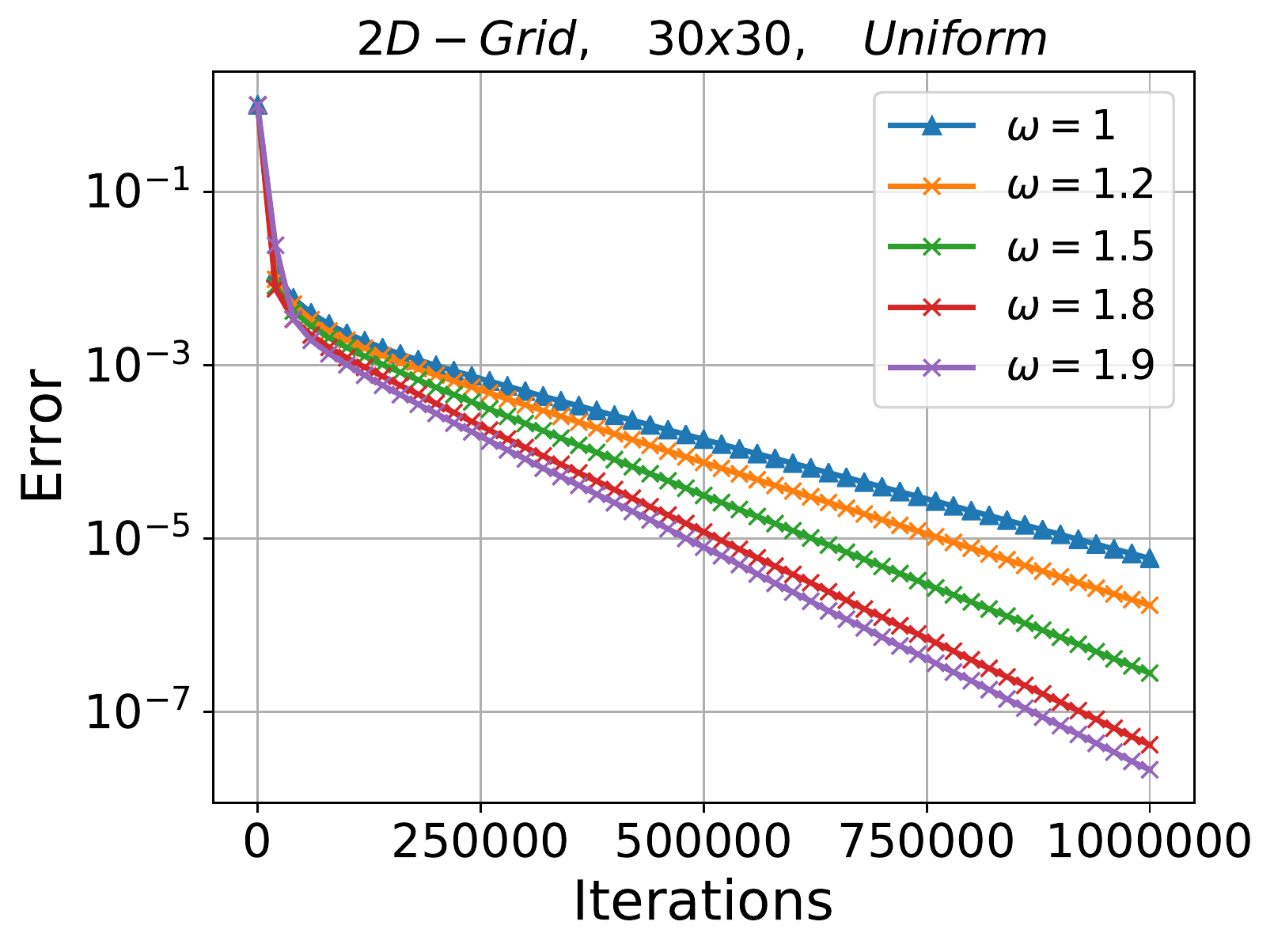}
\end{subfigure}%
\begin{subfigure}{.3\textwidth}
  \centering
  \includegraphics[width=1\linewidth]{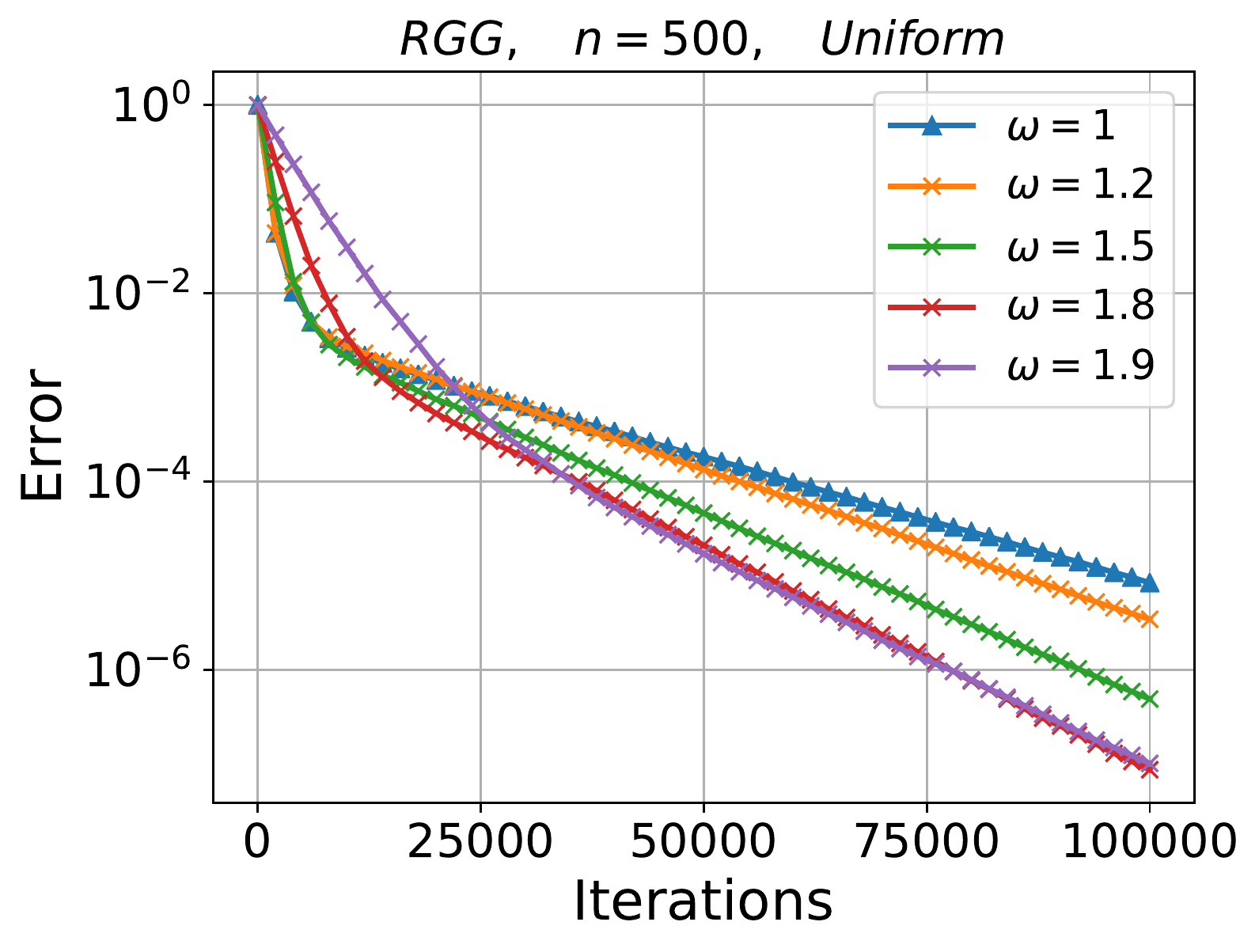}
\end{subfigure}
\begin{subfigure}{.3\textwidth}
  \centering
  \includegraphics[width=1\linewidth]{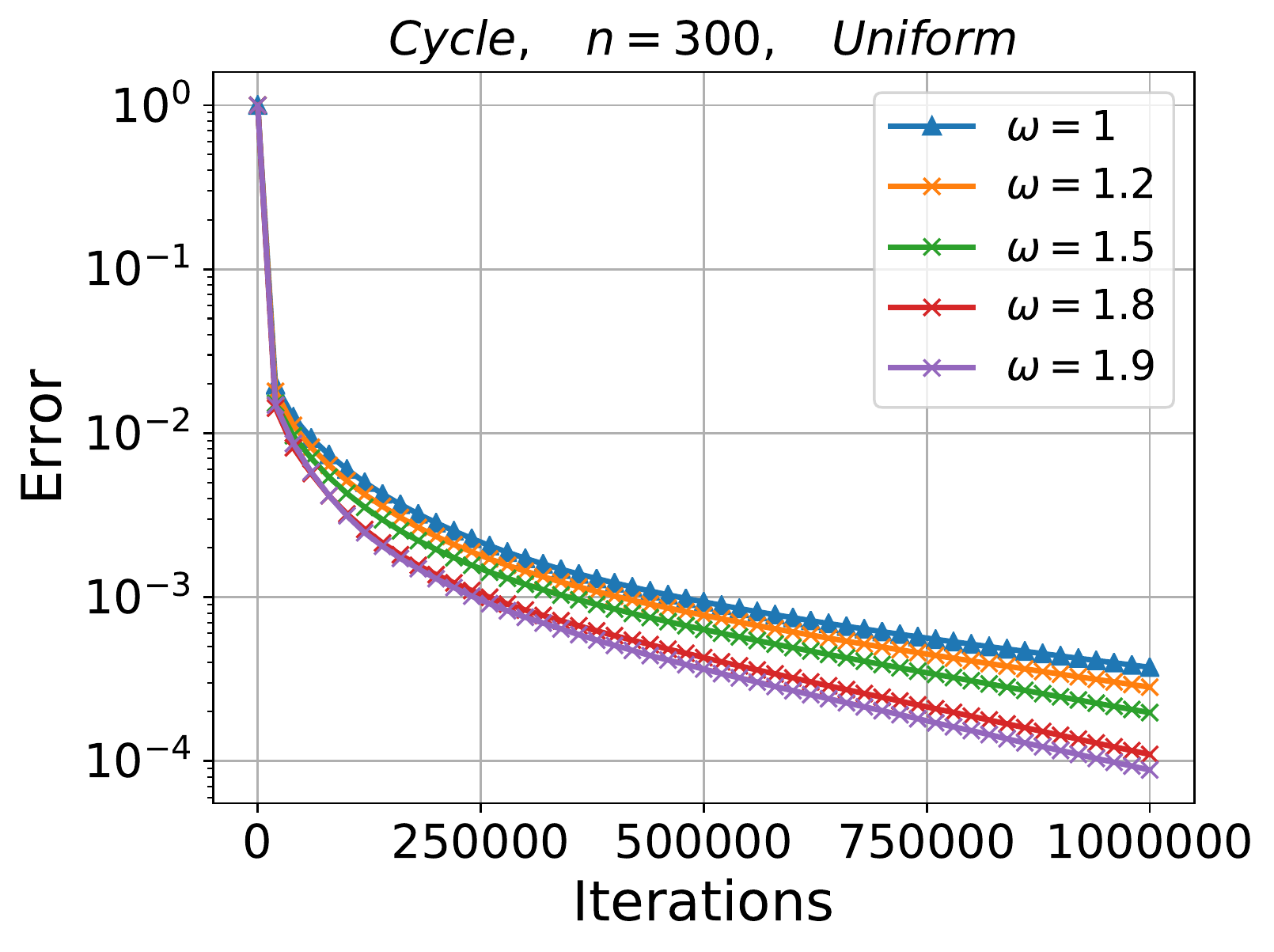}
\end{subfigure}\\
\begin{subfigure}{.3\textwidth}
  \centering
  \includegraphics[width=1\linewidth]{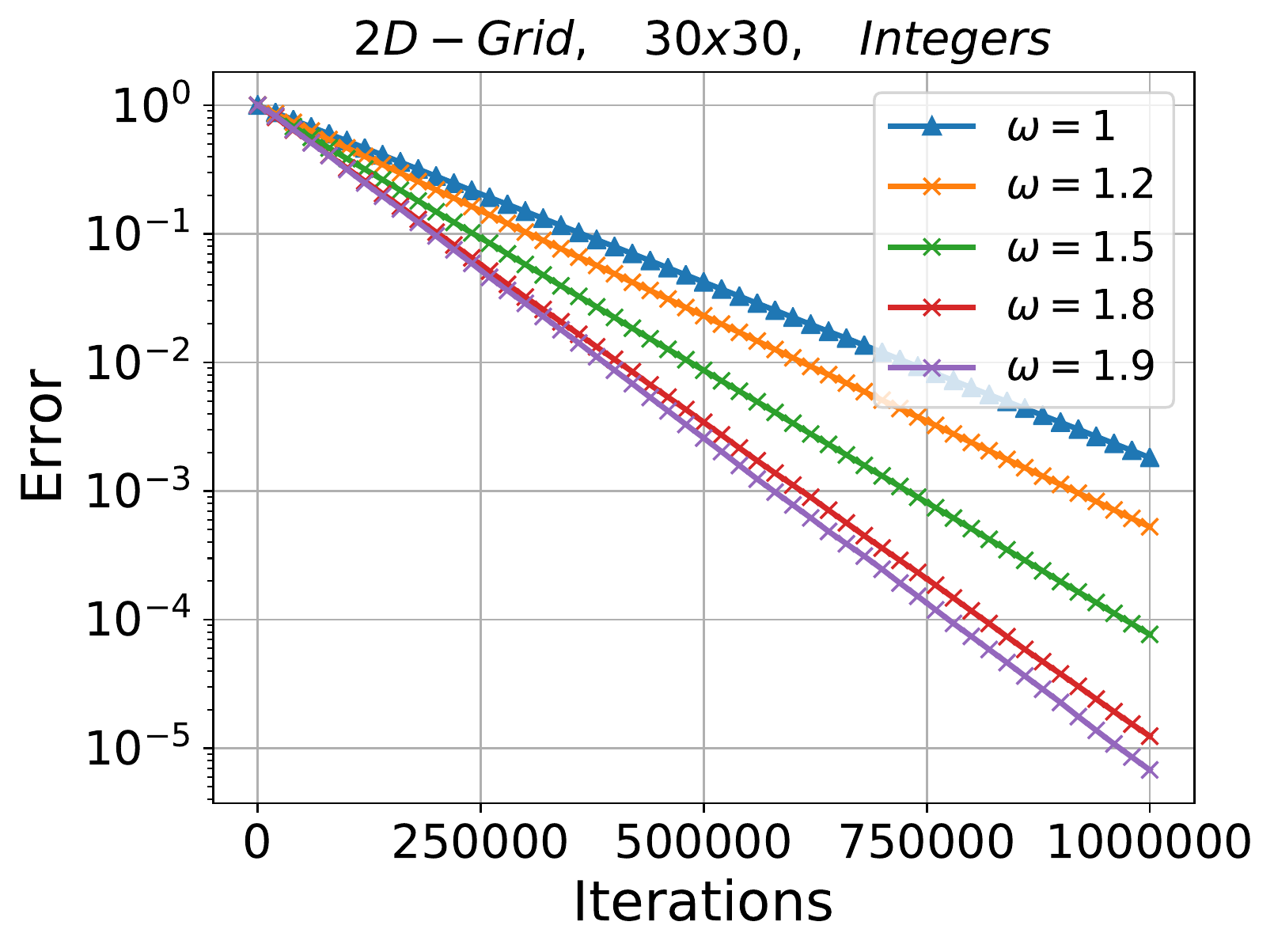}
\end{subfigure}%
\begin{subfigure}{.3\textwidth}
  \centering
  \includegraphics[width=1\linewidth]{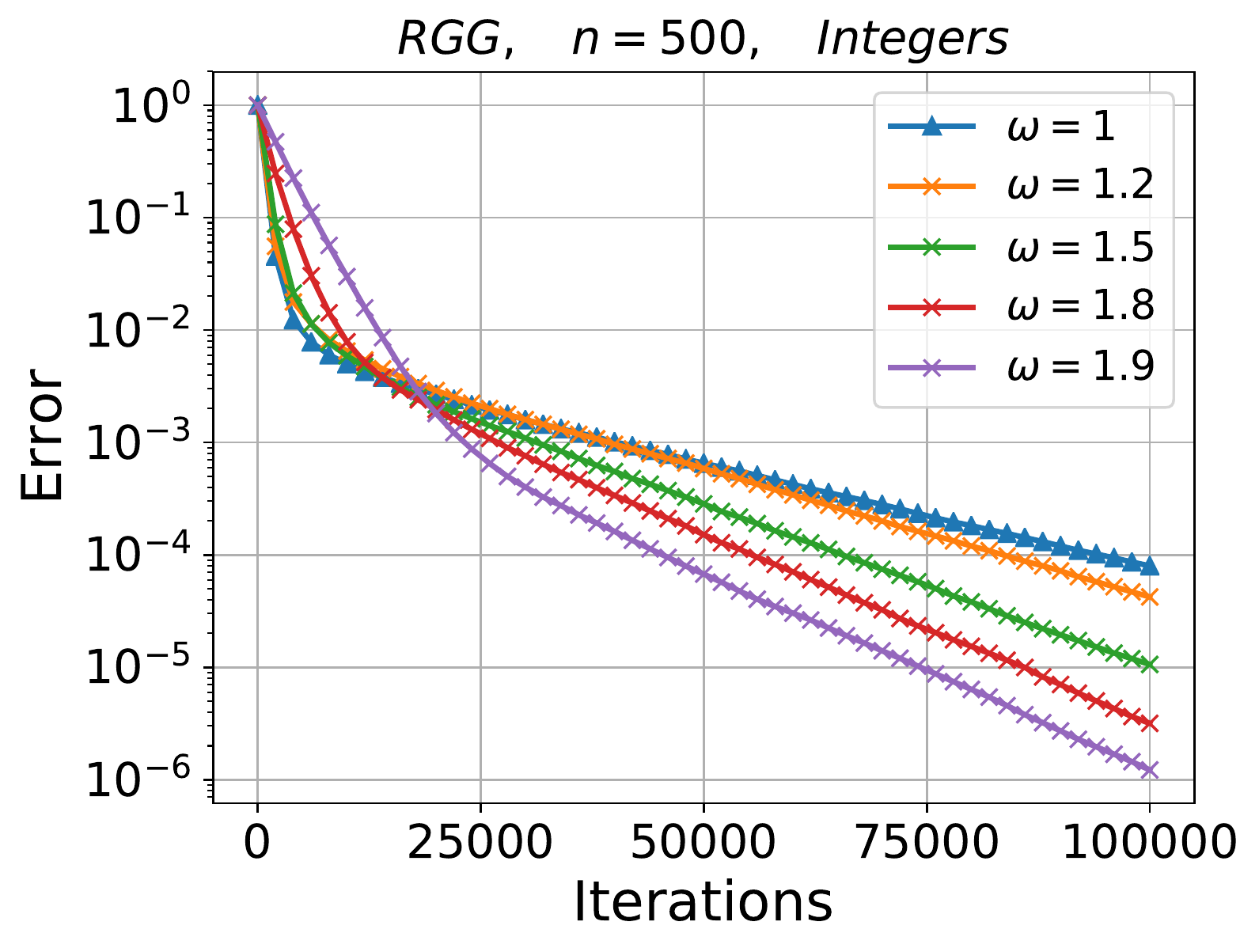}
\end{subfigure}
\begin{subfigure}{.3\textwidth}
  \centering
  \includegraphics[width=1\linewidth]{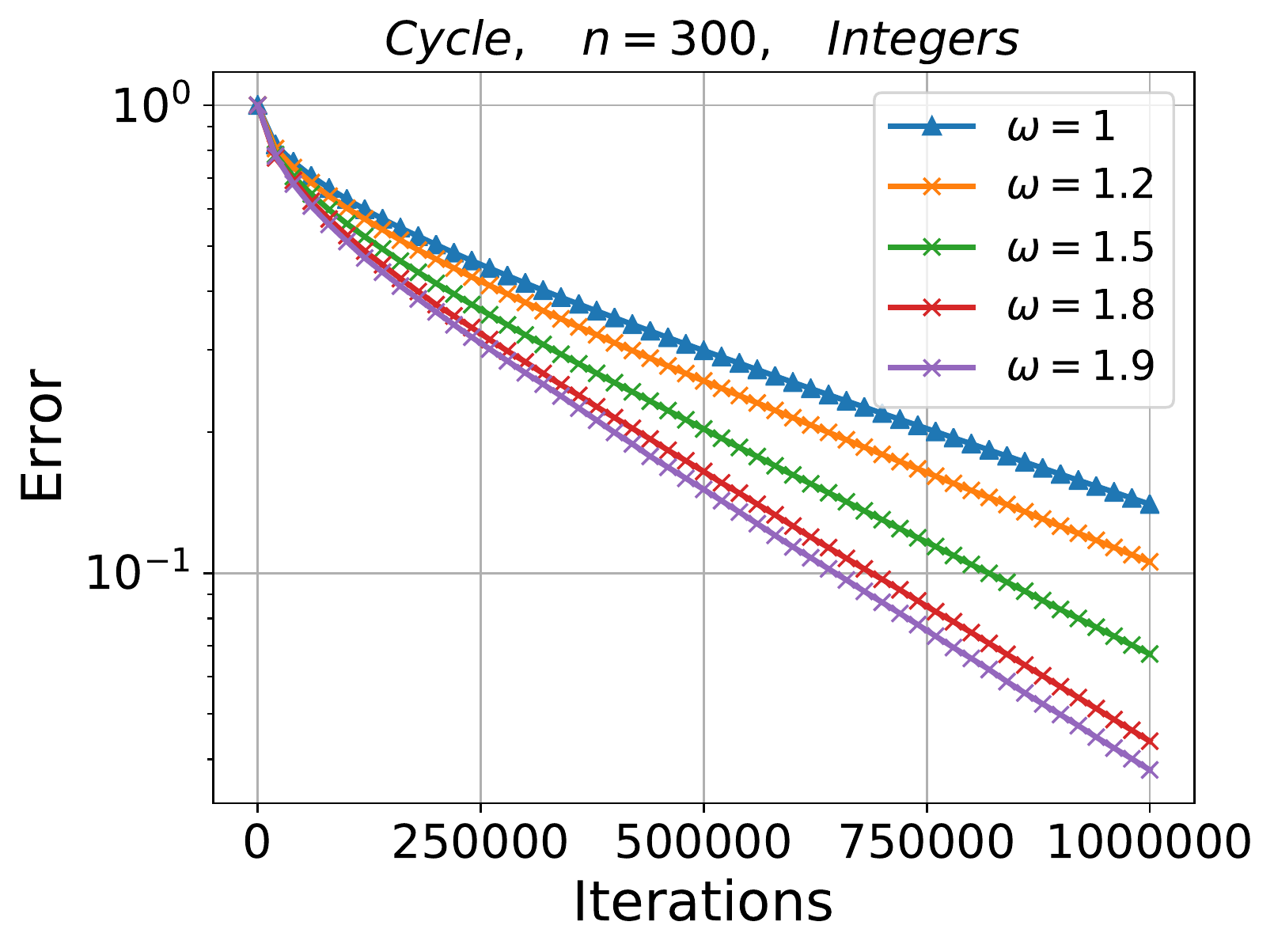}
\end{subfigure}
\caption{\footnotesize Performance of Relaxed randomized pairwise Gossip algorithm in a 2-dimension grid, random geometric graph (RGG) and a cycle graph. The case of $\omega=1$ corresponds to the randomized pairwise gossip algorithm proposed in \cite{boyd2006randomized} ; The $n$ in the title of each plot indicates the number of nodes of the network. For the grid graph this is $n \times n$. The title of each plot indicates the vector of starting values that is used. }
\label{RelaxedGossipFigure}
\end{figure}

\section{Conclusion and Future Direction of Research}
\label{conclusion}
In this work we present a general framework for the analysis and design of randomized gossip algorithms.  Using tools from numerical linear algebra and the area of randomized projection methods for solving linear systems we propose novel serial, block and accelerated gossip protocols for solving the average consensus and weighted average consensus problems. 

We believe that this work could open up several future avenues for research. Using similar approach with the one presented in this manuscript, many popular projection methods can be interpreted as gossip algorithms when used to solve linear systems encoding the underlying network. This can lead to the development of novel distributed protocols for average consensus.  

In addition, we speculate that the gossip protocols presented in this work can be extended to the more general setting of multi-agent consensus optimization where the goal is to minimize the average of convex or non-convex functions $\frac{1}{n} \sum_{i=1}^n f_i(x)$ in a decentralized way \cite{nedic2018network}. 

\section{Acknowledgments}
The authors would like to thank Mike Rabbat for useful discussions related to the literature of gossip algorithms and for his comments during the writing of this paper.

\bibliographystyle{plain}
{\footnotesize\bibliography{RandGossip}}

\begin{thebibliography}{10}

\bibitem{assran2018stochastic}
M.~Assran, N.~Loizou, N.~Ballas, and M.~Rabbat.
\newblock Stochastic gradient push for distributed deep learning.
\newblock {\em arXiv preprint arXiv:1811.10792}, 2018.

\bibitem{assran2018asynchronous}
M.~Assran and M.~Rabbat.
\newblock Asynchronous subgradient-push.
\newblock {\em arXiv preprint arXiv:1803.08950}, 2018.

\bibitem{aybat2017decentralized}
N.~S. Aybat and M.~G{\"u}rb{\"u}zbalaban.
\newblock Decentralized computation of effective resistances and acceleration
  of consensus algorithms.
\newblock In {\em Signal and Information Processing (GlobalSIP), 2017 IEEE
  Global Conference on}, pages 538--542. IEEE, 2017.

\bibitem{aysal2009broadcast}
T.C. Aysal, M.E. Yildiz, A.D. Sarwate, and A.~Scaglione.
\newblock Broadcast gossip algorithms for consensus.
\newblock {\em IEEE Trans. Signal Process.}, 57(7):2748--2761, 2009.

\bibitem{benezit2010order}
F.~B{\'e}n{\'e}zit, A.G. Dimakis, P.~Thiran, and M.~Vetterli.
\newblock Order-optimal consensus through randomized path averaging.
\newblock {\em IEEE Trans. Inf. Theory}, 56(10):5150--5167, 2010.

\bibitem{boyd2006randomized}
S.~Boyd, A.~Ghosh, B.~Prabhakar, and D.~Shah.
\newblock Randomized gossip algorithms.
\newblock {\em IEEE Transactions on Information Theory}, 14(SI):2508--2530,
  2006.

\bibitem{cao2006accelerated}
M.~Cao, D.A. Spielman, and E.M. Yeh.
\newblock Accelerated gossip algorithms for distributed computation.
\newblock In {\em Proc. of the 44th Annual Allerton Conference on
  Communication, Control, and Computation}, pages 952--959, 2006.

\bibitem{charalambous2016distributed}
T.~Charalambous, M.G. Rabbat, M.~Johansson, and C.N. Hadjicostis.
\newblock Distributed finite-time computation of digraph parameters: Left
  eigenvector, out-degree and spectrum.
\newblock {\em IEEE Trans. Control of Network Systems}, 3(2):137--148, June
  2016.

\bibitem{colin2016gossip}
I.~Colin, A.~Bellet, J.~Salmon, and S.~Cl{\'e}men{\c{c}}on.
\newblock Gossip dual averaging for decentralized optimization of pairwise
  functions.
\newblock In {\em Proceedings of the 33rd International Conference on
  International Conference on Machine Learning-Volume 48}, pages 1388--1396.
  JMLR. org, 2016.

\bibitem{cybenko1989dynamic}
G.~Cybenko.
\newblock Dynamic load balancing for distributed memory multiprocessors.
\newblock {\em J. Parallel Distrib. Comput.}, 7(2):279--301, 1989.

\bibitem{dai2014randomized}
L.~Dai, M.~Soltanalian, and K.~Pelckmans.
\newblock On the randomized {Kaczmarz} algorithm.
\newblock {\em IEEE Signal Processing Letters}, 21(3):330--333, 2014.

\bibitem{de2017sampling}
J.~A. De~Loera, J.~Haddock, and D.~Needell.
\newblock A sampling {Kaczmarz}--{Motzkin} algorithm for linear feasibility.
\newblock {\em SIAM Journal on Scientific Computing}, 39(5):S66--S87, 2017.

\bibitem{degroot1974reaching}
Morris~H DeGroot.
\newblock Reaching a consensus.
\newblock {\em Journal of the American Statistical Association},
  69(345):118--121, 1974.

\bibitem{dimakis2010gossip}
A.G. Dimakis, S.~Kar, J.M.F. Moura, M.G. Rabbat, and A.~Scaglione.
\newblock Gossip algorithms for distributed signal processing.
\newblock {\em Proceedings of the IEEE}, 98(11):1847--1864, 2010.

\bibitem{dimakis2008geographic}
A.G. Dimakis, A.D. Sarwate, and M.J. Wainwright.
\newblock Geographic gossip: {Efficient} averaging for sensor networks.
\newblock {\em IEEE Trans. Signal Process.}, 56(3):1205--1216, 2008.

\bibitem{dutight}
K.~Du.
\newblock Tight upper bounds for the convergence of the randomized extended
  {Kaczmarz} and {Gauss}--{Seidel} algorithms.
\newblock {\em Numerical Linear Algebra with Applications}, 26(3):e2233, 2019.

\bibitem{duchi2012dual}
J.C. Duchi, A.~Agarwal, and M.J. Wainwright.
\newblock Dual averaging for distributed optimization: convergence analysis and
  network scaling.
\newblock {\em IEEE Transactions on Automatic Control}, 57(3):592--606, 2012.

\bibitem{eldar2011acceleration}
Y.C. Eldar and D.~Needell.
\newblock Acceleration of randomized {K}aczmarz method via the
  {Johnson--Lindenstrauss} lemma.
\newblock {\em Numerical Algorithms}, 58(2):163--177, 2011.

\bibitem{freris2012fast}
N.M. Freris and A.~Zouzias.
\newblock Fast distributed smoothing of relative measurements.
\newblock In {\em Decision and Control (CDC), 2012 IEEE 51st Annual Conference
  on}, pages 1411--1416. IEEE, 2012.

\bibitem{ghadimi2014admm}
E.~Ghadimi, A.~Teixeira, M.G. Rabbat, and M.~Johansson.
\newblock The admm algorithm for distributed averaging: Convergence rates and
  optimal parameter selection.
\newblock In {\em 2014 48th Asilomar Conference on Signals, Systems and
  Computers}, pages 783--787. IEEE, 2014.

\bibitem{gower2019sgd}
R.~M. Gower, N.~Loizou, X.~Qian, A.~Sailanbayev, E.~Shulgin, and P.~Richtarik.
\newblock {SGD}: General analysis and improved rates.
\newblock {\em Proceedings of the 36th International Conference on Machine
  Learning (ICML)}, 2019.

\bibitem{gower2016randomized}
R.~M. Gower and P.~Richt{\'a}rik.
\newblock Randomized quasi-newton updates are linearly convergent matrix
  inversion algorithms.
\newblock {\em SIAM Journal on Matrix Analysis and Applications},
  38(4):1380--1409, 2017.

\bibitem{gower2016stochastic}
R.M. Gower, D.~Goldfarb, and P.~Richt{\'a}rik.
\newblock Stochastic block {BFGS}: squeezing more curvature out of data.
\newblock In {\em ICML}, pages 1869--1878, 2016.

\bibitem{gower2018accelerated}
R.M. Gower, F.~Hanzely, P.~Richt{\'a}rik, and S.~U. Stich.
\newblock Accelerated stochastic matrix inversion: general theory and speeding
  up {BFGS} rules for faster second-order optimization.
\newblock In {\em Advances in Neural Information Processing Systems}, pages
  1619--1629, 2018.

\bibitem{gower2015randomized}
R.M. Gower and P.~Richt{\'a}rik.
\newblock Randomized iterative methods for linear systems.
\newblock {\em SIAM Journal on Matrix Analysis and Applications},
  36(4):1660--1690, 2015.

\bibitem{gower2015stochastic}
R.M. Gower and P.~Richt{\'a}rik.
\newblock Stochastic dual ascent for solving linear systems.
\newblock {\em arXiv preprint arXiv:1512.06890}, 2015.

\bibitem{gower2016linearly}
R.M. Gower and P.~Richt{\'a}rik.
\newblock Linearly convergent randomized iterative methods for computing the
  pseudoinverse.
\newblock {\em arXiv preprint arXiv:1612.06255}, 2016.

\bibitem{haddock2018motzkin}
J.~Haddock and D.~Needell.
\newblock On motzkin’s method for inconsistent linear systems.
\newblock {\em BIT Numerical Mathematics}, pages 1--15, 2018.

\bibitem{hanzely2017privacy}
F.~Hanzely, J.~Kone{\v{c}}n{\'y}, N.~Loizou, P.~Richt{\'a}rik, and
  D.~Grishchenko.
\newblock Privacy preserving randomized gossip algorithms.
\newblock {\em arXiv preprint arXiv:1706.07636}, 2017.

\bibitem{hanzely2019privacy}
F.~Hanzely, J.~Kone{\v{c}}n{\`y}, N.~Loizou, P.~Richt{\'a}rik, and
  D.~Grishchenko.
\newblock A privacy preserving randomized gossip algorithm via controlled noise
  insertion.
\newblock {\em NeurIPS Privacy Preserving Machine Learning Workshop}, 2018.

\bibitem{he2011periodic}
F.~He, A.~S. Morse, J.~Liu, and S.~Mou.
\newblock Periodic gossiping.
\newblock {\em IFAC Proceedings Volumes}, 44(1):8718--8723, 2011.

\bibitem{hendrikx2018accelerated}
H.~Hendrikx, L.~Massouli{\'e}, and F.~Bach.
\newblock Accelerated decentralized optimization with local updates for smooth
  and strongly convex objectives.
\newblock {\em arXiv preprint arXiv:1810.02660}, 2018.

\bibitem{hernandes2018improved}
A.~G. Hernandes, M.~L. Proen{\c{c}}a~Jr, and T.~Abr{\~a}o.
\newblock Improved weighted average consensus in distributed cooperative
  spectrum sensing networks.
\newblock {\em Transactions on Emerging Telecommunications Technologies},
  29(3):e3259, 2018.

\bibitem{kaczmarz1937angenaherte}
S.~Kaczmarz.
\newblock Angen{\"a}herte aufl{\"o}sung von systemen linearer gleichungen.
\newblock {\em Bulletin International de l’Academie Polonaise des Sciences et
  des Lettres}, 35:355--357, 1937.

\bibitem{kokiopoulou2009polynomial}
E.~Kokiopoulou and P.~Frossard.
\newblock Polynomial filtering for fast convergence in distributed consensus.
\newblock {\em IEEE Transactions on Signal Processing}, 57(1):342--354, 2009.

\bibitem{koloskova2019decentralized}
A.~Koloskova, S.~U. Stich, and M.~Jaggi.
\newblock Decentralized stochastic optimization and gossip algorithms with
  compressed communication.
\newblock {\em arXiv preprint arXiv:1902.00340}, 2019.

\bibitem{krizhevsky2012imagenet}
A.~Krizhevsky, I.~Sutskever, and G.E. Hinton.
\newblock Imagenet classification with deep convolutional neural networks.
\newblock In {\em NIPS}, pages 1097--1105, 2012.

\bibitem{leventhal2010randomized}
D.~Leventhal and A.S. Lewis.
\newblock Randomized methods for linear constraints: convergence rates and
  conditioning.
\newblock {\em Mathematics of Operations Research}, 35(3):641--654, 2010.

\bibitem{lian2018asynchronous}
X.~Lian, C.~Zhang, C.-J. Hsieh, W.~Zhang, and J.~Liu.
\newblock Asynchronous decentralized parallel stochastic gradient descent.
\newblock In {\em International Conference on Machine Learning}, pages
  3049--3058, 2018.

\bibitem{liu2013analysis}
J.~Liu, B.D.O. Anderson, M.~Cao, and A.S. Morse.
\newblock Analysis of accelerated gossip algorithms.
\newblock {\em Automatica}, 49(4):873--883, 2013.

\bibitem{liu2016accelerated}
J.~Liu and S.~Wright.
\newblock An accelerated randomized {Kaczmarz} algorithm.
\newblock {\em Mathematics of Computation}, 85(297):153--178, 2016.

\bibitem{liu2011deterministic}
Ji~Liu, Shaoshuai Mou, A~Stephen Morse, Brian~DO Anderson, and Changbin Yu.
\newblock Deterministic gossiping.
\newblock {\em Proceedings of the IEEE}, 99(9):1505--1524, 2011.

\bibitem{liu2018privacy}
Y.~Liu, J.~Wu, I.~Manchester, and G.~Shi.
\newblock Privacy-preserving gossip algorithms.
\newblock {\em arXiv preprint arXiv:1808.00120}, 2018.

\bibitem{liu2017clique}
Yang Liu, Bo~Li, Brian Anderson, and Guodong Shi.
\newblock Clique gossiping.
\newblock {\em arXiv preprint arXiv:1706.02540}, 2017.

\bibitem{loizou2018provably}
N.~Loizou, M.~Rabbat, and P.~Richt{\'a}rik.
\newblock Provably accelerated randomized gossip algorithms.
\newblock In {\em ICASSP 2019-2019 IEEE International Conference on Acoustics,
  Speech and Signal Processing (ICASSP)}, pages 7505--7509. IEEE, 2019.

\bibitem{LoizouRichtarik}
N.~Loizou and P.~Richt{\'a}rik.
\newblock A new perspective on randomized gossip algorithms.
\newblock In {\em 2016 IEEE Global Conference on Signal and Information
  Processing (GlobalSIP)}, pages 440--444. IEEE, 2016.

\bibitem{loizou2017linearly}
N.~Loizou and P.~Richt{\'a}rik.
\newblock Linearly convergent stochastic heavy ball method for minimizing
  generalization error.
\newblock {\em NIPS-Workshop on Optimization for Machine Learning}, 2017.

\bibitem{loizou2017momentum}
N.~Loizou and P.~Richt{\'a}rik.
\newblock Momentum and stochastic momentum for stochastic gradient, {Newton},
  proximal point and subspace descent methods.
\newblock {\em arXiv preprint arXiv:1712.09677}, 2017.

\bibitem{loizou2018accelerated}
N.~Loizou and P.~Richt{\'a}rik.
\newblock Accelerated gossip via stochastic heavy ball method.
\newblock In {\em 2018 56th Annual Allerton Conference on Communication,
  Control, and Computing (Allerton)}, pages 927--934. IEEE, 2018.

\bibitem{loizou2019Inexact}
N.~Loizou and P.~Richt{\'a}rik.
\newblock Convergence analysis of inexact randomized iterative methods.
\newblock {\em arXiv preprint arXiv:1903.07971}, 2019.

\bibitem{MaConvergence15}
A.~Ma, D.~Needell, and A.~Ramdas.
\newblock Convergence properties of the randomized extended {G}auss-{S}eidel
  and {K}aczmarz methods.
\newblock {\em SIAM Journal on Matrix Analysis and Applications},
  36(4):1590--1604, 2015.

\bibitem{morshed2019accelerated}
M.~S. Morshed, M.S. Islam, et~al.
\newblock Accelerated sampling {Kaczmarz Motzkin} algorithm for linear
  feasibility problem.
\newblock {\em arXiv preprint arXiv:1902.03502}, 2019.

\bibitem{mou2010deterministic}
S.~Mou, C.~Yu, B.D.O Anderson, and A.~S. Morse.
\newblock Deterministic gossiping with a periodic protocol.
\newblock In {\em Decision and Control (CDC), 2010 49th IEEE Conference on},
  pages 5787--5791. IEEE, 2010.

\bibitem{necoara2014rate}
I.~Necoara and V.~Nedelcu.
\newblock Rate analysis of inexact dual first-order methods application to dual
  decomposition.
\newblock {\em IEEE Transactions on Automatic Control}, 59(5):1232--1243, 2014.

\bibitem{nedic2018network}
A.~Nedi{\'c}, A.~Olshevsky, and M.~G. Rabbat.
\newblock Network topology and communication-computation tradeoffs in
  decentralized optimization.
\newblock {\em Proceedings of the IEEE}, 106(5):953--976, 2018.

\bibitem{needell2010randomized}
D.~Needell.
\newblock Randomized {K}aczmarz solver for noisy linear systems.
\newblock {\em BIT Numerical Mathematics}, 50(2):395--403, 2010.

\bibitem{RBK}
D.~Needell and J.A. Tropp.
\newblock Paved with good intentions: analysis of a randomized block {K}aczmarz
  method.
\newblock {\em Linear Algebra Appl.}, 441:199--221, 2014.

\bibitem{l2015randomized}
D.~Needell, R.~Zhao, and A.~Zouzias.
\newblock Randomized block {Kaczmarz} method with projection for solving least
  squares.
\newblock {\em Linear Algebra Appl.}, 484:322--343, 2015.

\bibitem{nesterov1983method}
Y.~Nesterov.
\newblock A method of solving a convex programming problem with convergence
  rate $o(1/k^2)$.
\newblock {\em Soviet Mathematics Doklady}, 27:372--376, 1983.

\bibitem{nesterov2012efficiency}
Y.~Nesterov.
\newblock Efficiency of coordinate descent methods on huge-scale optimization
  problems.
\newblock {\em SIAM Journal on Optimization}, 22(2):341--362, 2012.

\bibitem{nesterov2013introductory}
Y.~Nesterov.
\newblock {\em Introductory Lectures on Convex Optimization: A Basic Course},
  volume~87.
\newblock Springer Science \& Business Media, 2013.

\bibitem{nutini2016convergence}
J.~Nutini, B.~Sepehry, I.~Laradji, M.~Schmidt, H.~Koepke, and A.~Virani.
\newblock Convergence rates for greedy {Kaczmarz} algorithms, and faster
  randomized {Kaczmarz} rules using the orthogonality graph.
\newblock In {\em Proceedings of the Thirty-Second Conference on Uncertainty in
  Artificial Intelligence}, pages 547--556. AUAI Press, 2016.

\bibitem{olshevsky2014linear}
A.~Olshevsky.
\newblock Linear time average consensus on fixed graphs and implications for
  decentralized optimization and multi-agent control.
\newblock {\em arXiv preprint arXiv:1411.4186}, 2014.

\bibitem{olshevsky2009convergence}
A.~Olshevsky and J.N. Tsitsiklis.
\newblock Convergence speed in distributed consensus and averaging.
\newblock {\em SIAM Journal on Control and Optimization}, 48(1):33--55, 2009.

\bibitem{oreshkin2010optimization}
B.~N. Oreshkin, M.~J. Coates, and M.~G. Rabbat.
\newblock Optimization and analysis of distributed averaging with short node
  memory.
\newblock {\em IEEE Transactions on Signal Processing}, 58(5):2850--2865, 2010.

\bibitem{pedroche2014convergence}
F.~Pedroche~S{\'a}nchez, M.~Rebollo~Pedruelo, C.~Carrascosa~Casamayor, and
  A.~Palomares~Chust.
\newblock Convergence of weighted-average consensus for undirected graphs.
\newblock {\em International Journal of Complex Systems in Science},
  4(1):13--16, 2014.

\bibitem{polyak1964some}
B.T. Polyak.
\newblock Some methods of speeding up the convergence of iteration methods.
\newblock {\em USSR Computational Mathematics and Mathematical Physics},
  4(5):1--17, 1964.

\bibitem{popa2017convergence}
C.~Popa.
\newblock Convergence rates for {Kaczmarz}-type algorithms.
\newblock {\em arXiv preprint arXiv:1701.08002}, 2017.

\bibitem{qu2015sdna}
Z.~Qu, P.~Richt{\'a}rik, M.~Tak{\'a}{\v{c}}, and O.~Fercoq.
\newblock {SDNA}: {Stochastic} dual {Newton} ascent for empirical risk
  minimization.
\newblock {\em ICML}, 2016.

\bibitem{qu2015quartz}
Z.~Qu, P.~Richt{\'a}rik, and T.~Zhang.
\newblock Quartz: Randomized dual coordinate ascent with arbitrary sampling.
\newblock In {\em Advances in {Neural} {Information} {Processing} {Systems}},
  pages 865--873, 2015.

\bibitem{rabbat2005generalized}
M.G. Rabbat, R.D. Nowak, and J.A. Bucklew.
\newblock Generalized consensus computation in networked systems with erasure
  links.
\newblock In {\em IEEE 6th Workshop on Signal Processing Advances in Wireless
  Communications}, pages 1088--1092. IEEE, 2005.

\bibitem{richtarik2014iteration}
P.~Richt{\'a}rik and M.~Tak{\'a}{\v{c}}.
\newblock Iteration complexity of randomized block-coordinate descent methods
  for minimizing a composite function.
\newblock {\em Mathematical Programming}, 144(1-2):1--38, 2014.

\bibitem{ASDA}
P.~Richt\'{a}rik and M.~Tak\'{a}\v{c}.
\newblock Stochastic reformulations of linear systems: algorithms and
  convergence theory.
\newblock {\em arXiv:1706.01108}, 2017.

\bibitem{schopfer2016linear}
F.~Sch{\"o}pfer and D.A. Lorenz.
\newblock Linear convergence of the randomized sparse {K}aczmarz method.
\newblock {\em arXiv preprint arXiv:1610.02889}, 2016.

\bibitem{SDCA}
Sh. Shalev-Shwartz and T.~Zhang.
\newblock Stochastic dual coordinate ascent methods for regularized loss.
\newblock {\em Journal of Machine Learning Research}, 14(1):567--599, 2013.

\bibitem{shi2015extra}
W.~Shi, Q.~Ling, G.~Wu, and W.~Yin.
\newblock Extra: An exact first-order algorithm for decentralized consensus
  optimization.
\newblock {\em SIAM Journal on Optimization}, 25(2):944--966, 2015.

\bibitem{RK}
T.~Strohmer and R.~Vershynin.
\newblock A randomized {K}aczmarz algorithm with exponential convergence.
\newblock {\em J. Fourier Anal. Appl.}, 15(2):262--278, 2009.

\bibitem{sutskever2013importance}
I.~Sutskever, J.~Martens, G.E. Dahl, and G.E. Hinton.
\newblock On the importance of initialization and momentum in deep learning.
\newblock {\em ICML (3)}, 28:1139--1147, 2013.

\bibitem{szegedy2015going}
C.~Szegedy, W.~Liu, Y.~Jia, P.~Sermanet, S.~Reed, D.~Anguelov, D.~Erhan,
  V.~Vanhoucke, and A.~Rabinovich.
\newblock Going deeper with convolutions.
\newblock In {\em CVPR}, pages 1--9, 2015.

\bibitem{tsianos2012communication}
K.~Tsianos, S.~Lawlor, and M.~G. Rabbat.
\newblock Communication/computation tradeoffs in consensus-based distributed
  optimization.
\newblock In {\em Conference on Neural Information Processing Systems}, 2012.

\bibitem{tsitsiklis1986distributed}
John Tsitsiklis, Dimitri Bertsekas, and Michael Athans.
\newblock Distributed asynchronous deterministic and stochastic gradient
  optimization algorithms.
\newblock {\em IEEE transactions on automatic control}, 31(9):803--812, 1986.

\bibitem{tu2017breaking}
S.~Tu, S.~Venkataraman, A.C. Wilson, A.~Gittens, M.I. Jordan, and B.~Recht.
\newblock Breaking locality accelerates block {Gauss}-{Seidel}.
\newblock In {\em ICML}, 2017.

\bibitem{ustebay2008greedy}
D.~Ustebay, M.~Coates, and M.~Rabbat.
\newblock Greedy gossip with eavesdropping.
\newblock In {\em Wireless Pervasive Computing, 2008. ISWPC 2008. 3rd
  International Symposium on}, pages 759--763. IEEE, 2008.

\bibitem{wei20131}
E.~Wei and A.~Ozdaglar.
\newblock On the o (1= k) convergence of asynchronous distributed alternating
  direction method of multipliers.
\newblock In {\em 2013 IEEE Global Conference on Signal and Information
  Processing}, pages 551--554. IEEE, 2013.

\bibitem{wilson2017marginal}
A.C. Wilson, R.~Roelofs, M.~Stern, N.~Srebro, and B.~Recht.
\newblock The marginal value of adaptive gradient methods in machine learning.
\newblock {\em arXiv preprint arXiv:1705.08292}, 2017.

\bibitem{wright2015coordinate}
S.J. Wright.
\newblock Coordinate descent algorithms.
\newblock {\em Mathematical Programming}, 151(1):3--34, 2015.

\bibitem{xiao2004fast}
L.~Xiao and S.~Boyd.
\newblock Fast linear iterations for distributed averaging.
\newblock {\em Systems \& Control Letters}, 53(1):65--78, 2004.

\bibitem{xiao2005scheme}
L.~Xiao, S.~Boyd, and S.~Lall.
\newblock A scheme for robust distributed sensor fusion based on average
  consensus.
\newblock In {\em Information Processing in Sensor Networks, 2005. IPSN 2005.
  Fourth International Symposium on}, pages 63--70. IEEE, 2005.

\bibitem{yu2017distributed}
C.~B. Yu, B.D.O Anderson, S.~Mou, J.~Liu, F.~He, and A.~S. Morse.
\newblock Distributed averaging using periodic gossiping.
\newblock {\em IEEE Transactions on Automatic Control}, 2017.

\bibitem{yuan2016convergence}
K.~Yuan, Q.~Ling, and W.~Yin.
\newblock On the convergence of decentralized gradient descent.
\newblock {\em SIAM Journal on Optimization}, 26(3):1835--1854, 2016.

\bibitem{zhang2015distributed}
W.~Zhang, Y.~Guo, H.~Liu, Y.~J. Chen, Z.~Wang, and J.~Mitola~III.
\newblock Distributed consensus-based weight design for cooperative spectrum
  sensing.
\newblock {\em IEEE Transactions on Parallel and Distributed Systems},
  26(1):54--64, 2015.

\bibitem{zhang2011distributed}
W.~Zhang, Z.~Wang, Y.~Guo, H.~Liu, Y.~Chen, and J.~Mitola~III.
\newblock Distributed cooperative spectrum sensing based on weighted average
  consensus.
\newblock In {\em 2011 IEEE Global Telecommunications Conference-GLOBECOM
  2011}, pages 1--6. IEEE, 2011.

\bibitem{zouzias2013randomized}
A.~Zouzias and N.M. Freris.
\newblock Randomized extended {K}aczmarz for solving least squares.
\newblock {\em SIAM Journal on Matrix Analysis and Applications},
  34(2):773--793, 2013.

\bibitem{zouzias2015randomized}
A.~Zouzias and N.M. Freris.
\newblock Randomized gossip algorithms for solving {Laplacian} systems.
\newblock In {\em Control Conference (ECC), 2015 European}, pages 1920--1925.
  IEEE, 2015.

\end{thebibliography}
\appendix
\section{Missing Proofs} 
 \subsection{Proof of Theorem \ref{thm:complexity_standard}}
 \label{ProofTave}
 \begin{proof}
Let $z^k \eqdef \|x^k-x^*\|$ , $x^0=c$ is the starting point and $\rho$ is as defined in \eqref{RateRho}.
From Theorem~\ref{ConvergenceSketchProject} we know that sketch and project method converges with\begin{equation}
 \label{ajcn}
 \Exp[\|x^k-x^*\|_{\bB}^2]\leq \rho^k \|x^0-x^*\|_{\bB}^2,
 \end{equation}
where $x^*$ is the solution of \eqref{best approximation}. Inequality \eqref{ajcn}, together with Markov inequality can be used to give the following bound
\begin{equation} 
\Prob(z^k / z^0 \geq \varepsilon^2) \leq \frac{\Exp(z^k/z^0)}{\varepsilon^2} \leq \frac{\rho^k}{\varepsilon^2}.
\end{equation}
Therefore, as long as $k$ is large enough so that $\rho^k \leq \varepsilon^3$, we have
$\Prob \left(z^k / z^0 \geq \varepsilon^2 \right) \leq \varepsilon$. 
That is, if 
$$\rho^k \leq \varepsilon^3 \Leftrightarrow k \geq \frac{3\log\varepsilon}{\log\rho}\Leftrightarrow k \geq \frac{3\log(1/\varepsilon)}{\log(1/ \rho)}, $$
then:$$\Prob \left(\frac{\|x^{k}-\bar{c}1\|}{\|x^{0}-\bar{c}1\|}\geq \varepsilon \right)\leq\varepsilon.$$

Hence, an upper bound for value $T_{ave}(\epsilon)$ can be obtained as follows,
\begin{eqnarray}
T_{ave}(\epsilon)&=&\sup_{c\in \R^n} \inf  \left\{k\;:\; \Prob \left(z^k > \varepsilon z^0 \right)\leq\varepsilon \right\} \leq \sup_{c\in \R^n} \inf  \left\{k\;:\; k \geq \frac{3\log(1/\varepsilon)}{\log(1/ \rho)}  \right\} \notag\\
 & =&\sup_{c\in \R^n} \frac{3 \log(1/\varepsilon)}{\log(1/ \rho)}= \frac{3 \log(1/\varepsilon)}{\log(1/ \rho)}
  \leq  \frac{3 \log(1/\varepsilon)}{1-\rho},
\end{eqnarray}
where in last inequality we use $1/\log(1/ \rho) \leq 1/1-\rho$ which is true because $\rho \in (0,1)$.  
\end{proof}

\subsection{Proof of Theorem \ref{TheoremRBK}}
\label{ProofRBK}

\begin{proof}
The following notation conventions are used in this proof.
With $q_k$ we indicate the number of connected components of subgraph $\cG_k$, while with $\cV_r$ we denote the set of nodes of each connected component $q_k$ $(r \in \{1,2,\dots,q_k\})$. Finally, $|\cV_r|$ shows the cardinality of set $\cV_r$. Notice that, if $\cV$ is the set of all nodes of the graph then $\cV= \underset{r=\{1,2,...q\}}{\cup} \cV_r$ and $|\cV|=\sum\limits_{{r}=1}^{q} |\cV_r|$. 

Note that from equation \eqref{RBKaLgorithm}, the update of RBK for $\bA=\bQ$(Incidence matrix) can be expressed as follows:
\begin{equation}
\label{constrained}
\begin{aligned}
& \underset{x}{\text{minimize}}
& & \phi^k(x)\eqdef \|x-x^k\|^2\\
& \text{subject to}
& & \bI_{:C}^\top \bQ x=0
\end{aligned}
\end{equation}
Notice that $\bI_{:C}^\top \bQ$ is a row submatrix of matrix $\bQ$ with rows those that correspond to the random set $C \subseteq \cE$ of the edges. From the expression of matrix $\bQ$ we have that
$$(\bI_{:C}^\top \bQ)_{e:}^\top=f_i-f_j , \quad \forall e=(i,j) \in C \subseteq \cE.$$

Now, using this, it can be seen that the constraint $\bI_{:C}^\top \bQ x=0$ of problem \eqref{best approximation} is equivalent to $q$ equations (number of connected components) where each one of them forces the values $x_i^{k+1}$ of the nodes $i \in \cV_r$ to be equal. That is, if we use $z_r$ to represent the value of all nodes that belong in the connected component $r$ then:
\begin{equation}
\label{zr}
 x_i^{k+1}= z_r \quad \forall i \in \cV_r,
\end{equation}
and the constrained optimization problem \eqref{best approximation} can expressed as unconstrained as follows:
\begin{equation}
\label{fin}
 \underset{z}{\text{minimize}}  \quad \phi^k(z)=\sum\limits_{i \in \cV_1}^{}(z_1-x_i^{k})^2+...+\sum\limits_{i \in \cV_q}^{}(z_q-x_i^{k})^2,
\end{equation}
where $z=(z_1,z_2,\dots,z_q) \in \R^q$ is the vector of all values $z_r$ when $r \in \{1,2,\dots,q\}$.
Since our problem is unconstrained the minimum of equation \eqref{fin} is obtained when $\nabla \phi^k(z)=0$.

By evaluating partial derivatives of \eqref{fin} we obtain:
$$\frac{\partial \phi^k(z)}{\partial z_r}=0 \Longleftrightarrow \sum\limits_{i \in \cV_r}^{} 2(z_r- x_i^k)=0.$$
As a result, 
$$z_r=\frac{\sum\limits_{i \in \cV_r}^{} x_i^k}{|\cV_r|}, \quad \forall r \in \{1,2,\dots,q\}. $$
Thus from \eqref{zr}, the value of each node $i \in \cV_r$ will be updated to 
$$x_i^{k+1}= z_r=\frac{\sum\limits_{i \in \cV_r}^{} x_i^k}{|\cV_r|}.$$ .
\end{proof}
\newpage
\section{Notation Glossary}
\label{NotationTable}

\begin{table}[!h]
\begin{center}
\begin{tabular}{|c|l|c|}
 \hline
 \multicolumn{2}{|c|}{{\bf The Basics}}\\
 \hline
$\mA, b$    & $m \times n$ matrix and $m\times 1$ vector defining the system $\mA x =b$\\
$\cL$ &  $\{x\;:\; \mA x = b\}$ (solution set of the linear system) \\
$\mB$    & $n \times n$ symmetric positive definite matrix \\
$\langle x, y \rangle_{\mB}$ & $x^\top \mB y$ ($\mB$-inner product) \\
$\|x\|_{\mB}$ & $\sqrt{\langle x, x \rangle_{\mB}}$ ($\mB$-norm)  \\
$\mM^{\dagger}$ & Moore-Penrose pseudoinverse of matrix $\mM$  \\
$\mS$    & a random real matrix with $m$ rows  \\
$\cD$    & distribution from which matrix $\mS$ is drawn ($\mS\sim \cD$) \\ 
$\mH$ & $\mS (\mS^\top \mA \mB^{-1} \mA^\top \mS)^{\dagger} \mS^\top$ \\
$\mZ$ & $\mA^\top \mH \mA$\\
$\range{\mM}$ & range space of matrix $\mM$  \\
${\rm Null}({\mM})$ & null space of matrix $\mM$  \\
$\Prob(\cdot)$ & probability of an event\\
$\Exp[\cdot]$ & expectation\\
 \hline
 \multicolumn{2}{|c|}{{\bf Projections}}\\
  \hline
$\Pi_{\cL,\mB}(x)$ & projection of $x$ onto $\cL$ in the $\mB$-norm\\
$\mB^{-1}\mZ$ & projection matrix, in the $\mB$-norm, onto $\range{\mB^{-1}\mA^\top \mS}$ \\
 \hline
 \multicolumn{2}{|c|}{{{\bf Graphs}} }\\
 \hline
$\cG = (\cV,\cE)$ & an undirected graph with vertices $\cV$ and edges $\cE$ \\
$n$ & $=|\cV|$ (number of vertices)\\
$m$ & $=|\cE|$ (number of edges) \\
$e = (i,j) \in \cE$ & edge of $\cG$ connecting nodes $i,j\in \cV$  \\
$d_i$ & degree of node $i$  \\
$c \in \R^n$ & $=(c_1,\dots,c_n)$; a vector of private values stored at the nodes of $\cG$  \\
$\bar{c}$ &$\bar{c}=\frac{\sum_i^n \bB_{ii} c_i}{\sum_i^n \bB_{ii}}$ (the weighted average of the private values)  \\
$\bQ \in \R^{m\times m}$ & Incidence matrix of $\cG$  \\
$\bL \in \R^{n\times n}$ & $=\bQ^\top \bQ$ (Laplacian matrix of $\cG$)  \\
$\bD \in \R^{n\times n}$ & $=\text{\textbf{Diag}}(d_1,d_2,\dots, d_n)$ (Degree matrix of $\cG$)  \\
$\bL^{rw} \in \R^{n\times n}$ & $=\bD^{-1} \bL$ (random walk normalized Laplacian matrix of $\cG$)  \\
$\bL^{sym} \in \R^{n\times n}$ & $=\bD^{-1/2} \bL \bD^{-1/2}$ (symmetric normalized Laplacian matrix of $\cG$)  \\
$\ac(\cG)$ & $=\lambda_{\min}^+(\bL)$ (algebraic connectivity of $\cG$)  \\
 \hline
 \multicolumn{2}{|c|}{{\bf Eigenvalues} }\\
 \hline
 $\mW$ & $\mB^{-1/2}\Exp[\mZ]\mB^{-1/2}$ (psd matrix)\\
$\lambda_1,\dots,\lambda_n$ & eigenvalues of $\mW$\\
$\lambda_{\max}, \lambda_{\min}^+$ & largest and smallest nonzero eigenvalues of $\mW$\\
 \hline
 \multicolumn{2}{|c|}{{\bf Algorithms} }\\
 \hline
$\omega$ & relaxation parameter / stepsize  \\
$\beta$ & heavy ball momentum parameter \\
$\rho$ & $1-\omega(2-\omega) \lambda_{\min}^+$\\
\hline
\end{tabular}
\end{center}
\caption{Frequently used notation.}
\label{tbl:notation}
\end{table}

\end{document}